\documentclass{amsart}

\usepackage[english]{babel}

\usepackage[letterpaper,top=2cm,bottom=2cm,left=3cm,right=3cm,marginparwidth=1.75cm]{geometry}

\usepackage{amssymb}
\usepackage{amsmath}
\usepackage{amsthm}
\usepackage{lipsum}
\usepackage{color}
\usepackage{enumerate}
\usepackage{caption}
\usepackage[labelformat=simple]{subcaption}

\usepackage{tikz}
\usetikzlibrary{graphs, graphs.standard,fit,positioning}

\newcommand{\np}{\mathcal{NP}}

\newtheorem{theorem}{Theorem}
\newtheorem{lemma}[theorem]{Lemma}
\newtheorem{corollary}[theorem]{Corollary}

\newtheorem{proposition}[theorem]{Proposition}

\title{Independent [$k$]-Roman Domination on Graphs}

\address{Federal University of Ceará, Campus Quixadá, Quixadá, Ceará, Brazil}
\author{At\'ilio G.~Luiz}
\email{gomes.atilio@ufc.br} 

\address{Federal University of Ceará, Campus Quixadá, Quixadá, Ceará, Brazil}
\author{Francisco Anderson Silva Vieira}
\email{andersonsilva@alu.ufc.br}

\begin{document}

\begin{abstract}
Given a function $f\colon V(G) \to \mathbb{Z}_{\geq 0}$ on a graph $G$, $AN(v)$ denotes the set of  neighbors of $v \in V(G)$ that have positive labels under $f$. In 2021, Ahangar et al.~introduced the notion of $[k]$-Roman Dominating Function ([$k$]-RDF) of a graph $G$, which is a function $f\colon V(G) \to \{0,1,\ldots,k+1\}$ such that $\sum_{u \in N[v]}f(u) \geq k + |AN(v)|$ for all $v \in V(G)$ with $f(v)<k$. The weight of $f$ is $\sum_{v \in V(G)}f(v)$. The $[k]$-Roman domination number, denoted by $\gamma_{[kR]}(G)$, is the minimum weight of a $[k]$-RDF of $G$. The notion of [$k$]-RDF for $k=1$ has been extensively investigated in the scientific literature since 2004, when introduced by Cockayne et al. as Roman Domination. An independent [$k$]-Roman dominating function ([$k$]-IRDF) $f\colon V(G) \to \{0,1,\ldots,k+1\}$ of a graph $G$ is a [$k$]-RDF of $G$ such that the set of vertices with positive labels is an independent set. The independent [$k$]-Roman domination number of $G$ is the minimum weight of a [$k$]-IRDF of $G$ and is denoted by $i_{[kR]}(G)$. In this paper, we propose the study of  independent [$k$]-Roman domination on graphs for arbitrary $k \geq 1$. We prove that, for all $k\geq 3$, the decision problems associated with $i_{[kR]}(G)$ and $\gamma_{[kR]}(G)$ are $\np$-complete for planar bipartite graphs with maximum degree 3. We also present lower and upper bounds for $i_{[kR]}(G)$. Moreover, we present lower and upper bounds for the parameter $i_{[kR]}(G)$ for two families of 3-regular graphs called generalized Blanu\v{s}a snarks and  Loupekine snarks.
\end{abstract}

\maketitle

% ----------------------------
% SECTION
% ----------------------------
\section{Introduction}
\label{sec:introduction}

Let $G=(V(G),E(G))$ be a graph with vertex set $V(G)$ and edge set $E(G)$. Two vertices $u,v \in V(G)$ are \textit{adjacent} or \textit{neighbors} if $uv \in E(G)$. We say that $G$ is trivial if $|V(G)|=1$. For every vertex $v \in V(G)$, the \textit{open neighborhood} of $v$ is the set $N(v) = \{u \in V(G) : uv \in E(G)\}$, and the \textit{closed neighborhood} of $v$ is the set $N[v] = \{v\}\cup N(v)$. The \emph{degree} of a vertex $v\in V(G)$ is the number of neighbors of $v$ and is denoted $d_G(v)$. A \emph{leaf vertex} of $G$ is a vertex $v \in V(G)$ with  $d_G(v)=1$. Graph $G$ is $r$-regular if $d_G(v) = r$ for all $v \in V(G)$. The maximum degree of $G$ is denoted by $\Delta(G)$. For any subset $S \subseteq V(G)$, the \textit{induced subgraph} $G[S]$ is the graph whose vertex set is $S$ and whose edge set consists of all edges in $E(G)$ that have both endpoints in $S$. As usual, $P_n$ denotes a path on $n \geq 1$ vertices and $C_n$ denotes a cycle on $n \geq 3$ vertices.

A set $S \subseteq V(G)$ is called a \emph{dominating set} of $G$ if every vertex $v \in V(G)\backslash S$ is adjacent to a vertex in $S$. The \emph{domination number} of $G$, denoted $\gamma(G)$, is the minimum cardinality of a dominating set of $G$. A set of vertices $S \subseteq V(G)$ is called \emph{independent} if no two vertices in $S$ are adjacent. An \textit{independent dominating set} of $G$ is an independent set $S \subseteq V(G)$ that is also dominating. The \emph{independent domination number} of $G$, denoted $i(G)$, is the minimum cardinality of an independent dominating set of $G$. The domination on graphs has been extensively studied in the scientific literature, giving rise to many variations~\cite{haynes2023domination}, a  well-known of them being the Roman domination~\cite{Cockayne2004Roman}.

Let $G$ be a graph. For any function $f\colon V(G) \to \mathbb{Z}_{\geq 0}$ and $S \subseteq V(G)$, define $f(S) = \sum_{v \in S}f(v)$. A \emph{Roman dominating function} (RDF) on $G$ is a function $f \colon V(G) \to \{0,1,2\}$ such that every vertex $u \in V(G)$ with \emph{label} $f(u) = 0$ is adjacent to at least one vertex $v \in V(G)$ with label $f(v) = 2$. The \emph{weight} of an RDF $f$ is defined as $\omega(f) = f(V(G)) = \sum_{v \in V(G)}f(v)$. The \emph{Roman domination number} of $G$ is the minimum weight over all RDFs on $G$ and is denoted by $\gamma_R(G)$.

The conception of Roman domination on graphs was motivated by defense strategies devised by the Roman Empire during the reign of Emperor Constantine, 272-337 AD~\cite{Stewart1999,ReVelle2000}. The idea behind Roman domination is that labels 1 or 2 represent either one or two Roman legions stationed at a given Roman province (vertex $v$). A neighboring province (an adjacent vertex $u$) is considered to be \emph{unsecured} if no legions are stationed there (i.e.~$f(u)=0$). An unsecured province $u$ can be secured by sending a legion to $u$ from an adjacent province $v$, by respecting the condition that a legion cannot be sent from a province $v$ if doing so leaves that province without a legion. Thus, two legions must be stationed at a province ($f(v) = 2$) before one of the legions can be sent to an adjacent province. 

Results on Roman domination and its variants have been collected in~\cite{Chellali2020,Chellali2020b,Chellali2020c,Chellali2021,Chellali2022}, summing up to more than two hundred papers. Many of these variants aim to increase the effectiveness of the defensive strategy modeled by Roman domination. In 2021, Ahangar et al.~\cite{ABDOLLAHZADEHAHANGAR2021125444} introduced the notion of [$k$]-Roman domination, a generalization of Roman domination, which groups many of these Roman domination's variants under the same definition. The idea behind [$k$]-Roman domination is that any unsecured province could be defended by at least $k$ legions without leaving any secure neighboring province without military forces.

Let $G$ be a graph and $f\colon V(G) \to \mathbb{Z}_{\geq 0}$ be a function. We say that a vertex $v$ of $G$ is \emph{active} under $f$ if $f(v) \geq 1$. For any vertex $v \in V(G)$, the \emph{active neighborhood} of $v$ under $f$, denoted by $AN(v)$, is the set of vertices $w \in N(v)$ such that $f(w) \geq 1$. For any integer $k\geq 1$, a \emph{[$k$]-Roman dominating function} on $G$, also called [$k$]-RDF, is a function $f\colon V(G) \to \{0,1,\ldots,k+1\}$ such that $f(N[v]) \geq k+|AN(v)|$ for every vertex $v\in V(G)$ with $f(v) < k$. The \emph{weight} of a [$k$]-RDF $f$ on $G$ is defined as  $\omega(f) = f(V(G)) = \sum_{v \in V(G)}f(v)$. The \emph{[$k$]-Roman domination number} of $G$ is the minimum weight that a [$k$]-RDF of $G$ can have, and is denoted by $\gamma_{[kR]}(G)$. A [$k$]-RDF of $G$ with weight $\gamma_{[kR]}(G)$ is called a \emph{$\gamma_{[kR]}$-function} of $G$ or \emph{$\gamma_{[kR]}(G)$-function}.
Given a [$k$]-RDF $f \colon V(G) \to \{0,1,\ldots,k+1\}$ of a graph $G$, define $V_i = \{u \in V(G) : f(u)=i\}$ for $0 \leq i \leq k+1$. We call $(V_0,V_1,\ldots,V_{k+1})$ the \emph{ordered partition} of $V(G)$ under $f$. Since there exists a 1-1 correspondence between the functions $f \colon V(G) \to \{0,1,\ldots,k+1\}$ and the ordered partitions $(V_0,V_1,\ldots,V_{k+1})$ of $V(G)$, it is common to use the notation $f = (V_0,V_1,\ldots,V_{k+1})$ to refer to a [$k$]-RDF of $G$. By the definition of ordered partition, we can alternatively define the weight of a [$k$]-RDF $f$ as $\omega(f) = \sum_{p = 0}^{k+1}p|V_p|$. Figure~\ref{fig:DRDFgraphs} shows some graphs endowed with [$k$]-RDFs.

\begin{figure}[!htb]
     \centering
    \begin{subfigure}[b]{0.4\textwidth}
    \centering
    \begin{tikzpicture}[main_node/.style={circle,fill=white,draw,minimum size=.2em,inner sep=1.5pt},cp/.style={fill=white,draw,rounded corners=.09cm,minimum size=1.4em,inner sep=1.4pt},scale=0.9]
    \node[main_node,label=below:{}] (1) at (0,0) {1};
    \node[cp,label=below:{}] (2) at (1.5,0) {$k$};
    \draw (1) -- (2);
    \end{tikzpicture}
    \captionsetup{width=1.5\linewidth}
    \caption{Path with a [$k$]-RDF with weight $k+1$.}
         \label{fig:exampleDRDF0}
     \end{subfigure}
     \hspace{.1cm}
     \begin{subfigure}[b]{0.4\textwidth}
    \centering
    \begin{tikzpicture}[main_node/.style={circle,fill=white,draw,minimum size=.2em,inner sep=1.5pt},cp/.style={fill=white,draw,rounded corners=.09cm,minimum size=1.4em,inner sep=1.4pt},scale=0.9]
    \node[main_node,label=below:{}] (1) at (0,1) {0};
    \node[main_node,label=below:{}] (2) at (0,0) {0};
    \node[main_node,label=below:{}] (3) at (0,-1) {0};
    \node[cp,label=below:{}] (4) at (1.5,0) {$k+1$};
    \node[cp,label=below:{}] (5) at (3,0) {$k+1$};
    \node[main_node,label=below:{}] (6) at (4.5,1) {0};
    \node[main_node,label=below:{}] (7) at (4.5,0) {0};
    \node[main_node,label=above:{}] (8) at (4.5,-1) {0};
    \draw (1) -- (4) -- (2) (3) -- (4);
    \draw (6) -- (5) -- (7)  (4) -- (5) -- (8);
    \end{tikzpicture}
         \caption{Tree with a [$k$]-RDF with weight $2k+2$.}
         \label{fig:exampleDRDF34}
     \end{subfigure}
\caption{Two graphs endowed with [$k$]-Roman dominating functions.}
\label{fig:DRDFgraphs}
\end{figure}
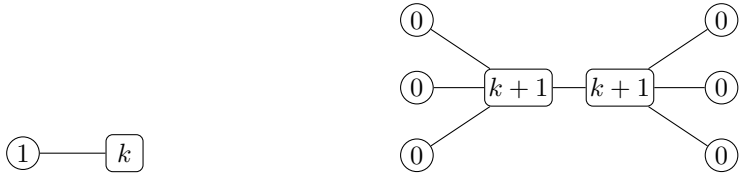

For every $k\geq 1$, the [$k$]-Roman Domination Problem is to determine $\gamma_{[kR]}(G)$ for an arbitrary graph $G$. Khalili et al.~\cite{Khalili2023} proved that the decision version of the [$k$]-Roman Domination Problem is $\np$-Complete even when restricted to  bipartite and chordal graphs. Moreover, Valenzuela-Tripodoro et al.~\cite{Valenzuela-Tripodoro2024} proved that the decision version of [$k$]-RDP is $\np$-Complete even when restricted to star convex and comb convex bipartite graphs.

Note that [1]-Roman domination is equivalent to the original Roman domination definition. In addition, [2]-Roman domination has been previously studied~\cite{BEELER201623} under the name of Double Roman domination, as well as [3]-Roman domination has been investigated~\cite{ABDOLLAHZADEHAHANGAR2021125444} under the name of Triple Roman domination, and [4]-Roman domination has been recently studied under the name of Quadruple Roman domination~\cite{Amjadi2022}. Recently, Khalili et al.~\cite{Khalili2023} and Valenzuela-Tripodoro et al.~\cite{Valenzuela-Tripodoro2024} presented sharp upper and lower bounds for the [$k$]-Roman domination number for all $k\geq 1$.

Given a [$k$]-Roman dominating function $f = (V_0,V_1,\ldots,V_{k+1})$ on a graph $G$, we observe that the set of vertices $S = V_1 \cup V_2 \cup \cdots \cup V_{k+1}$ is a dominating set of $G$ since $V(G)\backslash S = V_0$ and every vertex in $V_0$ is adjacent to a vertex in $S$. This connection between dominating sets and the set of active vertices of a graph $G$ under a [$k$]-Roman dominating function makes it possible to relate the parameters $\gamma(G)$ and $\gamma_{[kR]}(G)$ as well as to extend some restrictions traditionally imposed on dominating sets to [$k$]-Roman dominating functions. An example is the concept of independent dominating set: one may require the dominating set of active vertices of $G$ to be also  independent. Indeed, in their seminal paper, Cockayne et al.~\cite{Cockayne2004Roman} introduced the notion of Roman dominating functions $f=(V_0,V_1,V_2)$ whose set of active vertices $V_1\cup V_2$ is an independent set, which are called \emph{independent Roman dominating functions}. In 2019, Maimani et al.~\cite{Maimani2020a} introduced the notion of \emph{independent double Roman dominating function}, which is a [2]-Roman dominating function $f = (V_0, V_1, V_2, V_3)$ of a graph $G$ such that the set of active vertices $V_1 \cup V_2 \cup V_3$ is an independent set. When studying independent Roman domination and independent double Roman domination, one can observe some differences but many similarities. Thus, based on the previous observations, we propose a generalization of independent Roman domination and independent double Roman domination, defined as follows.

A [$k$]-Roman dominating function $f = (V_0,V_1,\ldots,V_{k+1})$ on a graph $G$ is called an \emph{independent [$k$]-Roman dominating function}, or [$k$]-IRDF for short, if the set of active vertices $V_1 \cup V_2 \cup \cdots \cup V_{k+1}$ is an independent set. The \emph{independent [$k$]-Roman domination number} $i_{[kR]}(G)$ is the minimum weight of a [$k$]-IRDF on $G$, and a [$k$]-IRDF of $G$ with weight $i_{[kR]}(G)$ is called an \emph{$i_{[kR]}$-function} of $G$ or \emph{$i_{[kR]}(G)$-function}. The \emph{Independent [$k$]-Roman Domination Problem} consists in determining $i_{[kR]}(G)$ for an arbitrary graph $G$. Since every [$k$]-IRDF is a [$k$]-RDF, we trivially obtain that $\gamma_{[kR]}(G) \leq i_{[kR]}(G)$ for every graph $G$. As an example, Figure~\ref{fig:IRDFgraphs} shows some graphs with $i_{[kR]}$-functions.

\begin{figure}[!htb]
     \centering
     \begin{subfigure}[b]{0.2\textwidth}
    \centering
    \begin{tikzpicture}[main_node/.style={circle,fill=white,draw,minimum size=.2em,inner sep=1.5pt},cp/.style={fill=white,draw,rounded corners=.09cm,minimum size=1.4em,inner sep=1.4pt},scale=0.9]
    \node[main_node,label=below:{}] (1) at (0,0) {0};
    \node[cp,label=below:{}] (2) at (1.5,0) {$k+1$};
    \draw (1) -- (2);
    \end{tikzpicture}
         \caption{Path with a [$k$]-IRDF.}
         \label{fig:exampleDRDF02}
     \end{subfigure}
     \hspace{.1cm}
     \begin{subfigure}[b]{0.35\textwidth}
    \centering
    \begin{tikzpicture}[main_node/.style={circle,fill=white,draw,minimum size=.2em,inner sep=1.5pt},cp/.style={fill=white,draw,rounded corners=.09cm,minimum size=1.4em,inner sep=1.4pt},scale=0.9]
    \node[main_node,label=below:{}] (1) at (0,1) {0};
    \node[main_node,label=below:{}] (2) at (0,0) {0};
    \node[main_node,label=below:{}] (3) at (0,-1) {0};
    \node[cp,label=below:{}] (4) at (1.5,0) {$k+1$};
    \node[main_node,label=below:{}] (5) at (3,0) {0};
    \node[cp,label=below:{}] (6) at (4.5,1) {$k$};
    \node[cp,label=below:{}] (7) at (4.5,0) {$k$};
    \node[cp,label=above:{}] (8) at (4.5,-1) {$k$};
    \draw (1) -- (4) -- (2) (3) -- (4);
    \draw (6) -- (5) -- (7)  (4) -- (5) -- (8);
    \end{tikzpicture}
         \caption{Tree with a [$k$]-IRDF.}
         \label{fig:exampleDRDF}
     \end{subfigure}
    \hspace{.1cm}
    \begin{subfigure}[b]{0.35\textwidth}
    \centering
    \begin{tikzpicture}[main_node/.style={circle,fill=white,draw,minimum size=.2em,inner sep=1.5pt},cp/.style={fill=white,draw,rounded corners=.09cm,minimum size=1.3em,inner sep=1.5pt},scale=0.9]
    \node[main_node,label=below:{}] (1) at (0,0) {0};
    \node[cp,label=below:{}] (2) at (1.5,0) {$k+1$};
    \node[main_node,label=below:{}] (3) at (3,0) {0};
    \node[main_node,label=below:{}] (5) at (3,2) {0};
    \node[cp,label=below:{}] (6) at (1.5,2) {$k+1$};
    \node[main_node,label=below:{}] (7) at (0,2) {0};
    \draw (1) -- (2) -- (3) --  (5) -- (6) -- (7) -- (1);
    \end{tikzpicture}  
         \caption{Cycle with a [$k$]-IRDF.}
         \label{fig:exampleDRDF2}
     \end{subfigure}

\caption{Three graphs endowed with independent [$k$]-Roman dominating functions.}
\label{fig:IRDFgraphs}
\end{figure}
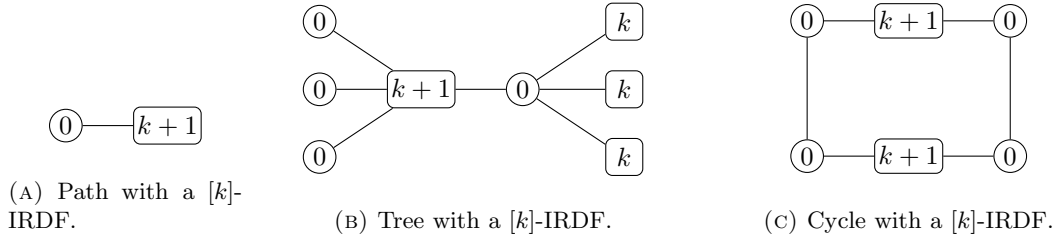

From the definition of independent  [$k$]-Roman domination, we know that the active vertices $v \in V(G)$ with $f(v) < k$ must have at least one active neighbor since the condition $f(N[v]) \geq k+|AN(v)|$ must be satisfied. In addition to the previous condition, an independent [$k$]-Roman domination function also imposes that the set of active vertices must be independent. However, these two conditions considered simultaneously imply that an independent [$k$]-Roman dominating function does not assign labels from the set $\{1,2,\ldots,k-1\}$ to the vertices of a graph $G$. These initial observations concerning [$k$]-IRDFs are explicitly stated in the following propositions.

\begin{proposition}
\label{prop:easyIneqGamma}
If $G$ is a graph, then $\gamma_{[kR]}(G) \leq i_{[kR]}(G)$.\qed
\end{proposition}

\begin{proposition}
\label{prop:emptyV1idR}
If $f=(V_0,V_1,\ldots,V_{k+1})$ is a [$k$]-IRDF of a graph $G$, then $V_i = \emptyset$ for all $i \in \{1,2,\ldots,k-1\}$.\qed
\end{proposition}

By Proposition~\ref{prop:emptyV1idR}, we can represent a [$k$]-IRDF $f=(V_0,V_1,\ldots,V_{k+1})$ simply as $f=(V_0,V_k,V_{k+1})$. Moreover, note that the weight of a [$k$]-IRDF $f=(V_0,V_k,V_{k+1})$ is also given by $\omega(f)=k|V_k|+(k+1)|V_{k+1}|$.

In this paper, we propose the study of independent [$k$]-Roman domination on graphs for arbitrary $k \geq 1$. The next sections of this paper are organized as follows. In Section~\ref{sec:complexity}, we prove that, for all $k \geq 3$, the decision versions of the Independent [$k$]-Roman Domination Problem and [$k$]-Roman Domination Problem are $\mathcal{NP}$-complete, even when restricted to planar bipartite graphs with maximum degree 3. 
In Section~\ref{sec:bounds}, we present some sharp lower and upper bounds for the independent [$k$]-Roman domination number of arbitrary graphs. 
In Sections~\ref{sec:blanusa} and~\ref{sec:loupekine}, we present specific lower bounds and upper bounds for the independent [$k$]-Roman domination number for two infinite families of 3-regular graphs called Generalized Blanu\v{s}a Snarks and Loupekine snarks.
Section~\ref{sec:conclusion}  presents our concluding remarks.

% ----------------------------
% SECTION
% ----------------------------
\section{Complexity results}
\label{sec:complexity}

In this section, we show that, for every integer $k \geq 3$, the decision versions of the [$k$]-Roman Domination Problem ([$k$]-ROM-DOM) and the Independent [$k$]-Roman Domination Problem ([$k$]-IROM-DOM) are $\np$-complete when restricted to graphs with maximum degree 3. We remark that the $\np$-completeness of [1]-ROM-DOM and [1]-IROM-DOM, when restricted to this same class of graphs, have already been established~\cite{Luiz2024}. In the remaining of this section, we deal with $k \geq 3$.  Consider the following decision problems. 

\bigskip 

\noindent \textbf{[$k$]-ROM-DOM}\\
\noindent \textbf{Instance:} A graph $G$ and a positive integer $\ell$.\\
\noindent \textbf{Question:} Does $G$ have a [$k$]-Roman dominating function with weight at most $\ell$?

\bigskip

\noindent \textbf{[$k$]-IROM-DOM}\\
\noindent \textbf{Instance:} A graph $G$ and a positive integer $\ell$.\\
\noindent \textbf{Question:} Does $G$ have an independent [$k$]-Roman dominating function with weight at most $\ell$?

\bigskip

For $k\geq 3$, we show that [$k$]-ROM-DOM and [$k$]-IROM-DOM are $\mathcal{NP}$-complete when restricted to graphs with maximum degree 3 through a reduction from the vertex cover problem. A \textit{vertex cover} of a graph $G$ is a set of vertices $S \subseteq V(G)$ such that each edge of $G$ is incident to some vertex in $S$. The \textit{vertex covering number} of $G$, denoted $\tau(G)$, is the cardinality of a smallest vertex cover of $G$. Given a graph $G$ and a positive integer $\ell$, the \textit{Vertex Cover Problem}  consists in deciding whether $G$ has a vertex cover $S$ with cardinality at most $\ell$. B.~Mohar~\cite{MOHAR2001102} proved  that the Vertex Cover Problem is $\mathcal{NP}$-complete even when restricted to 2-connected planar 3-regular graphs. We use this result  to construct a polynomial time reduction from the Vertex Cover Problem to [$k$]-ROM-DOM ([$k$]-IROM-DOM). The construction is described as follows: given a 2-connected planar 3-regular graph $G$, construct a new graph $F$ from $G$ by replacing each edge $e = uv \in E(G)$ by a gadget $G_e$ illustrated in Figure~\ref{fig:gagdet}. Note that $F$ can be constructed in polynomial time on $|E(G)|$ and also that $F$ is a planar bipartite graph with maximum degree 3.

\begin{figure}[!htb]
    \centering
    \begin{tikzpicture}[main_node/.style={circle,fill=black,draw,minimum size=.2em,inner sep=1.5pt]},scale=0.8]
    \node[main_node,label=below:{$u$}] (1) at (-3,0) {};
    \node[main_node,label=below:{$x_1^e$}] (2) at (-1.5,0) {};
    \node[main_node,label=below:{$x_2^e$}] (3) at (0,0) {};
    \node[main_node,label=below:{$x_3^e$}] (4) at (1.5,0) {};
    \node[main_node,label=below:{$x_4^e$}] (5) at (3,0) {};
    \node[main_node,label=below:{$x_5^e$}] (6) at (4.5,0) {};
    \node[main_node,label=below:{$v$}] (7) at (6,0) {};
    \node[main_node,label=above:{$x_6^e$}] (8) at (-1.5,1.5) {};
    \node[main_node,label=above:{$x_7^e$}] (9) at (0,1.5) {};
    \node[main_node,label=above:{$x_8^e$}] (10) at (1.5,1.5) {};
    \node[main_node,label=above:{$x_9^e$}] (11) at (3,1.5) {};
    \node[main_node,label=above:{$x_{10}^e$}] (12) at (4.5,1.5) {};
    \draw (8) -- (9) -- (10) -- (11) -- (12);
    \draw (1) -- (2) -- (3) -- (4) -- (5) -- (6) -- (7);
    \draw (3) -- (9) (5) -- (11);
\end{tikzpicture}
    \caption{Gadget $G_e$ used in the reduction.}
    \label{fig:gagdet}
\end{figure}
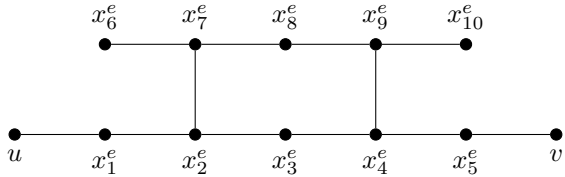

In order to prove the $\np$-completeness result, given in Theorem~\ref{thm:npcompleteness}, we need the following auxiliary results.

\begin{lemma}[Khalili et al.~\cite{Khalili2023}]
\label{lemma:label1}
If $k \geq 2$, then in a $\gamma_{[kR]}(G)$-function of a graph $G$, no vertex needs to be assigned the label 1.\qed
\end{lemma}

\begin{proposition}
\label{prop:prohibitedLabel1}
Let $k\geq 1$ be an integer. 
If $G$ is a connected graph with at least 3 vertices, then in a [$k$]-Roman dominating function of $G$ with weight $\gamma_{[kR]}(G)$, no leaf vertex of $G$ needs to be assigned the label $k+1$.
\end{proposition}

\begin{proof}
Let $G$ be a connected graph on at least 3 vertices and $f$ a $\gamma_{[kR]}$-function of $G$. Let $v \in V(G)$ be a leaf vertex and let $w \in V(G)$ be the neighbor of $v$. For the purpose of contradiction, suppose that $f$ needed to assign $k+1$ to $v$, that is, $f(v)=k+1$. Since $f$ is a $\gamma_{[kR]}$-function, $f(w) \leq k$ (otherwise, by assigning 0 to $v$ we obtain a [$k$]-RDF with weight smaller then $\omega(f)$). We modify the labeling $f$ by exchange the labels of the vertices $v$ and $w$ and maintaining the labels of all the other vertices the same. Note that $f$ continues to be a [$k$]-Roman dominating function with the same weight as before and the vertex $v$ does not have the weight $k+1$ anymore.
\end{proof}

\begin{proposition}
\label{prop:aux6543}
Let $k\geq 1$ be an integer and $G$ be a connected graph with at least 3 vertices. In any $\gamma_{[kR]}$-function $f$ of $G$, no leaf vertex needs to be assigned a label different from 0 or $k$. 
\end{proposition}

\begin{proof}
Let $G$ and $f$ be as in the hypothesis and let $v\in V(G)$ be a leaf vertex with neighbor $w \in V(G)$. By Lemma~\ref{lemma:label1} and Proposition~\ref{prop:prohibitedLabel1}, $f(v) \not\in \{1,k+1\}$. If $f(v) \in \{0,k\}$, then $f$ is the desired function. Thus, suppose that $f$ needs to assign a label in the set $\{2,3,\ldots,k-1\}$ to vertex $v$, that is, $f(v) \in \{2,3,\ldots,k-1\}$. In this case, the neighbor $w$ of $v$ has $f(w)\neq 0$ and is, thus, an active neighbor of $v$. By the definition of [$k$]-RDF, $f(N[v]) = f(v)+f(w) \geq k+|AN(v)| = k + 1$. Thus, $f(v)+f(w) \geq k+1$. We modify the labeling $f$ by assigning label $f(v)+f(w)$ to vertex $w$, by assigning label 0 to vertex $v$, and maintaining the labels of all the remaining vertices of $G$ the same. Note that $f$ continues to be a [$k$]-RDF with the same weight as before and the new label of $v$ does not belong to the set $\{2,3,\ldots,k-1\}$, which is a contradiction.
\end{proof}

\begin{lemma}
\label{lemma:aux1}
Let $k \geq 2$ be an integer. 
Given a 2-connected planar 3-regular graph $G$, let $F$ be the graph constructed from  $G$ by replacing each edge $e=uv$ in $G$ by a gadget $G_e$ shown in Figure~\ref{fig:gagdet}. Then, any $\gamma_{[kR]}$-function $f$ of $F$ satisfies   $(f(x_6^e),f(x_7^e)) \in \{(0,k+1),(k,0)\}$ and $(f(x_{9}^e),f(x_{10}^e)) \in \{(k+1,0),(0,k)\}$.
\end{lemma}

\begin{proof}
Let $f$ be a $\gamma_{[kR]}$-function of $F$. We only analyze the values $f(x_6)$ and $f(x_7)$ since the analysis for $f(x_{9}^e)$ and $f(x_{10}^e)$ is analogous and follows from the symmetry of $F$ along the vertical axis.

Note that $x_6^e$ is a leaf vertex and $N(x_6^e)=\{x_7^e$\}. By Proposition~\ref{prop:aux6543}, either $f(x_6^e) = 0$ or $f(x_6^e) = k$. If $f(x_6^e)=0$, then $f(x_7^e)=k+1$ by the definition of [$k$]-RDF, and the result follows. Thus, suppose that $f(x_6^e) = k$. By Lemma~\ref{lemma:label1}, $f(x_7^e)\neq 1$. If $f(x_7^e) \geq 2$, then $f(x_6^e)+f(x_7^e) \geq k+2$. Hence, it would be possible to obtain a [$k$]-RDF with smaller weight by assigning label $f(x_6^e)+f(x_7^e)-1$ to $x_7^e$ and 0 to $v$, thus contradicting the choice of $f$. Therefore, we obtain that $f(x_7^e)=0$, and the result follows.
\end{proof}

\begin{lemma}
\label{lemma:aux2}
Let $k \geq 3$ be an integer. 
Given a 2-connected planar 3-regular graph $G$, let $F$ be a graph constructed from $G$ by replacing each edge $e=uv$ in $G$ by a gadget $G_e$ illustrated in Figure~\ref{fig:gagdet}. Let $U_e =\{x_2^e,x_3^e,x_4^e,x_6^e,x_7^e,x_8^e,x_9^e,x_{10}^e\} \subset V(G_e)$. Then, in any $\gamma_{[kR]}$-function $f$ of the graph $F$, we have that the function $f$ restricted to $U_e$ is a [$k$]-RDF of $F[U_e]$ with  weight $f(U_e) = 3k+2$. Moreover, $(f(x_2^e),f(x_4^e)) \in \{(0,0),(k+1,0),(0,k+1)\}$. 
\end{lemma}

\begin{proof}
Let $k \geq 3$ be an integer. 
Let $G$ and $F$ be as in the hypothesis. Let $f\colon V(F) \to \{0,1,\ldots,k+1\}$ be a $\gamma_{[kR]}$-function of $F$. 
For each gadget $G_e \subset F$, define $U_e=\{x_2^e,x_3^e,x_4^e,x_6^e,x_7^e,x_8^e,x_9^e,x_{10}^e\} \subset V(G_e)$. By Lemma~\ref{lemma:aux1}, $(f(x_6^e),f(x_7^e)) \in \{(0,k+1),(k,0)\}$ and $(f(x_{9}^e),f(x_{10}^e)) \in \{(k+1,0),(0,k)\}$. Thus, there are four cases to analyze, depending on the values of the labels $f(x_6^e)$, $f(x_7^e)$, $f(x_9^e)$ and $f(x_{10}^e)$.

\smallskip

\noindent \textbf{Case 1:} 
$(f(x_6^e),f(x_7^e)) =(k,0)$ and $(f(x_{9}^e),f(x_{10}^e)) = (0,k)$. We claim that this case cannot occur. For the purpose of contradiction, suppose it occurs. Since $f$ is a $\gamma_{[kR]}$-function and $x_8^e$ has no active neighbor, we have that $f(x_8^e)=k$. 
Note that $\sum_{i=6}^{10}f(x_i^e) = 3k$. Thus, we can redefine the labels of some vertices of $F$ so as to obtain another [$k$]-RDF $f'$ of $F$ with smaller weight than $f$  such that $\sum_{i=6}^{10}f'(x_i^e) = 2k+2 < 3k$, as follows: define $(f'(x_6^e),f'(x_7^e),f'(x_8^e),f'(x_9^e),f'(x_{10}^e))=(0,k+1,0,k+1,0)$ and make $f'(x)=f(x)$ for every remaining vertex $x$ of $F$. This contradicts the choice of $f$ as a $\gamma_{[kR]}$-function.

\smallskip 

\noindent \textbf{Case 2:} $(f(x_6^e),f(x_7^e)) =(0,k+1)$ and $(f(x_{9}^e),f(x_{10}^e)) = (k+1,0)$. Since $f$ is a $\gamma_{[kR]}$-function, we have that $f(x_8^e)=0$. By the definition of [$k$]-RDF, we have that $f(N[x_3^e]) = f(x_2^e)+f(x_3^e)+f(x_4^e) \geq k+|AN(x_3^e)|\geq k$. All these facts imply that $f(U_e) = \sum_{w \in U_e}f(w) \geq 3k+2$. Moreover, a [$k$]-RDF of $F[U_e]$ with weight $3k+2$ is obtained by assigning labels $f(x_2^e)=0$, $f(x_3^e)=k$ and $f(x_4^e)=0$. Therefore, $f(U_e)=3k+2$,  $(f(x_2^e),f(x_4^e))=(0,0)$, and the result follows.

\smallskip 

\noindent \textbf{Case 3:} $(f(x_6^e),f(x_7^e)) =(0,k+1)$ and $(f(x_{9}^e),f(x_{10}^e)) = (0,k)$. Since $f$ is a $\gamma_{[kR]}$-function, we have that $f(x_8^e)=0$. Moreover, since $f(x_9^e)=0$ and $f(x_8^e)+f(x_{10}^e) < k+1$ we obtain that $f(x_4^e)\neq 0$. By the definition of [$k$]-RDF and since $f(x_4^e)\neq 0$, we have that $f(N[x_3^e]) = f(x_2^e)+f(x_3^e)+f(x_4^e) \geq k+|AN(x_3^e)|\geq k+1$. All these facts imply that $f(U_e) = \sum_{w \in U_e}f(w) \geq 3k+2$. From the previous facts, we have that $f(U_e) = 3k+2$ only if $f(N[x_3^e]) = k+|AN(x_3^e)| = k+1$, which implies that $f(x_2^e)=0$. Thus, $f(x_2^e)+f(x_3^e)+f(x_4^e) = 0+f(x_3^e)+f(x_4^e) = k+1$, i.e.,  $f(x_3^e)+f(x_4^e) = k+1$. Since $f$ is a $\gamma_{[kR]}$-function, we obtain that $f(x_3^e)=0$ and $f(x_4^e)=k+1$. Therefore, $(f(x_2^e),f(x_4^e))=(0,k+1)$ and the result follows.

\smallskip 

\noindent \textbf{Case 4:} $(f(x_6^e),f(x_7^e)) =(k,0)$ and $(f(x_{9}^e),f(x_{10}^e)) = (k+1,0)$. The proof for this case is analogous to the proof of the previous case and follows from the symmetry of $G_e$ along the vertical axis.

\smallskip 

Therefore, in any $\gamma_{[kR]}$-function $f$ of the graph $F$, we have that the function $f$ restricted to $U_e$ is a [$k$]-RDF of $F[U_e]$ with  weight $f(U_e) = 3k+2$. Moreover, $(f(x_2^e),f(x_4^e)) \in \{(0,0),(k+1,0),(0,k+1)\}$.
\end{proof}

\begin{theorem}
\label{thm:BetaGammaUpperBound}
Let $k \geq 3$ be an integer.
Given a 2-connected planar 3-regular graph $G$, let $F$ be a planar bipartite graph with $\Delta(F)=3$ constructed from $G$ by replacing each edge $e=uv$ in $G$ by a gadget $G_e$ illustrated in Figure~\ref{fig:gagdet}. Then, \[\gamma_{[kR]}(F) = i_{[kR]}(F) = \tau(G) + k|V(G)| + (3k+2)|E(G)|.\]
\end{theorem}

\begin{proof}
Let $G$ and $F$ be as in the statement of the theorem.
Let $C$ be a vertex cover of $G$ with $|C| = \tau(G)$.  

We initially prove that $i_{[kR]}(F) \leq \tau(G) + k|V(G)| + (3k+2)|E(G)|$.
In order to do this, we construct an appropriate [$k$]-IRDF $f=(V_0,V_k,V_{k+1})$ of $F$ as follows. First, define two empty sets $D_k$ and $D_{k+1}$. For each gadget $G_e \subset F$, associated with an edge $e = uv \in E(G)$, do the following: if $v \in C$, then, add the vertex $x_6^e$ to $D_k$ and add the vertices $x_2^e$ and $x_9^e$ to $D_{k+1}$; otherwise, add the vertex $x_{10}^e$ to $D_k$ and add the vertices $x_4^e$ and $x_7^e$ to $D_{k+1}$. Note that $|D_k| = |E(G)|$ and $|D_{k+1}| = 2|E(G)|$. Define the function $f=(V_0,V_k,V_{k+1})$ such that $V_0 = V(F)\backslash(V(G) \cup D_k \cup D_{k+1})$, $V_k = D_k \cup V(G)\backslash C$ and $V_{k+1}=D_{k+1} \cup C$.
From the definition of $f$, we have that $f$ is a [$k$]-IRDF of $F$ with weight $\omega(f) = k|D_k \cup V(G)\backslash C| + (k+1)|D_{k+1} \cup C| = k(|E(G)|+|V(G)|-\tau(G)) + (k+1)(2|E(G)|+\tau(G)) = \tau(G) + k|V(G)| + (3k+2)|E(G)|$. Therefore, $i_{[kR]}(F) \leq \omega(f) = \tau(G) + k|V(G)| + (3k+2)|E(G)|$.

Next, we show that $\gamma_{[kR]}(F) \geq \tau(G) + k|V(G)| + (3k+2)|E(G)|$.
Let $f = (V_0, \emptyset, V_2,\ldots,V_{k+1})$ be a $\gamma_{[kR]}$-function of $F$. Let $G_e$ be a gadget of $F$, for any edge $e = uv \in E(G)$. Define the set $U_e = \{x_2^e,x_3^e,x_4^e,x_6^e,x_7^e,x_8^e$, $x_9^e,x_{10}^e\} \subset V(G_e) \subset V(F)$. By Lemma~\ref{lemma:aux2}, the function $f$ restricted to $U_e$ has weight $3k+2$ and is a [$k$]-RDF of the subgraph induced by $U_e$. Moreover, $(f(x_2^e),f(x_4^e)) \in \{(0,k+1),(k+1,0),(0,0)\}$. Let $S = \{x \in U_e : f(x) \neq 0, e \in E(G) \}$. Let $V' \subset V(F)$ be the set of vertices that are not adjacent to some vertex in $S$ and are not in $S$, that is, $V' = V(F)\backslash N[S]$. Let $F'$ be the induced subgraph $F[V']$. For each $e \in E(G)$, all the vertices in $U_e$ and at most one of the vertices $x_1^e$ and $x_5^e$ are not in $V'$. This implies that $F'$ is a forest of trees with $|V(G)|$ components such that each component is a star whose central vertex is a vertex $z \in V(G)$. Let $T$ be a component of $F'$. If $T$ is a single vertex (i.e.~$V(T)=\{z\}$), then $f(z)=k$. On the other hand, if $T$ is not a single vertex, then $z$ is the central vertex of the star $T$ and $f(z)=k+1$. Let $D = V(G) \cap V_{k+1}$. From the above discussion, we conclude that $D$ is a vertex cover of $G$. Since each subset $U_e$ contributes with $3k+2$ to the weight of $f$ and there are $|E(G)|$ of these subsets, then they contribute to a total of $(3k+2)|E(G)|$ to the weight of $f$.  From these facts we obtain that $\omega(f) = (k+1)|D| + k(|V(G)|-|D|) + (3k+2)|E(G)| = |D|+k|V(G)|+(3k+2)|E(G)|$. Thus, $\tau(G) \leq |D| = \omega(f)-k|V(G)|-(3k+2)|E(G)| = \gamma_{[kR]}(G)-k|V(G)|-(3k+2)|E(G)|$. Therefore, $\gamma_{[kR]}(G) \geq \tau(G)+k|V(G)|+(3k+2)|E(G)|$.

Since $\gamma_{[kR]}(G) \leq i_{[kR]}(G)$ (see Proposition~\ref{prop:easyIneqGamma}), we have that $\tau(G)+k|V(G)|+(3k+2)|E(G)| \leq \gamma_{[kR]}(G) \leq i_{[kR]}(G) \leq \tau(G)+k|V(G)|+(3k+2)|E(G)|$, and the result follows.
\end{proof}

\begin{theorem}
\label{thm:npcompleteness}
Let $k\geq 3$ be an integer. Then, [$k$]-ROM-DOM (resp.~[$k$]-IROM-DOM) is $\mathcal{NP}$-complete even when restricted to planar bipartite graphs $G$ with $\Delta(G)=3$.
\end{theorem}

\begin{proof}
We first show that [$k$]-ROM-DOM (resp. [$k$]-IROM-DOM) is a member of $\mathcal{NP}$. Given any instance $(G,\ell)$ of [$k$]-ROM-DOM (resp. [$k$]-IROM-DOM) and a certificate function $f\colon V(G) \to \{0,1,\ldots,k+1\}$, we can verify (in polynomial time) if $\sum_{v \in V(G)}f(v)  \leq \ell$ and if $\sum_{u \in N[v]}f(u) \geq |AN(v)|+k$ for every $v \in V(G)$ (resp.~in the case of [$k$]-IROM-DOM, it is also necessary to check if $V_k \cup V_{k+1}$ is an independent set). Next, we show that [$k$]-ROM-DOM (resp. [$k$]-IROM-DOM) is $\mathcal{NP}$-hard. Recall that we showed how to construct a planar bipartite graph $F$ with $\Delta(F)=3$ from a given 2-connected planar 3-regular graph $G$ in polynomial time on $|E(G)|$. From Theorem~\ref{thm:BetaGammaUpperBound}, we deduce that there exists a polynomial time algorithm that calculates $\tau(G)$ if and only if  there exists a polynomial time algorithm that calculates $\gamma_{[kR]}(F)$ (resp.~$i_{[kR]}(F)$). However, since VCP is $\mathcal{NP}$-complete even when  restricted to 2-connected planar 3-regular graphs, we obtain, from this reduction,  that [$k$]-ROM-DOM (resp. [$k$]-IROM-DOM) is $\mathcal{NP}$-complete even when restricted to planar bipartite graphs with maximum degree 3. 
\end{proof}

% ----------------------------
% SECTION
% ----------------------------
\section{Bounds for the independent [$k$]-Roman domination number}
\label{sec:bounds}

In this section, we present some lower and upper bounds for the independent [$k$]-Roman domination number of arbitrary graphs. Since the set $V_k \cup V_{k+1}$ is an independent dominating set in every [$k$]-IRDF $f=(V_0,V_k,V_{k+1})$ of a graph $G$, it seems reasonable that $i_{[kR]}(G)$ and $i(G)$ are related, such as shown in the following proposition.

\begin{proposition}
\label{prop:boundsIndNumber}
Let $k \geq 1$ be an integer. 
If $G$ is a graph, then $k\cdot i(G) \leq i_{[kR]}(G) \leq (k+1)\cdot i(G)$.
\end{proposition}

\begin{proof}
Given a minimum independent dominating set $S$ of a graph $G$, we define a [$k$]-IRDF $f=(V_0,V_k,V_{k+1})$ of $G$ with weight $(k+1)i(G)$ by making $V_{k+1}=S$, $V_k=\emptyset$ and $V_0=V(G)\backslash S$. Hence, $i_{[kR]}(G) \leq \omega(f) = (k+1)i(G)$.

Now, let $f=(V_0,V_k,V_{k+1})$ be an $i_{[kR]}$-function of a graph $G$. Since $i(G)\leq |V_k|+|V_{k+1}|$, we have that $k\cdot i(G)\leq k(|V_k|+|V_{k+1}|) \leq k|V_k|+(k+1)|V_{k+1}| = i_{[kR]}(G)$, and the result follows.
\end{proof}

The lower bound presented in Proposition~\ref{prop:boundsIndNumber} is tight since it is attained by empty graphs. Moreover, graphs whose independent [$k$]-Roman domination number equals the upper bound given in Proposition~\ref{prop:boundsIndNumber} receive a specific name. We say that a graph $G$ is \textit{independent [$k$]-Roman} when  $i_{[kR]}(G) = (k+1)i(G)$. The next lemma is a generalization of a result of Shao et al.~\cite{8515207} and presents a characterization of independent [$k$]-Roman graphs.

\begin{lemma}
\label{lemma:indRomanGraph}
Let $G$ be a graph. Then, $G$ is independent [$k$]-Roman if and only if $G$ has an $i_{[kR]}$-function $f=(V_0,V_k,V_{k+1})$ such that $V_k = \emptyset$.
\end{lemma}

\begin{proof}
Let $G$ be a graph. First, suppose that $G$ has an $i_{[kR]}(G)$-function $f=(V_0,V_k,V_{k+1})$ such that $V_k = \emptyset$. This implies that $i(G) \leq |V_k|+|V_{k+1}| = |V_{k+1}|$ and that  $i_{[kR]}(G)=(k+1)|V_{k+1}|$. Thus, $(k+1)i(G) \leq (k+1)|V_{k+1}| = i_{[kR]}(G)$. By Proposition~\ref{prop:boundsIndNumber}, $i_{[kR]}(G) \leq (k+1) i(G)$. Therefore, $i_{[kR]}(G) = (k+1) i(G)$ and $G$ is independent [$k$]-Roman.

Now, consider $G$ independent [$k$]-Roman. For the purpose of contradiction, suppose that every $i_{[kR]}(G)$-function $f=(V_0,V_k,V_{k+1})$ has $V_k \neq \emptyset$. Let $f=(V_0,V_k,V_{k+1})$ be an $i_{[kR]}$-function of $G$ with $|V_k|$ as minimum as possible. From the definition of [$k$]-IRDF, we know that $i(G) \leq |V_k\cup V_{k+1}| = |V_k|+| V_{k+1}|$. 

In fact, we claim that $i(G) = |V_k|+| V_{k+1}|$. In order to prove this claim, suppose that there exists a minimum independent dominating set $S$ of $G$ such that $|S| < |V_k|+|V_{k+1}|$. Let $g = (V_0^g,V_k^g,V_{k+1}^g)$ be a function with $V_k^g=\emptyset$, $V_{k+1}^g=S$ and $V_0^g = V(G)\backslash S$. Thus, $g$ is a [$k$]-IRDF of $G$ with $|V_{k+1}^g|=i(G)$. Since $G$ is independent [$k$]-Roman, we have that $(k+1)i(G)=i_{[kR]}(G)$. Moreover, by the definition of [$k$]-IRDF, we know that $i_{[kR]}(G) = k|V_k|+(k+1)|V_{k+1}|$. Then, we have that $\omega(g)=(k+1)|V_{k+1}^g| = (k+1)i(G) = i_{[kR]}(G) = k|V_k|+(k+1)|V_{k+1}| = \omega(f)$. Thus, we have $\omega(g)=\omega(f)=i_{[kR]}(G)$, but $V_k^g = \emptyset$. In other words, we found an $i_{[kR]}(G)$-function $g = (V_0^g,V_k^g,V_{k+1}^g)$ with $V_k^g=\emptyset$, which is a contradiction. Therefore, $i(G) = |V_k|+| V_{k+1}|$ as claimed.

Since $G$ is independent [$k$]-Roman, we have that $(k+1)i(G)=i_{[kR]}(G)$ and,  thus, $(k+1)(|V_k|+|V_{k+1}|) = (k+1)i(G) = i_{[kR]}(G) = k|V_k|+(k+1)|V_{k+1}|$, implying that $|V_k|=0$, which is a contradiction. 

Therefore, we conclude that $G$ has an $i_{[kR]}(G)$-function $f=(V_0,V_k,V_{k+1})$ such that $V_k = \emptyset$.
\end{proof}

Another useful kind of lower bound connects the parameter with the maximum degree and number of vertices of the graph. As an example, in what concerns the [$k$]-Roman domination number, Valenzuela-Tripodoro et al.~\cite{Valenzuela-Tripodoro2024} presented the following lower bound for the [$k$]-Roman domination number of nontrivial connected graphs.

\begin{theorem}[Valenzuela-Tripodoro et al.~\cite{Valenzuela-Tripodoro2024}]
\label{thm:lowerBoundkRoman}
Let $k\geq 1$ be an integer. Let $G$ be a  nontrivial connected graph with maximum degree $\Delta(G) \geq k$. Then $\gamma_{[kR]}(G) \geq  
\frac{|V(G)|(k+1)}{\Delta(G)+1} $.\qed
\end{theorem}

Since $i_{[kR]}(G) \geq \gamma_{[kR]}(G)$, we immediatelly obtain the following corollary from Theorem~\ref{thm:lowerBoundkRoman}.

\begin{corollary}
\label{cor:lowerBoundkIndRoman}
Let $k\geq 1$ be an integer. Let $G$ be a  nontrivial connected graph with maximum degree $\Delta(G) \geq k$. Then $i_{[kR]}(G) \geq  
\frac{|V(G)|(k+1)}{\Delta(G)+1} $.\qed
\end{corollary}

We remark that Corollary~\ref{cor:lowerBoundkIndRoman} only applies for the cases where $k \leq \Delta(G)$. In the next theorem we present a new lower bound for the independent [$k$]-Roman domination number of connected graphs $G$ with $k \geq \Delta(G)$ and $k \geq 4$.

\begin{theorem}
\label{thm:lowerBoundKgDelta}
Let $k \geq 4$ be an integer. If $G$ is a  nontrivial connected graph with $k \geq \Delta(G) \geq 1$, then
\[
i_{[kR]}(G) \geq \frac{|V(G)|(k+1)}{\Delta(G)+1}.
\]
Moreover, if $i_{[kR]}(G) = \frac{|V(G)|(k+1)}{\Delta(G)+1}$, then $G$ is independent [$k$]-Roman.
\end{theorem}

\begin{proof}
Let $k \geq 4$ be an integer and $G$ be a nontrivial connected graph with maximum degree $\Delta \geq 1$. Suppose that $k \geq \Delta$. Let $f\colon V(G) \to \{0,k,k+1\}$ be an $i_{[kR]}(G)$-function. Recall that $V_i = \{w \in V(G) : f(w)=i\}$ for $i \in \{0,k,k+1\}$. In this proof, we use a discharging procedure similar to the approach used by Shao et al.~\cite{8515207}. Our discharging procedure is described as follows. Firstly, each vertex $v \in V(G)$ is assigned the initial charge $s(v) = f(v)$. Next, we apply the discharging procedure defined by means of the following two rules:

\medskip 

\textbf{Rule 1:} every vertex $v\in V(G)$ with $s(v)=k+1$ sends a charge of $\frac{k+1}{\Delta+1}$ to each vertex in $N(v)\cap V_0$;

\smallskip 

\textbf{Rule 2:} every vertex $v\in V(G)$ with $s(v)=k$ sends a charge of $\frac{(k-2)(k+1)}{k(\Delta+1)}$ to each vertex in $N(v)\cap V_0$.

\medskip 

Denote by $s'(v)$ the final charge of vertex $v$ after applying the discharging procedure. Note that:

\begin{enumerate}[I.]
\item for each vertex $v\in V(G)$ with $f(v)=k+1$, since it sends charge to at most $d_G(v)$ vertices, by Rule 1 we obtain that the final charge of $v$ is $s'(v) \geq s(v)-d_G(v)\frac{k+1}{\Delta+1} \geq (k+1)-\frac{\Delta(k+1)}{\Delta+1} = \frac{k+1}{\Delta+1}$, that is, $s'(v) \geq \frac{k+1}{\Delta+1}$;

\item for each vertex $v\in V(G)$ with $f(v)=k$, since it sends charge to at most $d_G(v)$ vertices, by Rule 2 we obtain that the final charge of $v$ is $s'(v) \geq s(v)-d_G(v)\frac{(k-2)(k+1)}{k(\Delta+1)} \geq k -\frac{\Delta(k-2)(k+1)}{k(\Delta+1)} = \frac{k^2+\Delta k+2\Delta}{k(\Delta+1)} > \frac{k^2+\Delta k}{k(\Delta+1)} = \frac{k+\Delta}{\Delta+1} \geq \frac{k+1}{\Delta+1}$, that is, $s'(v) > \frac{k+1}{\Delta+1}$.
\end{enumerate}

From the previous analysis, we obtain that $s'(v) \geq \frac{k+1}{\Delta+1}$ for all $v\in V(G)$ with $f(v)>0$. Now, let us analyze an arbitrary vertex $v \in V(G)$ with $f(v)=0$. Since $f$ is a [$k$]-IRDF, we have that $f(N[v]) = f(N(v)) \geq |AN(v)|+k$. So, either $v$ has at least one neighbor $w \in V_{k+1}$ or $v$ has at least two neighbors $u_1,u_2 \in V_k$. If $v$ has at least one neighbor $w \in V_{k+1}$, then $s'(v) \geq \frac{k+1}{\Delta+1}$ since $w$ sent a charge of $\frac{k+1}{\Delta+1}$ to $v$. On the other hand, if $v$ has at least two neighbors $u_1,u_2 \in V_k$, each of these neighbors sent a charge of $\frac{(k-2)(k+1)}{k(\Delta+1)}$ to $v$ and, thus, $s'(v) \geq 2\cdot \frac{(k-2)(k+1)}{k(\Delta+1)} \geq \frac{k+1}{\Delta+1}$ for $k \geq 4$. Hence, we obtain that $s'(v) \geq \frac{k+1}{\Delta+1}$ for all $v\in V(G)$. Moreover, since the discharging procedure does not change the total value of charge in $G$, we obtain that
\begin{equation}\label{eq:gf654}
i_{[kR]}(G) = \omega(f) = \sum_{v \in V(G)}f(v) = \sum_{v \in V(G)}s(v) = \sum_{v\in V(G)}s'(v) \geq \sum_{v \in V(G)}\frac{k+1}{\Delta+1}=|V(G)|\cdot \frac{(k+1)}{\Delta+1}. 
\end{equation}
Therefore, $i_{[kR]}(G) = \omega(f) \geq \frac{|V(G)|(k+1)}{\Delta+1}$. From now on, suppose that $\omega(f)=\frac{|V(G)|(k+1)}{\Delta+1}$. In this case, by the inequality chain~\eqref{eq:gf654}, we have that $s'(v)=\frac{k+1}{\Delta+1}$ for all $v \in V(G)$. This implies that no vertex of $G$ was assigned label $k$ since $s'(w) > \frac{k+1}{\Delta+1}$ for every vertex $w\in V(G)$ with $f(w)=k$. Hence, by Lemma~\ref{lemma:indRomanGraph}, $G$ is  independent [$k$]-Roman.
\end{proof}

The next result follows immediately from Corollary~\ref{cor:lowerBoundkIndRoman} and Theorem~\ref{thm:lowerBoundKgDelta}.

\begin{theorem}
\label{thm:lowerBoundDelta3}
Let $k \geq 1$ be an integer. If $G$ is a  nontrivial connected graph with $\Delta(G) \geq 3$, then
$i_{[kR]}(G) \geq \frac{|V(G)|(k+1)}{\Delta(G)+1}$.
\end{theorem}

The lower bound presented in Theorem~\ref{thm:lowerBoundDelta3} is tight, which can be seen by analyzing Cartesian products of some paths and cycles. Given arbitrary graphs $G$ and $H$, the \textit{Cartesian product} of $G$ and $H$ is the graph $G\square H$ with vertex set $V(G\square H)=\{(u,v) : u \in V(G), v \in V(H)\}$. Two vertices $(u_1,v_1)$ and $(u_2,v_2)$ of $G\square H$ are adjacent if and only if either $u_1 = u_2$ and $v_1v_2 \in E(H)$; or $v_1 = v_2$ and $u_1u_2 \in E(G)$. Let $P_2 = (w_1,w_2)$ be a path with two vertices and $C_{4p} = (v_1,v_2,\ldots, v_{4p})$ be a cycle with $4p$ vertices, $p \geq 1$. As an example, Figure~\ref{fig:CartesianProd} shows the graph $P_2\square C_8$. By Theorem~\ref{thm:lowerBoundDelta3},  $i_{[kR]}(P_2\square C_{4p}) \geq 2p(k+1)$. In addition, a [$k$]-IRDF of $P_2\square C_{4p}$ with weight $2p(k+1)$ is easily obtained by assigning label $k+1$ to the set of vertices $\{(w_1v_i) \colon i=2t, t\equiv 1\pmod{2}\} \cup \{(w_2v_j) \colon j=4t, t\geq 1\}$. Therefore, $i_{[kR]}(P_2\square C_{4p}) = 2p(k+1)$ and, by Theorem~\ref{thm:lowerBoundKgDelta}, $P_2\square C_{4p}$ is independent [$k$]-Roman for all $k \geq 4$.

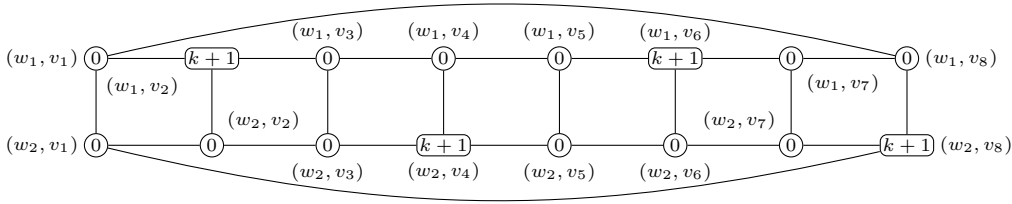
\begin{figure}[!htb]
\centering
\begin{tikzpicture}
[cpb/.style={fill=white,draw,rounded corners=.09cm,minimum size=.8em,inner sep=1pt]},
cpw/.style={circle,fill=white,draw,minimum size=.8em,inner sep=1pt]},scale=.77]
  \node[cpw, label={[label distance=-0.05cm]west:\scriptsize $(w_1,v_1)$}] (1) at (-2,2) {\scriptsize $0$};
  \node[cpw, label={[label distance=-0.05cm]west:\scriptsize $(w_2,v_1)$}] (2) at (-2,0.5) {\scriptsize $0$};
  \node[cpb,label={[label distance=-0.05cm]south west:\scriptsize $(w_1,v_2)$}] (3) at (0,2) {\scriptsize $k+1$};
  \node[cpw,label={[label distance=-0.05cm]north east:\scriptsize $(w_2,v_2)$}] (4) at (0,0.5) {\scriptsize $0$};
  \node[cpw,label={[label distance=-0.05cm]north:\scriptsize $(w_1,v_3)$}] (5) at (2,2) {\scriptsize 0};
  \node[cpw,label={[label distance=-0.05cm]south:\scriptsize $(w_2,v_3)$}] (6) at (2,0.5) {\scriptsize $0$};
  \node[cpw,label={[label distance=-0.05cm]north:\scriptsize $(w_1,v_4)$}] (7) at (4,2) {\scriptsize $0$};
  \node[cpb,label={[label distance=-0.05cm]south:\scriptsize $(w_2,v_4)$}] (8) at (4,0.5) {\scriptsize $k+1$};
  \node[cpw,label={[label distance=-0.05cm]north:\scriptsize $(w_1,v_5)$}] (9) at (6,2) {\scriptsize $0$};
  \node[cpw,label={[label distance=-0.05cm]south:\scriptsize $(w_2,v_5)$}] (10) at (6,0.5) {\scriptsize $0$};
  \node[cpb,label={[label distance=-0.05cm]north:\scriptsize $(w_1,v_6)$}] (11) at (8,2) {\scriptsize $k+1$};
  \node[cpw,label={[label distance=-0.05cm]south:\scriptsize $(w_2,v_6)$}] (12) at (8,0.5) {\scriptsize $0$};
  \node[cpw,label={[label distance=-0.05cm]south east:\scriptsize $(w_1,v_7)$}] (13) at (10,2) {\scriptsize $0$};
  \node[cpw,label={[label distance=-0.05cm]north west:\scriptsize $(w_2,v_7)$}] (14) at (10,0.5) {\scriptsize $0$};
  \node[cpw,label={[label distance=-0.05cm]east:\scriptsize $(w_1,v_8)$}] (15) at (12,2) {\scriptsize $0$};
  \node[cpb,label={[label distance=-0.05cm]east:\scriptsize $(w_2,v_8)$}] (16) at (12,0.5) {\scriptsize $k+1$};

   \path[-] (1) edge [bend left=13] (15);
   \path[-] (2) edge [bend right=13] (16);
    
   \draw (1) -- (2)  (3) -- (4)  (5) -- (6)  (7) -- (8)  (9) -- (10) (11) -- (12) (13) -- (14) (15) -- (16);

   \draw (1) -- (3) -- (5) -- (7) -- (9) -- (11) -- (13) -- (15) (2)--(4)--(6)--(8)--(10)--(12)--(14)--(16);
   
\end{tikzpicture}
\caption{Cartesian product $P_2 \square C_8$ with an $i_{[kR]}$-function.}
\label{fig:CartesianProd}
\end{figure}

% ----------------------------
% SECTION
% ----------------------------
\section{The infinite family of Generalized Blanu\v{s}a Snarks}
\label{sec:blanusa}

A \emph{cut-edge} of a graph $G$ is an edge whose deletion increases the number of connected components of $G$. A \emph{snark} is a connected 3-regular graph $G$ without cut-edges that does not admit an assignment of labels $f \colon E(G) \to \{1,2,3\}$ to its edges so that any two adjacent edges receive distinct labels. The origin of snarks is connected with the Four-Color Problem~\cite{Tait1880} and their study began in 1898 when the first snark was constructed by Petersen~\cite{Petersen1898}. In 1946, Blanu\v{s}a constructed two snarks, called Blanu\v{s}a snarks \cite{blanusa1946problem}. From Blanu\v{s}a snarks, Watkins constructed two infinite families of snarks, called Generalized Blanu\v{s}a Snarks~\cite{watkins1983construction}, which are considered in this section. 

Luiz~\cite{Luiz2024} determined the exact value of the parameter $i_{[1R]}(G)$ for every generalized Blanu\v{s}a snark $G$. Therefore, in this section, we only analyze values of $i_{[kR]}$ for the generalized Blanu\v{s}a snarks for values of $k \geq 2$.

The members of the family of \emph{generalized Blanu\v{s}a snarks} are graphs formed from  subgraphs called \emph{construction blocks}, denoted $B_0^1$, $B_0^2$ and $L$ (see Figure~\ref{fig:BlanusaBlocks}). A generalized Blanu\v{s}a snark contains as subgraphs one of the graphs $B_0^1, B_0^2$ and $i$ copies of the graph $L$, called $L_1,L_2,\ldots,L_i$. Vertices $a$, $b$, $c$ and $d$, belonging to both $B_0^1$ and $B_0^2$, and the vertices $x_j$, $y_j$, $w_j$ and $z_j$, belonging to $L$, are called \emph{border vertices}.

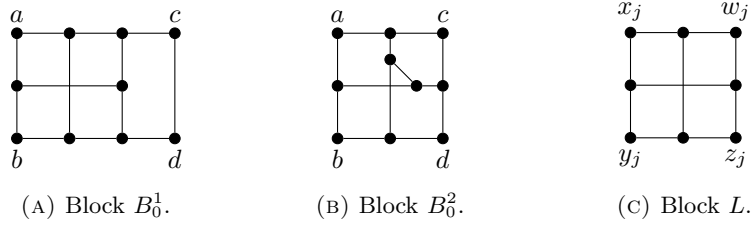
\begin{figure}[!htb]
    \centering
    \begin{subfigure}[b]{0.3\textwidth}
    \centering
    \begin{tikzpicture}
        [cp/.style={circle,fill=black,draw,minimum size=.4em,inner sep=.8pt]},scale=.7]
          \node[cp, label={[label distance=-0.05cm]north:$a$}] (a) at (-2,2) {};
          \node[cp] (1) at (-2,1) {};
          \node[cp, label={[label distance=-0.05cm]south:$b$}] (c) at (-2,0) {};
          \node[cp] (2) at (-1,2) {};
          \node[cp] (3) at (-1,0) {};
          \node[cp] (4) at (0,2) {};
          \node[cp] (5) at (0,1) {};
          \node[cp] (6) at (0,0) {};
          \node[cp, label={[label distance=-0.05cm]north:$c$}] (b) at (1,2) {};
          \node[cp, label={[label distance=-0.05cm]south:$d$}] (d) at (1,0) {};
          
           \draw (a) -- (1) -- (c) -- (3) -- (6) -- (5) -- (4) -- (2) -- (a);
           
           \draw (2) -- (3);
           \draw (1) -- (5);
           \draw (6) -- (d) -- (b) -- (4);
        \end{tikzpicture}
    \captionsetup{width=2cm}
    \caption{Block $B_0^1$.}
    \label{fig:blocoA1}
    \end{subfigure}
    \hspace{-1cm}
    \begin{subfigure}[b]{0.3\textwidth}
        \centering
        \begin{tikzpicture}
        [cp/.style={circle,fill=black,draw,minimum size=.4em,inner sep=1pt]},scale=.7]
          \node[cp, label={[label distance=-0.05cm]north:$a$}] (a) at (-2,2) {};
          \node[cp] (1) at (-2,1) {};
          \node[cp, label={[label distance=-0.05cm]south:$b$}] (c) at (-2,0) {};
          \node[cp] (2) at (-1,2) {};
          \node[cp] (3) at (-1,0) {};
          \node[cp, label={[label distance=-0.05cm]north:$c$}] (b) at (0,2) {};
          \node[cp] (4) at (0,1) {};
          \node[cp, label={[label distance=-0.05cm]south:$d$}] (d) at (0,0) {};
          
          \node[cp] (m) at (-1,1.5) {};
          \node[cp] (n) at (-0.5,1) {};
          
           \draw (a) -- (1) -- (c) -- (3) -- (d) -- (4) -- (b) -- (2) -- (a);
           
           \draw (2) -- (m) -- (3);
           \draw (m) -- (n) -- (4);
           \draw (1) -- (n);
           
        \end{tikzpicture}
        \captionsetup{width=2cm}
        \caption{Block $B_0^2$.}
        \label{fig:blocoA2}
    \end{subfigure}
    \hspace{-1cm}
    \begin{subfigure}[b]{0.3\textwidth}
        \centering
        \begin{tikzpicture}
        [cp/.style={circle,fill=black,draw,minimum size=.4em,inner sep=1pt]},scale=.7]
          \node[cp, label={[label distance=-0.05cm]north:$x_j$}] (xj) at (-2,2) {};
          \node[cp] (1) at (-2,1) {};
          \node[cp, label={[label distance=-0.05cm]south:$y_j$}] (yj) at (-2,0) {};
          \node[cp] (2) at (-1,2) {};
          \node[cp] (3) at (-1,0) {};
          \node[cp,label={[label distance=-0.05cm]north:$w_j$}] (wj) at (0,2) {};
          \node[cp] (4) at (0,1) {};
          \node[cp,label={[label distance=-0.05cm]south:$z_j$}] (zj) at (0,0) {};
          
           \draw (xj) -- (1) -- (yj) -- (3) -- (zj) -- (4) -- (wj) -- (2) -- (xj);
           
           \draw (1) -- (4);
           \draw (2) -- (3);
           
        \end{tikzpicture}
        \captionsetup{width=4cm}
        \caption{Block $L$.}
        \label{fig:blocoLj}
    \end{subfigure}
    \caption{Construction blocks $B_0^1$, $B_0^2$ and $L$ of the generalized Blanu\v{s}a snarks.}
    \label{fig:BlanusaBlocks}
\end{figure}

In the next paragraphs, we define these families of graphs based on a recursive construction. Let $\mathfrak{B}^1 = \{B^1_1, B^1_2, B^1_3, \ldots\}$ and  $\mathfrak{B}^2 = \{B^2_1, B^2_2, B^2_3, \ldots\}$ be the first and the second families of generalized Blanu\v{s}a snarks, respectively. The first member of $\mathfrak{B}^1$, the snark $B_1^1$, has vertex set $V(B_1^1) = V(B_0^1) \cup V(L_1)$ and edge set $E(B_1^1) = E(B_0^1) \cup E(L_1) \cup \{cy_1, dx_1, az_1, bw_1\}$ (see Figure~\ref{fig:snarkB11}). The second snark in $\mathfrak{B}^1$, snark $B^1_2$, has vertex set  $V(B_2^1) = V(B_0^1) \cup V(L_1) \cup V(L_2)$ and edge set $E(B_2^1) = E(B_0^1) \cup E(L_1) \cup E(L_2) \cup \{cy_1, dx_1, w_1y_2, z_1x_2, az_2, bw_2\}$ (see Figure~\ref{fig:snarkB12}).
The smallest snark of family $\mathfrak{B}^2$, graph $B_1^2$, has vertex set $V(B_1^2) = V(B_0^2) \cup V(L_1)$ and edge set $E(B_1^2) = E(B_0^2) \cup E(L_1) \cup \{bw_1, az_1, cy_1, dx_1\}$ (see Figure~\ref{fig:snarkB21}). The second snark in $\mathfrak{B}^2$, $B^2_2$, has vertex set $V(B_2^2) = V(B_0^2) \cup V(L_1) \cup V(L_2)$ and edge set $E(B_2^2) = E(B_0^2) \cup E(L_1) \cup E(L_2) \cup \{bw_2, az_2, cy_1, dx_1, z_1x_2, w_1y_2\}$ (see Figure~\ref{fig:snarkB22}).

\begin{figure}[!htb]
     \centering
     \begin{subfigure}[b]{.37\textwidth}
         \centering
         \begin{tikzpicture}
        [cp/.style={circle,fill=black,draw,minimum size=.4em,inner sep=1pt]},scale=.6]
          \node[cp, label={[label distance=-0.05cm]north:$a$}] (a) at (-2,2) {};
          \node[cp] (1) at (-2,1) {};
          \node[cp, label={[label distance=-0.05cm]south:$b$}] (b) at (-2,0) {};
          \node[cp] (2) at (-1,2) {};
          \node[cp] (3) at (-1,0) {};
          \node[cp] (4) at (0,2) {};
          \node[cp] (5) at (0,1) {};
          \node[cp] (6) at (0,0) {};
          \node[cp, label={[label distance=-0.05cm]north:$c$}] (c) at (1,2) {};
          \node[cp, label={[label distance=-0.05cm]south:$d$}] (d) at (1,0) {};
          \node[cp, label={[label distance=-0.05cm]south:$y_1$}] (7) at (2,0) {};
          \node[cp] (8) at (2,1) {};
          \node[cp, label={[label distance=-0.05cm]north:$x_1$}] (9) at (2,2) {};
          \node[cp] (10) at (3,0) {};
          \node[cp] (11) at (3,2) {};
          \node[cp, label={[label distance=-0.05cm]south:$z_1$}] (12) at (4,0) {};
          \node[cp] (13) at (4,1) {};
          \node[cp, label={[label distance=-0.05cm]north:$w_1$}] (14) at (4,2) {};
          
           \draw (a) -- (1) -- (b) -- (3) -- (6) -- (5) -- (4) -- (2) -- (a);
           
           \draw (2) -- (3);
           \draw (1) -- (5);
           \draw (6) -- (d) -- (c) -- (4);

           \draw (7)--(8)--(9)--(11)--(14)--(13)--(12)--(10)--(7);
           \draw (8)--(13);
           \draw (10)--(11);

            \draw (7)--(c);
            \draw (9)--(d);

            \draw (a) .. controls (-2.4,2) .. (-2.4,2.4);
            \draw (-2.4,2.4) .. controls (-2.4,2.8) .. (1.3,2.8); 
            \draw (1.3,2.8) .. controls (4.7,2.8) .. (4.7,1.4);
            \draw (4.7,1.4) .. controls (4.7,0) .. (12);

            \draw (b) .. controls (-2.4,0) .. (-2.4, -0.4);
            \draw (-2.4,-0.4) .. controls (-2.4,-0.8) .. (1.4,-0.8);
            \draw (1.4,-0.8) .. controls (5.2,-0.8) .. (5.2,0.6);
            \draw (5.2,0.6) .. controls (5.2,2) .. (14);
        \end{tikzpicture}
        \captionsetup{width=4cm}
         \caption{Generalized Blanu\v{s}a snark $B^1_1$.}
         \label{fig:snarkB11}
     \end{subfigure}
     \hspace{.5cm}
     \begin{subfigure}[b]{.55\textwidth}
         \centering          
            \begin{tikzpicture}
	            [cp/.style={circle,fill=black,draw,minimum size=.4em,inner sep=1pt]},scale=.6]
                  \node[cp, label={[label distance=-0.05cm]north:$a$}] (a) at (-5,2) {};
                  \node[cp] (1) at (-5,1) {};
                  \node[cp, label={[label distance=-0.05cm]south:$b$}] (c) at (-5,0) {};
                  \node[cp] (2) at (-4,2) {};
                  \node[cp] (3) at (-4,0) {};
                  \node[cp] (4) at (-3,2) {};
                  \node[cp] (5) at (-3,1) {};
                  \node[cp] (6) at (-3,0) {};
                  \node[cp, label={[label distance=-0.05cm]north:$c$}] (b) at (-2,2) {};
                  \node[cp, label={[label distance=-0.05cm]south:$d$}] (d) at (-2,0) {};
                  
                   \draw (a) -- (1) -- (c) -- (3) -- (6) -- (5) -- (4) -- (2) -- (a);
                   
                   \draw (2) -- (3);
                   \draw (1) -- (5);
                   \draw (6) -- (d) -- (b) -- (4);
                   
                  \node[cp, label={[label distance=-0.05cm]north:$x_1$}] (x1) at (-1,2) {};
                  \node[cp] (7) at (-1,1) {};
                  \node[cp, label={[label distance=-0.05cm]south:$y_1$}] (y1) at (-1,0) {};
                  \node[cp] (8) at (0,2) {};
                  \node[cp] (9) at (0,0) {};
                  \node[cp, label={[label distance=-0.05cm]north:$w_1$}] (w1) at (1,2) {};
                  \node[cp] (10) at (1,1) {};
                  \node[cp, label={[label distance=-0.05cm]south:$z_1$}] (z1) at (1,0) {};
                                    
                   \draw (x1) -- (7) -- (y1) -- (9) -- (z1) -- (10) -- (w1) -- (8) -- (x1);
                   
                   \draw (7) -- (10);
                   \draw (8) -- (9);
                   
                  \node[cp, label={[label distance=-0.05cm]north:$x_2$}] (x2) at (2,2) {};
                  \node[cp] (11) at (2,1) {};
                  \node[cp, label={[label distance=-0.05cm]south:$y_2$}] (y2) at (2,0) {};
                  \node[cp] (12) at (3,2) {};
                  \node[cp] (13) at (3,0) {};
                  \node[cp, label={[label distance=-0.05cm]north:$w_2$}] (w2) at (4,2) {};
                  \node[cp] (14) at (4,1) {};
                  \node[cp, label={[label distance=-0.05cm]south:$z_2$}] (z2) at (4,0) {};
                  
                   \draw (x2) -- (11) -- (y2) -- (13) -- (z2) -- (14) -- (w2) -- (12) -- (x2);
                   
                   \draw (11) -- (14);
                   \draw (12) -- (13);
                       
                   %%aresta de ligação
                   \draw (b) -- (y1);
                   \draw (d) -- (x1);
                   \draw (w1) -- (y2);
                   \draw (z1) -- (x2);
                   
                   \draw (w2) .. controls (5.5,2) .. (5.4,-0.6);
                   \draw (5.4,-0.6) .. controls (5.4,-1) .. (-4,-0.9);
                   \draw (-4,-0.9) .. controls (-5.6,-0.9) .. (-5.6, -0.45);
                   \draw (-5.6, -.45) .. controls (-5.6, 0) .. (c);
                   
                   \draw (a) .. controls (-5.6,2) .. (-5.6,2.45);
                   \draw (-5.6,2.45) .. controls (-5.6,2.9) .. (2,2.9);
                   \draw (2,2.9) .. controls (4.8,2.9) .. (4.8,1);
                   \draw (4.8,1) .. controls (4.8,0) .. (z2);
                   
	            \end{tikzpicture}
                
            \captionsetup{width=4cm}
            \caption{Generalized Blanu\v{s}a snark $B_2^1$.}
            \label{fig:snarkB12}
     \end{subfigure}
     \captionsetup{width=14cm}
     \caption{The first two smallest members of the family $\mathfrak{B}^1$.}
     \label{fig:firstTwoBlanusa}
\end{figure}
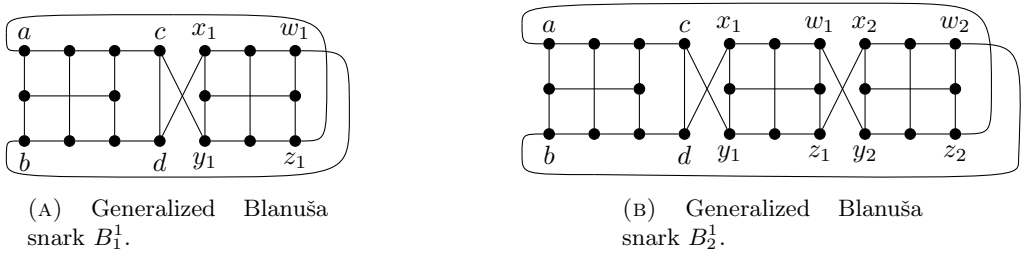

\begin{figure}[!htb]
     \centering
     \begin{subfigure}[b]{.37\textwidth}
        \centering
        \begin{tikzpicture}
            [cp/.style={circle,fill=black,draw,minimum size=.4em,inner sep=1pt]},scale=.65]

              \node[cp, label={[label distance=-0.05cm]north:$a$}] (a) at (-2,2) {};
              \node[cp] (1) at (-2,1) {};
              \node[cp, label={[label distance=-0.05cm]south:$b$}] (b) at (-2,0) {};
              \node[cp] (2) at (-1,2) {};
              \node[cp] (3) at (-1,0) {};
              \node[cp, label={[label distance=-0.05cm]north:$c$}] (c) at (0,2) {};
              \node[cp] (4) at (0,1) {};
              \node[cp, label={[label distance=-0.05cm]south:$d$}] (d) at (0,0) {};
              \node[cp] (m) at (-1,1.5) {};
              \node[cp] (n) at (-0.5,1) {};

              \node[cp,label={[label distance=-0.05cm]south:$y_1$}] (5) at (1,0) {};
              \node[cp] (6) at (1,1) {};
              \node[cp,label={[label distance=-0.05cm]north:$x_1$}] (7) at (1,2) {};
              \node[cp] (8) at (2,0) {};
              \node[cp] (9) at (2,2) {};
              \node[cp, label={[label distance=-0.05cm]south:$z_1$}] (10) at (3,0) {};
              \node[cp] (11) at (3,1) {};
              \node[cp, label={[label distance=-0.05cm]north:$w_1$}] (12) at (3,2) {};
              
               \draw (a) -- (1) -- (b) -- (3) -- (d) -- (4) -- (c) -- (2) -- (a);
               
               \draw (2) -- (m) -- (3);
               \draw (m) -- (n) -- (4);
               \draw (1) -- (n);

               \draw (5)--(6)--(7)--(9)--(12)--(11)--(10)--(8)--(5);
               \draw (8)--(9);
               \draw (6)--(11);

               \draw (c) -- (5);
               \draw (d) -- (7);

               \draw (a) .. controls (-2.4,2) .. (-2.4,2.4);
               \draw (-2.4,2.4) .. controls (-2.4,2.8) .. (0.6,2.8);
               \draw (0.5,2.8) .. controls (3.7,2.8) .. (3.7,1.4);
               \draw (3.7,1.4) .. controls (3.7,0) .. (10);

               \draw (b) .. controls (-2.4,0) .. (-2.4,-.4);
               \draw (-2.4,-.4) .. controls (-2.4,-.8) .. (0.9,-0.8);
               \draw (.9,-.8) .. controls (4.2,-.8) .. (4.2,.8);
               \draw (4.2,.8) .. controls (4.2,2) .. (12);
            \end{tikzpicture}
        \captionsetup{width=4cm}
         \caption{Generalized Blanu\v{s}a snark $B^2_1$.}
         \label{fig:snarkB21}
     \end{subfigure}
     \hspace{.5cm}
     \begin{subfigure}[b]{.55\textwidth}
        \centering          
        \begin{tikzpicture}
            [cp/.style={circle,fill=black,draw,minimum size=.4em,inner sep=1pt]},scale=.65]

              \node[cp, label={[label distance=-0.05cm]north:$a$}] (a) at (-2,2) {};
              \node[cp] (1) at (-2,1) {};
              \node[cp, label={[label distance=-0.05cm]south:$b$}] (b) at (-2,0) {};
              \node[cp] (2) at (-1,2) {};
              \node[cp] (3) at (-1,0) {};
              \node[cp, label={[label distance=-0.05cm]north:$c$}] (c) at (0,2) {};
              \node[cp] (4) at (0,1) {};
              \node[cp, label={[label distance=-0.05cm]south:$d$}] (d) at (0,0) {};
              \node[cp] (m) at (-1,1.5) {};
              \node[cp] (n) at (-0.5,1) {};

              \node[cp,label={[label distance=-0.05cm]south:$y_1$}] (5) at (1,0) {};
              \node[cp] (6) at (1,1) {};
              \node[cp,label={[label distance=-0.05cm]north:$x_1$}] (7) at (1,2) {};
              \node[cp] (8) at (2,0) {};
              \node[cp] (9) at (2,2) {};
              \node[cp, label={[label distance=-0.05cm]south:$z_1$}] (10) at (3,0) {};
              \node[cp] (11) at (3,1) {};
              \node[cp, label={[label distance=-0.05cm]north:$w_1$}] (12) at (3,2) {};

               \node[cp,label={[label distance=-0.05cm]south:$y_2$}] (13) at (4,0) {};
              \node[cp] (14) at (4,1) {};
              \node[cp,label={[label distance=-0.05cm]north:$x_2$}] (15) at (4,2) {};
              \node[cp] (16) at (5,0) {};
              \node[cp] (17) at (5,2) {};
              \node[cp, label={[label distance=-0.05cm]south:$z_2$}] (18) at (6,0) {};
              \node[cp] (19) at (6,1) {};
              \node[cp, label={[label distance=-0.05cm]north:$w_2$}] (20) at (6,2) {};

              \draw (13)--(14)--(15)--(17)--(20)--(19)--(18)--(16)--(13);
              \draw (16)--(17);
              \draw (14)--(19);

              \draw (12)--(13);
              \draw (10)--(15);
              
               \draw (a) -- (1) -- (b) -- (3) -- (d) -- (4) -- (c) -- (2) -- (a);
               
               \draw (2) -- (m) -- (3);
               \draw (m) -- (n) -- (4);
               \draw (1) -- (n);

               \draw (5)--(6)--(7)--(9)--(12)--(11)--(10)--(8)--(5);
               \draw (8)--(9);
               \draw (6)--(11);

               \draw (c) -- (5);
               \draw (d) -- (7);

               \draw (a) .. controls (-2.4,2) .. (-2.4,2.4);
               \draw (-2.4,2.4) .. controls (-2.4,2.9) .. (1.6,2.9);
               \draw (1.6,2.9) .. controls (6.7,2.9) .. (6.7,1.4);
               \draw (6.7,1.4) .. controls (6.7,0) .. (18);

               \draw (b) .. controls (-2.4,0) .. (-2.4,-.4);
               \draw (-2.4,-.4) .. controls (-2.4,-.9) .. (1.9,-0.9);
               \draw (1.9,-.9) .. controls (7.2,-.9) .. (7.2,.8);
               \draw (7.2,.8) .. controls (7.2,2) .. (20);
            \end{tikzpicture}       
        \captionsetup{width=4cm}
        \caption{Generalized Blanu\v{s}a snark $B^2_2$.}
        \label{fig:snarkB22}
     \end{subfigure}
     \captionsetup{width=14cm}
     \caption{The first two smallest members of the  family $\mathfrak{B}^2$.}
     \label{fig:firstTwoBlanusa2}
\end{figure}
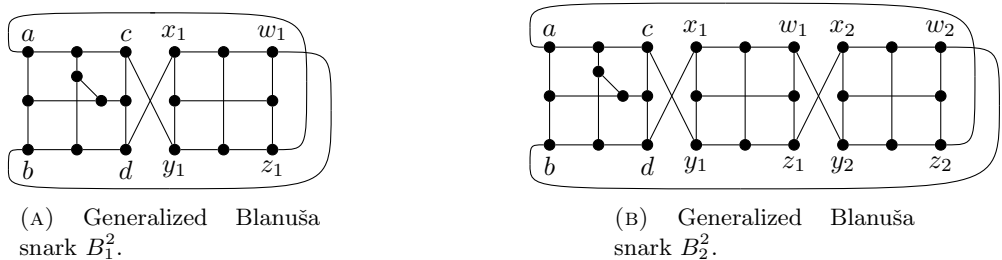

In order to construct larger generalized Blanu\v{s}a snarks, we use a subgraph $LG_i$, called \emph{link graph}, with vertex set $V(LG_i) = V(L_{i-1}) \cup V(L_i)$ and edge set $E(LG_i) = E(L_{i-1}) \cup E(L_i) \cup \{w_{i-1}y_{i}, z_{i-1}x_i\}$ (see  Figure~\ref{fig:linkgraphHi}). 
Let $t \in \{1,2\}$. For each integer $i$, with $i \geq 3$, the snark $B_i^t$ is obtained recursively from the snark $B_{i-2}^t$ and the link graph $LG_i$ according to the following rules: 

\begin{enumerate}[(i)]
\item $V(B_i^t) = V(B_{i-2}^t) \cup V(LG_i)$;

\item $E(B_i^t) = (E(B_{i-2}^t)\backslash E_{i-2}^{out}) \cup E(LG_i) \cup E_i^{in}$, where 
\begin{itemize}
\item $E_{i-2}^{out} = \{az_{i-2},bw_{i-2}\}$; and 
\item $E_i^{in} = \{w_{i-2}y_{i-1}, z_{i-2}x_{i-1}, az_i, bw_i\}$.
\end{itemize}
\end{enumerate}

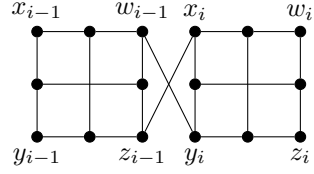
\begin{figure}[!htb]
    \centering
     \begin{tikzpicture}
        [cp/.style={circle,fill=black,draw,minimum size=.4em,inner sep=1pt]},scale=.7]
          \node[cp, label={[label distance=-0.05cm]north:$x_{i-1}$}] (xj) at (-2.5,2) {};
          \node[cp] (1) at (-2.5,1) {};
          \node[cp, label={[label distance=-0.05cm]south:$y_{i-1}$}] (yj) at (-2.5,0) {};
          \node[cp] (2) at (-1.5,2) {};
          \node[cp] (3) at (-1.5,0) {};
          \node[cp,label={[label distance=-0.05cm]north:$w_{i-1}$}] (wj) at (-.5,2) {};
          \node[cp] (4) at (-.5,1) {};
          \node[cp,label={[label distance=-0.05cm]south:$z_{i-1}$}] (zj) at (-.5,0) {};
          \node[cp, label={[label distance=-0.05cm]north:$x_i$}] (xj2) at (0.5,2) {};
          \node[cp] (5) at (.5,1) {};
          \node[cp, label={[label distance=-0.05cm]south:$y_i$}] (yj2) at (0.5,0) {};
          \node[cp] (6) at (1.5,2) {};
          \node[cp] (7) at (1.5,0) {};
          \node[cp,label={[label distance=-0.05cm]north:$w_i$}] (wj2) at (2.5,2) {};
          \node[cp] (8) at (2.5,1) {};
          \node[cp,label={[label distance=-0.05cm]south:$z_i$}] (zj2) at (2.5,0) {};

           \draw (xj) -- (1) -- (yj) -- (3) -- (zj) -- (4) -- (wj) -- (2) -- (xj);
           
           \draw (1) -- (4);
           \draw (2) -- (3);

           \draw (xj2) -- (5) -- (yj2) -- (7) -- (zj2) -- (8) -- (wj2) -- (6) -- (xj2);
           
           \draw (5) -- (8);
           \draw (6) -- (7);
           \draw (wj) -- (yj2);
           \draw (zj) -- (xj2);
           
        \end{tikzpicture}
    \captionsetup{width=4.07cm}
    \caption{The link graph $LG_i$.}
    \label{fig:linkgraphHi}
\end{figure}

Theorem~\ref{thm:upperiDRBlanusa} establishes an upper bound for the independent [$k$]-Roman domination number of generalized Blanu\v{s}a snarks.

\begin{theorem}
\label{thm:upperiDRBlanusa}
Let $k\geq 2$ be an integer. If $B_i^t$ is a generalized Blanu\v{s}a snark, with $t \in \{1,2\}$ and $i\geq 1$, then, 
\[ i_{[kR]}(B_i^t) \leq
  \begin{cases}
    (k+1)(2i+2)+2k  & \quad \text{if } t = 1 \text{ and } i \geq 3 \text{ with } i \text{ odd};\\
    (k+1)(2i+3)  & \quad otherwise.
  \end{cases}
\]
\end{theorem}

\begin{proof}
Initially, we separately show that the snark $B_i^t$, with $i=1$ and $t \in \{1,2\}$ has a [$k$]-IRDF with weight equal to $5(k+1) = (k+1)(2i+3)$. This special case is shown in Figure~\ref{fig:casoEspecialB1}, with $B_1^1$ and $B_1^2$ endowed with their respective [$k$]-IRDFs.

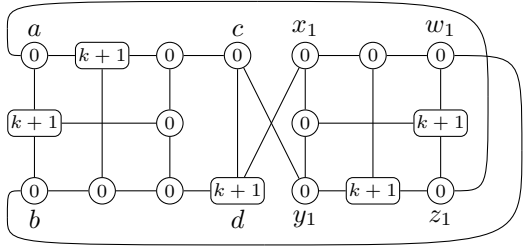
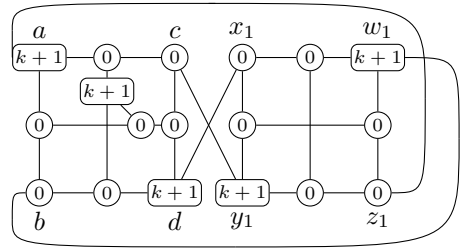
\begin{figure}[!htb]
        \centering
        \begin{subfigure}[b]{.45\textwidth}
            \centering
            \begin{tikzpicture}
        [cp/.style={circle,fill=white,draw,minimum size=1em,inner sep=1pt]},
        cpb/.style={fill=white,draw,rounded corners=.09cm,minimum size=1em,inner sep=1pt]},
        scale=.9]
          \node[cp, label={[label distance=-0.05cm]north:$a$}] (a) at (-2,2) {\scriptsize $0$};
          \node[cpb] (1) at (-2,1) {\scriptsize $k+1$};
          \node[cp, label={[label distance=-0.05cm]south:$b$}] (b) at (-2,0) {\scriptsize $0$};
          \node[cpb] (2) at (-1,2) {\scriptsize $k+1$};
          \node[cp] (3) at (-1,0) {\scriptsize $0$};
          \node[cp] (4) at (0,2) {\scriptsize $0$};
          \node[cp] (5) at (0,1) {\scriptsize $0$};
          \node[cp] (6) at (0,0) {\scriptsize $0$};
          \node[cp, label={[label distance=-0.05cm]north:$c$}] (c) at (1,2) {\scriptsize $0$};
          \node[cpb, label={[label distance=-0.05cm]south:$d$}] (d) at (1,0) {\scriptsize $k+1$};
          \node[cp, label={[label distance=-0.05cm]south:$y_1$}] (7) at (2,0) {\scriptsize $0$};
          \node[cp] (8) at (2,1) {\scriptsize $0$};
          \node[cp, label={[label distance=-0.05cm]north:$x_1$}] (9) at (2,2) {\scriptsize $0$};
          \node[cpb] (10) at (3,0) {\scriptsize $k+1$};
          \node[cp] (11) at (3,2) {\scriptsize $0$};
          \node[cp, label={[label distance=-0.05cm]south:$z_1$}] (12) at (4,0) {\scriptsize $0$};
          \node[cpb] (13) at (4,1) {\scriptsize $k+1$};
          \node[cp, label={[label distance=-0.05cm]north:$w_1$}] (14) at (4,2) {\scriptsize $0$};
          
           \draw (a) -- (1) -- (b) -- (3) -- (6) -- (5) -- (4) -- (2) -- (a);
           
           \draw (2) -- (3);
           \draw (1) -- (5);
           \draw (6) -- (d) -- (c) -- (4);

           \draw (7)--(8)--(9)--(11)--(14)--(13)--(12)--(10)--(7);
           \draw (8)--(13);
           \draw (10)--(11);

            \draw (7)--(c);
            \draw (9)--(d);

            \draw (a) .. controls (-2.4,2) .. (-2.4,2.4);
            \draw (-2.4,2.4) .. controls (-2.4,2.8) .. (1.3,2.8); 
            \draw (1.3,2.8) .. controls (4.7,2.8) .. (4.7,1.4);
            \draw (4.7,1.4) .. controls (4.7,0) .. (12);

            \draw (b) .. controls (-2.4,0) .. (-2.4, -0.4);
            \draw (-2.4,-0.4) .. controls (-2.4,-0.8) .. (1.4,-0.8);
            \draw (1.4,-0.8) .. controls (5.2,-0.8) .. (5.2,0.6);
            \draw (5.2,0.6) .. controls (5.2,2) .. (14);
        \end{tikzpicture}
        \captionsetup{width=7.2cm}
            \caption{Snark $B^1_1$ with a [$k$]-IRDF with weight $5(k+1)$.}
        \end{subfigure}
        \hfill
        \begin{subfigure}[b]{.45\textwidth}
            \centering
            \begin{tikzpicture}
            [cp/.style={circle,fill=white,draw,minimum size=1em,inner sep=1pt]},
            cpb/.style={fill=white,draw,rounded corners=.09cm,minimum size=1em,inner sep=1pt]},
            scale=.9]

              \node[cpb, label={[label distance=-0.05cm]north:$a$}] (a) at (-2,2) {\scriptsize$k+1$};
              \node[cp] (1) at (-2,1) {\scriptsize $0$};
              \node[cp, label={[label distance=-0.05cm]south:$b$}] (b) at (-2,0) {\scriptsize $0$};
              \node[cp] (2) at (-1,2) {\scriptsize $0$};
              \node[cp] (3) at (-1,0) {\scriptsize $0$};
              \node[cp, label={[label distance=-0.05cm]north:$c$}] (c) at (0,2) {\scriptsize $0$};
              \node[cp] (4) at (0,1) {\scriptsize $0$};
              \node[cpb, label={[label distance=-0.05cm]south:$d$}] (d) at (0,0) {\scriptsize $k+1$};
              \node[cpb] (m) at (-1,1.5) {\scriptsize $k+1$};
              \node[cp] (n) at (-0.5,1) {\scriptsize $0$};

              \node[cpb,label={[label distance=-0.05cm]south:$y_1$}] (5) at (1,0) {\scriptsize $k+1$};
              \node[cp] (6) at (1,1) {\scriptsize $0$};
              \node[cp,label={[label distance=-0.05cm]north:$x_1$}] (7) at (1,2) {\scriptsize $0$};
              \node[cp] (8) at (2,0) {\scriptsize $0$};
              \node[cp] (9) at (2,2) {\scriptsize $0$};
              \node[cp, label={[label distance=-0.05cm]south:$z_1$}] (10) at (3,0) {\scriptsize $0$};
              \node[cp] (11) at (3,1) {\scriptsize $0$};
              \node[cpb, label={[label distance=-0.05cm]north:$w_1$}] (12) at (3,2) {\scriptsize $k+1$};
              
               \draw (a) -- (1) -- (b) -- (3) -- (d) -- (4) -- (c) -- (2) -- (a);
               
               \draw (2) -- (m) -- (3);
               \draw (m) -- (n) -- (4);
               \draw (1) -- (n);

               \draw (5)--(6)--(7)--(9)--(12)--(11)--(10)--(8)--(5);
               \draw (8)--(9);
               \draw (6)--(11);

               \draw (c) -- (5);
               \draw (d) -- (7);

               \draw (a) .. controls (-2.4,2) .. (-2.4,2.4);
               \draw (-2.4,2.4) .. controls (-2.4,2.8) .. (0.6,2.8);
               \draw (0.5,2.8) .. controls (3.7,2.8) .. (3.7,1.4);
               \draw (3.7,1.4) .. controls (3.7,0) .. (10);

               \draw (b) .. controls (-2.4,0) .. (-2.4,-.4);
               \draw (-2.4,-.4) .. controls (-2.4,-.8) .. (0.9,-0.8);
               \draw (.9,-.8) .. controls (4.2,-.8) .. (4.2,.8);
               \draw (4.2,.8) .. controls (4.2,2) .. (12);
            \end{tikzpicture}
            \captionsetup{width=7.2cm}
            \caption{Snark $B^2_1$ with a [$k$]-IRDF with weight $5(k+1)$.}
        \end{subfigure}
        \captionsetup{width=12cm}
        \caption{Independent [$k$]-Roman domination functions of snarks $B^1_1$ and $B^2_1$ with weight $5(k+1)$.}
        \label{fig:casoEspecialB1}
    \end{figure}

Next, we prove by strong induction on $i$ that every snark $B_i^t$, with $t \in \{1,2\}$ and $i \geq 2$, has a [$k$]-IRDF $f_i$ with the following properties: (i) $f_i(a)=k+1$, $f_i(b)=0$, $f_i(w_i)=k+1$ and $f_i(z_i)=0$; (ii) $\omega(f_i) = (k+1)(2i+2)+2k$ if $t = 1$, $i \geq 3$ and $i$ odd; or $\omega(f_i) = (k+1)(2i+3)$ otherwise. We call \textit{special} a [$k$]-IRDF $f_i$ of $B_i^t$ that satisfies the previous two properties. The induction is based on the recursive construction of the families $\mathfrak{B}^1$ and $\mathfrak{B}^2$.

For the base case, consider the snarks $B_i^t$ with $i \in \{2,3\}$ and $t \in \{1,2\}$. For $i=2$, Figures~\ref{fig:B12Rotulado} and~\ref{fig:B22Rotulado} exhibit the snarks $B_2^1$ and $B_2^2$, respectively, with their special [$k$]-IRDFs with weight $7(k+1) = (k+1)(2i+3)$. For $i=3$, the snark $B_3^1$ is illustrated in Figure~\ref{fig:B13Rotulado} with a special [$k$]-IRDF $f_3$ with weight $8(k+1)+2k$; and the snark $B_3^2$ is illustrated in Figure~\ref{fig:B23Rotulado} with a special [$k$]-IRDF $f_3$ with weight $9(k+1)$.

\begin{figure}[!htb]
        \centering
        \begin{subfigure}[b]{.7\textwidth}
            \centering
              \begin{tikzpicture}
	            [cp/.style={circle,fill=white,draw,minimum size=1em,inner sep=1pt]},
             cpb/.style={fill=white,draw,rounded corners=.09cm,minimum size=1em,inner sep=1pt]},
             scale=.8]
                  \node[cpb, label={[label distance=-0.05cm]north:$a$}] (a) at (-5,2) {\scriptsize $k+1$};
                  \node[cp] (1) at (-5,1) {\scriptsize $0$};
                  \node[cp, label={[label distance=-0.05cm]south:$b$}] (c) at (-5,0) {\scriptsize $0$};
                  \node[cp] (2) at (-4,2) {\scriptsize $0$};
                  \node[cp] (3) at (-4,0) {\scriptsize $0$};
                  \node[cp] (4) at (-3,2) {\scriptsize $0$};
                  \node[cp] (5) at (-3,1) {\scriptsize $0$};
                  \node[cpb] (6) at (-3,0) {\scriptsize $k+1$};
                  \node[cpb, label={[label distance=-0.05cm]north:$c$}] (b) at (-2,2) {\scriptsize $k+1$};
                  \node[cp, label={[label distance=-0.05cm]south:$d$}] (d) at (-2,0) {\scriptsize $0$};
                  
                   \draw (a) -- (1) -- (c) -- (3) -- (6) -- (5) -- (4) -- (2) -- (a);
                   
                   \draw (2) -- (3);
                   \draw (1) -- (5);
                   \draw (6) -- (d) -- (b) -- (4);
                   
                  \node[cpb, label={[label distance=-0.05cm]north:$x_1$}] (x1) at (-1,2) {\scriptsize $k+1$};
                  \node[cp] (7) at (-1,1) {\scriptsize $0$};
                  \node[cp, label={[label distance=-0.05cm]south:$y_1$}] (y1) at (-1,0) {\scriptsize $0$};
                  \node[cp] (8) at (0,2) {\scriptsize $0$};
                  \node[cp] (9) at (0,0) {\scriptsize $0$};
                  \node[cp, label={[label distance=-0.05cm]north:$w_1$}] (w1) at (1,2) {\scriptsize $0$};
                  \node[cp] (10) at (1,1) {\scriptsize $0$};
                  \node[cpb, label={[label distance=-0.05cm]south:$z_1$}] (z1) at (1,0) {\scriptsize $k+1$};
                                    
                   \draw (x1) -- (7) -- (y1) -- (9) -- (z1) -- (10) -- (w1) -- (8) -- (x1);
                   
                   \draw (7) -- (10);
                   \draw (8) -- (9);
                   
                  \node[cp, label={[label distance=-0.05cm]north:$x_2$}] (x2) at (2,2) {\scriptsize $0$};
                  \node[cp] (11) at (2,1) {\scriptsize $0$};
                  \node[cpb, label={[label distance=-0.05cm]south:$y_2$}] (y2) at (2,0) {\scriptsize $k+1$};
                  \node[cp] (12) at (3,2) {\scriptsize $0$};
                  \node[cp] (13) at (3,0) {\scriptsize $0$};
                  \node[cpb, label={[label distance=-0.05cm]north:$w_2$}] (w2) at (4,2) {\scriptsize $k+1$};
                  \node[cp] (14) at (4,1) {\scriptsize $0$};
                  \node[cp, label={[label distance=-0.05cm]south:$z_2$}] (z2) at (4,0) {\scriptsize $0$};
                  
                   \draw (x2) -- (11) -- (y2) -- (13) -- (z2) -- (14) -- (w2) -- (12) -- (x2);
                   
                   \draw (11) -- (14);
                   \draw (12) -- (13);
                       
                   %%aresta de ligação
                   \draw (b) -- (y1);
                   \draw (d) -- (x1);
                   \draw (w1) -- (y2);
                   \draw (z1) -- (x2);
                   
                   \draw (w2) .. controls (5.5,2) .. (5.4,-0.6);
                   \draw (5.4,-0.6) .. controls (5.4,-1) .. (-4,-0.9);
                   \draw (-4,-0.9) .. controls (-5.6,-0.9) .. (-5.6, -0.45);
                   \draw (-5.6, -.45) .. controls (-5.6, 0) .. (c);
                   
                   \draw (a) .. controls (-5.6,2) .. (-5.6,2.45);
                   \draw (-5.6,2.45) .. controls (-5.6,2.9) .. (2,2.9);
                   \draw (2,2.9) .. controls (4.8,2.9) .. (4.8,1);
                   \draw (4.8,1) .. controls (4.8,0) .. (z2);
                   
	            \end{tikzpicture}
        \captionsetup{width=10cm}
        \caption{\textit{Snark} $B^1_2$ with a special [$k$]-IRDF with weight $7(k+1)$.}
        \label{fig:B12Rotulado}
        \vspace{.3cm}
        \end{subfigure}
        
        \begin{subfigure}[b]{.8\textwidth}
            \centering
            \begin{tikzpicture}
	            [cp/.style={circle,fill=white,draw,minimum size=1em,inner sep=1pt},
             cpb/.style={fill=white,draw,rounded corners=.09cm,minimum size=1em,inner sep=1pt]},
             scale=.8] 
                  \node[cpb, label={[label distance=-0.05cm]north:$a$}] (a) at (-5,2) {\scriptsize $k+1$};
                  \node[cp] (1) at (-5,1) {\scriptsize $0$};
                  \node[cp, label={[label distance=-0.05cm]south:$b$}] (c) at (-5,0) {\scriptsize $0$};
                  \node[cp] (2) at (-4,2) {\scriptsize $0$};
                  \node[cp] (3) at (-4,0) {\scriptsize $0$};
                  \node[cpb] (4) at (-3,2) {\scriptsize $k$};
                  \node[cp] (5) at (-3,1) {\scriptsize $0$};
                  \node[cpb] (6) at (-3,0) {\scriptsize $k+1$};
                  \node[cp, label={[label distance=-0.05cm]north:$c$}] (b) at (-2,2) {\scriptsize $0$};
                  \node[cp, label={[label distance=-0.05cm]south:$d$}] (d) at (-2,0) {\scriptsize $0$};
                  
                   \draw (a) -- (1) -- (c) -- (3) -- (6) -- (5) -- (4) -- (2) -- (a);
                   
                   \draw (2) -- (3);
                   \draw (1) -- (5);
                   \draw (6) -- (d) -- (b) -- (4);

                  \node[cpb, label={[label distance=-0.05cm]north:$x_1$}] (x1) at (-1,2) {\scriptsize $k$};
                  \node[cp] (7) at (-1,1) {\scriptsize $0$};
                  \node[cpb, label={[label distance=-0.05cm]south:$y_1$}] (y1) at (-1,0) {\scriptsize $k+1$};
                  \node[cp] (8) at (0,2) {\scriptsize $0$};
                  \node[cp] (9) at (0,0) {\scriptsize $0$};
                  \node[cpb, label={[label distance=-0.05cm]north:$w_1$}] (w1) at (1,2) {\scriptsize $k+1$};
                  \node[cp] (10) at (1,1) {\scriptsize $0$};
                  \node[cp, label={[label distance=-0.05cm]south:$z_1$}] (z1) at (1,0) {\scriptsize $0$};

                   \draw (x1) -- (7) -- (y1) -- (9) -- (z1) -- (10) -- (w1) -- (8) -- (x1);
                   
                   \draw (7) -- (10);
                   \draw (8) -- (9);
                   
                  \node[cpb, label={[label distance=-0.05cm]north:$x_2$}] (x2) at (2,2) {\scriptsize $k+1$};
                  \node[cp] (11) at (2,1) {\scriptsize $0$};
                  \node[cp, label={[label distance=-0.05cm]south:$y_2$}] (y2) at (2,0) {\scriptsize $0$};
                  \node[cp] (12) at (3,2) {\scriptsize $0$};
                  \node[cp] (13) at (3,0) {\scriptsize $0$};
                  \node[cp, label={[label distance=-0.05cm]north:$w_2$}] (w2) at (4,2) {\scriptsize $0$};
                  \node[cp] (14) at (4,1) {\scriptsize $0$};
                  \node[cpb, label={[label distance=-0.05cm]south:$z_2$}] (z2) at (4,0) {\scriptsize $k+1$};

                   \draw (x2) -- (11) -- (y2) -- (13) -- (z2) -- (14) -- (w2) -- (12) -- (x2);
                   
                   \draw (11) -- (14);
                   \draw (12) -- (13);

                    \node[cp, label={[label distance=-0.05cm]north:$x_3$}] (x3) at (5,2) {\scriptsize $0$};
                  \node[cp] (15) at (5,1) {\scriptsize $0$};
                  \node[cpb, label={[label distance=-0.05cm]south:$y_3$}] (y3) at (5,0) {\scriptsize $k+1$};
                  \node[cp] (16) at (6,2) {\scriptsize $0$};
                  \node[cp] (17) at (6,0) {\scriptsize $0$};
                  \node[cpb, label={[label distance=-0.05cm]north:$w_3$}] (w3) at (7,2) {\scriptsize $k+1$};
                  \node[cp] (18) at (7,1) {\scriptsize $0$};
                  \node[cp, label={[label distance=-0.05cm]south:$z_3$}] (z3) at (7,0) {\scriptsize $0$};

                   \draw (x3) -- (15) -- (y3) -- (17) -- (z3) -- (18) -- (w3) -- (16) -- (x3);
                   
                   \draw (15) -- (18);
                   \draw (16) -- (17);
                   
                   %%aresta de ligação
                   \draw (b) -- (y1);
                   \draw (d) -- (x1);
                   \draw (w1) -- (y2);
                   \draw (z1) -- (x2);
                   \draw (w2) -- (y3);
                   \draw (z2) -- (x3);
                   
                   \draw (w3) .. controls (8.5,2) .. (8.4,-0.6);
                   \draw (8.4,-0.6) .. controls (8.4,-1) .. (-4,-0.9);
                   \draw (-4,-0.9) .. controls (-5.6,-0.9) .. (-5.6,-.45);
                   \draw (-5.6, -.45) .. controls (-5.6, 0) .. (c);
                   
                   \draw (a)  .. controls (-5.6,2) .. (-5.6,2.45);
                   \draw (-5.6,2.45) .. controls (-5.6,3) .. (1.2,3);
                   \draw (1.2,3) .. controls (7.8,3) .. (7.8,1);
                   \draw (7.8,1) .. controls (7.8,0) .. (z3);
                   
	            \end{tikzpicture}
                % \vspace{0.3cm}
                \captionsetup{width=10cm}
                \caption{\textit{Snark} $B^1_3$ with a special [$k$]-IRDF with weight $8(k+1)+2k$.}
                \label{fig:B13Rotulado}
        \end{subfigure}
   \captionsetup{width=12cm}
    \caption{Special independent [$k$]-Roman dominating functions for the snarks $B^1_2$ and $B^1_3$.}
    \label{fig:snarks8765}
    \end{figure}
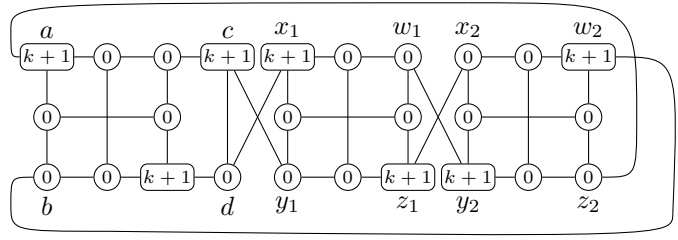
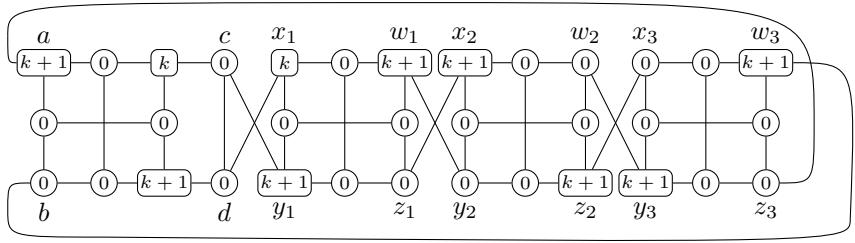

For the inductive step, consider a snark $B_i^t$ with $i \geq 4$ and $t \in \{1,2\}$. By the recursive construction of generalized Blanu\v{s}a snarks, we know that $B_i^t$ can be constructed from the link graph $LG_i$ and the snark $B_{i-2}^t$. Figure~\ref{fig:subgrafoHiRotulado1} shows the link graph $LG_i$ with a vertex labeling $\varphi \colon V(LG_i) \to \{0,k+1\}$ with weight $4(k+1)$. Also, by induction hypothesis, the snark $B_{i-2}^t$ has a special [$k$]-IRDF $f_{i-2}$ with weight $\omega(f_{i-2}) = (k+1)(2(i-2)+2)+2k$ when $t=1$ and $i\geq 3$, $i$ odd; or with weight $\omega(f_{i-2}) = (k+1)(2(i-2)+3)$ otherwise. Since $f_{i-2}$ is special, we also have that $f_{i-2}(a) = k+1$, $f_{i-2}(b) = 0$, $f_{i-2}(w_{i-2}) = k+1$,
$f_{i-2}(z_{i-2}) = 0$, for $a,b,w_{i-2},z_{i-2} \in V(B_{i-2}^t)$. Thus, we define a vertex labeling $f_i$ for $B_i^t$ as follows. For every vertex $v \in V(B_i^t)$, 
\[ f_i(v) =
  \begin{cases}
    f_{i-2}(v) & \quad \text{if } v \in V(B_{i-2}^t) \cap V(B_i^t);\\
    \varphi(v) & \quad \text{if } v \in V(LG_i) \cap V(B_i^t).
  \end{cases}
\]

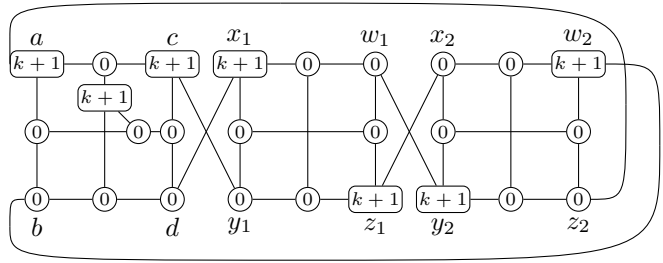
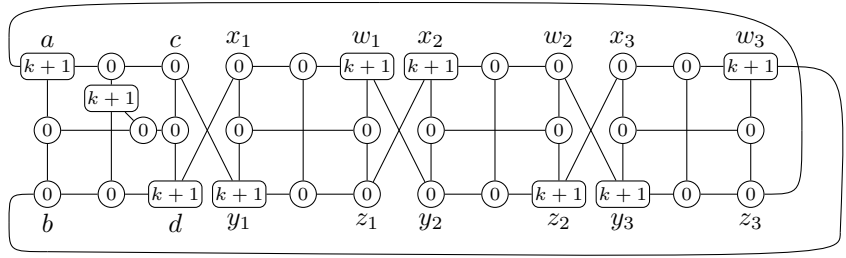
\begin{figure}[!htb]
\centering
\begin{subfigure}[b]{0.45\textwidth}
    \centering
    \begin{tikzpicture}
    [cp/.style={circle,fill=white,draw,minimum size=.9em,inner sep=1pt]},
    cpb/.style={fill=white,draw,rounded corners=.09cm,minimum size=1em,inner sep=1pt]},
    scale=.9]
      \node[cpb, label={[label distance=-0.05cm]north:$a$}] (a) at (-2,2) {\scriptsize $k+1$};
      \node[cp] (1) at (-2,1) {\scriptsize $0$};
      \node[cp, label={[label distance=-0.05cm]south:$b$}] (b) at (-2,0) {\scriptsize $0$};
      \node[cp] (2) at (-1,2) {\scriptsize $0$};
      \node[cp] (3) at (-1,0) {\scriptsize $0$};
      \node[cpb, label={[label distance=-0.05cm]north:$c$}] (c) at (0,2) {\scriptsize $k+1$};
      \node[cp] (4) at (0,1) {\scriptsize $0$};
      \node[cp, label={[label distance=-0.05cm]south:$d$}] (d) at (0,0) {\scriptsize $0$};
      \node[cpb] (m) at (-1,1.5) {\scriptsize $k+1$};
      \node[cp] (n) at (-0.5,1) {\scriptsize $0$};

      \node[cp,label={[label distance=-0.05cm]south:$y_1$}] (5) at (1,0) {\scriptsize $0$};
      \node[cp] (6) at (1,1) {\scriptsize $0$};
      \node[cpb,label={[label distance=-0.05cm]north:$x_1$}] (7) at (1,2) {\scriptsize $k+1$};
      \node[cp] (8) at (2,0) {\scriptsize $0$};
      \node[cp] (9) at (2,2) {\scriptsize $0$};
      \node[cpb, label={[label distance=-0.05cm]south:$z_1$}] (10) at (3,0) {\scriptsize $k+1$};
      \node[cp] (11) at (3,1) {\scriptsize $0$};
      \node[cp, label={[label distance=-0.05cm]north:$w_1$}] (12) at (3,2) {\scriptsize $0$};

       \node[cpb,label={[label distance=-0.05cm]south:$y_2$}] (13) at (4,0) {\scriptsize $k+1$};
      \node[cp] (14) at (4,1) {\scriptsize $0$};
      \node[cp,label={[label distance=-0.05cm]north:$x_2$}] (15) at (4,2) {\scriptsize $0$};
      \node[cp] (16) at (5,0) {\scriptsize $0$};
      \node[cp] (17) at (5,2) {\scriptsize $0$};
      \node[cp, label={[label distance=-0.05cm]south:$z_2$}] (18) at (6,0) {\scriptsize $0$};
      \node[cp] (19) at (6,1) {\scriptsize $0$};
      \node[cpb, label={[label distance=-0.05cm]north:$w_2$}] (20) at (6,2) {\scriptsize $k+1$};

      \draw (13)--(14)--(15)--(17)--(20)--(19)--(18)--(16)--(13);
      \draw (16)--(17);
      \draw (14)--(19);

      \draw (12)--(13);
      \draw (10)--(15);
      
       \draw (a) -- (1) -- (b) -- (3) -- (d) -- (4) -- (c) -- (2) -- (a);
       
       \draw (2) -- (m) -- (3);
       \draw (m) -- (n) -- (4);
       \draw (1) -- (n);

       \draw (5)--(6)--(7)--(9)--(12)--(11)--(10)--(8)--(5);
       \draw (8)--(9);
       \draw (6)--(11);

       \draw (c) -- (5);
       \draw (d) -- (7);

       \draw (a) .. controls (-2.4,2) .. (-2.4,2.4);
       \draw (-2.4,2.4) .. controls (-2.4,2.9) .. (1.6,2.9);
       \draw (1.6,2.9) .. controls (6.7,2.9) .. (6.7,1.4);
       \draw (6.7,1.4) .. controls (6.7,0) .. (18);

       \draw (b) .. controls (-2.4,0) .. (-2.4,-.4);
       \draw (-2.4,-.4) .. controls (-2.4,-.9) .. (1.9,-0.9);
       \draw (1.9,-.9) .. controls (7.2,-.9) .. (7.2,.8);
       \draw (7.2,.8) .. controls (7.2,2) .. (20);
    \end{tikzpicture}
    \captionsetup{width=7.2cm}
    \caption{Snark $B^2_2$ with a special [$k$]-IRDF with weight $7(k+1)$.}
    \label{fig:B22Rotulado}
    \vspace{.3cm}
\end{subfigure}
    
\begin{subfigure}[b]{.8\textwidth}
    \centering
    \begin{tikzpicture}
        [cp/.style={circle,fill=white,draw,minimum size=1em,inner sep=1pt},
        cpb/.style={fill=white,draw,rounded corners=.09cm,minimum size=1em,inner sep=1pt]},
        scale=.85] 
           \node[cpb, label={[label distance=-0.05cm]north:$a$}] (a) at (-4,2) {\scriptsize $k+1$};
          \node[cp] (1) at (-4,1) {\scriptsize $0$};
          \node[cp, label={[label distance=-0.05cm]south:$b$}] (b) at (-4,0) {\scriptsize $0$};
          \node[cp] (2) at (-3,2) {\scriptsize $0$};
          \node[cp] (3) at (-3,0) {\scriptsize $0$};
          \node[cp, label={[label distance=-0.05cm]north:$c$}] (c) at (-2,2) {\scriptsize $0$};
          \node[cp] (4) at (-2,1) {\scriptsize $0$};
          \node[cpb, label={[label distance=-0.05cm]south:$d$}] (d) at (-2,0) {\scriptsize $k+1$};
          \node[cpb] (m) at (-3,1.5) {\scriptsize $k+1$};
          \node[cp] (n) at (-2.5,1) {\scriptsize $0$};
          
           \draw (a) -- (1) -- (b) -- (3) -- (d) -- (4) -- (c) -- (2) -- (a);
       
           \draw (2) -- (m) -- (3);
           \draw (m) -- (n) -- (4);
           \draw (1) -- (n);

          \node[cp, label={[label distance=-0.05cm]north:$x_1$}] (x1) at (-1,2) {\scriptsize $0$};
          \node[cp] (7) at (-1,1) {\scriptsize $0$};
          \node[cpb, label={[label distance=-0.05cm]south:$y_1$}] (y1) at (-1,0) {\scriptsize $k+1$};
          \node[cp] (8) at (0,2) {\scriptsize $0$};
          \node[cp] (9) at (0,0) {\scriptsize $0$};
          \node[cpb, label={[label distance=-0.05cm]north:$w_1$}] (w1) at (1,2) {\scriptsize $k+1$};
          \node[cp] (10) at (1,1) {\scriptsize $0$};
          \node[cp, label={[label distance=-0.05cm]south:$z_1$}] (z1) at (1,0) {\scriptsize $0$};

           \draw (x1) -- (7) -- (y1) -- (9) -- (z1) -- (10) -- (w1) -- (8) -- (x1);
           
           \draw (7) -- (10);
           \draw (8) -- (9);
           
          \node[cpb, label={[label distance=-0.05cm]north:$x_2$}] (x2) at (2,2) {\scriptsize $k+1$};
          \node[cp] (11) at (2,1) {\scriptsize $0$};
          \node[cp, label={[label distance=-0.05cm]south:$y_2$}] (y2) at (2,0) {\scriptsize $0$};
          \node[cp] (12) at (3,2) {\scriptsize $0$};
          \node[cp] (13) at (3,0) {\scriptsize $0$};
          \node[cp, label={[label distance=-0.05cm]north:$w_2$}] (w2) at (4,2) {\scriptsize $0$};
          \node[cp] (14) at (4,1) {\scriptsize $0$};
          \node[cpb, label={[label distance=-0.05cm]south:$z_2$}] (z2) at (4,0) {\scriptsize $k+1$};

           \draw (x2) -- (11) -- (y2) -- (13) -- (z2) -- (14) -- (w2) -- (12) -- (x2);
           
           \draw (11) -- (14);
           \draw (12) -- (13);

            \node[cp, label={[label distance=-0.05cm]north:$x_3$}] (x3) at (5,2) {\scriptsize $0$};
          \node[cp] (15) at (5,1) {\scriptsize $0$};
          \node[cpb, label={[label distance=-0.05cm]south:$y_3$}] (y3) at (5,0) {\scriptsize $k+1$};
          \node[cp] (16) at (6,2) {\scriptsize $0$};
          \node[cp] (17) at (6,0) {\scriptsize $0$};
          \node[cpb, label={[label distance=-0.05cm]north:$w_3$}] (w3) at (7,2) {\scriptsize $k+1$};
          \node[cp] (18) at (7,1) {\scriptsize $0$};
          \node[cp, label={[label distance=-0.05cm]south:$z_3$}] (z3) at (7,0) {\scriptsize $0$};

           \draw (x3) -- (15) -- (y3) -- (17) -- (z3) -- (18) -- (w3) -- (16) -- (x3);
           
           \draw (15) -- (18);
           \draw (16) -- (17);
           
           %%aresta de ligação
           \draw (c) -- (y1);
           \draw (d) -- (x1);
           \draw (w1) -- (y2);
           \draw (z1) -- (x2);
           \draw (w2) -- (y3);
           \draw (z2) -- (x3);
           
           \draw (w3) .. controls (8.5,2) .. (8.4,-0.6);
           \draw (8.4,-0.6) .. controls (8.4,-1) .. (-4,-0.9);
           \draw (-4,-0.9) .. controls (-4.6,-0.9) .. (-4.6,-.45);
           \draw (-4.6, -.45) .. controls (-4.6, 0) .. (b);
           
           \draw (a)  .. controls (-4.6,2) .. (-4.6,2.45);
           \draw (-4.6,2.45) .. controls (-4.6,3) .. (1.2,3);
           \draw (1.2,3) .. controls (7.8,3) .. (7.8,1);
           \draw (7.8,1) .. controls (7.8,0) .. (z3);
        \end{tikzpicture}
    \captionsetup{width=10cm}
    \caption{Snark $B^2_3$ with a special [$k$]-IRDF with weight $9(k+1)$.}
    \label{fig:B23Rotulado}
\end{subfigure}
\centering
\captionsetup{width=12cm}
\caption{Special independent [$k$]-Roman dominating functions for $B^2_2$ and $B^2_3$.}
\label{fig:ouygdf65}
\end{figure}

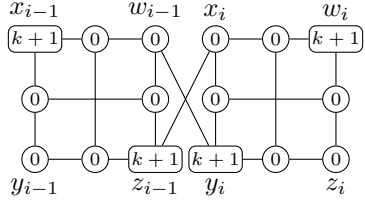
\begin{figure}[!htb]
\centering
\begin{tikzpicture}
[cp/.style={circle,fill=white,draw,minimum size=1em,inner sep=1pt]},
cpb/.style={fill=white,draw,rounded corners=.09cm,minimum size=1em,inner sep=1pt]},
scale=.8]
  \node[cpb, label={[label distance=-0.05cm]north:$x_{i-1}$}] (xj) at (-2.5,2) {\scriptsize $k+1$};
  \node[cp] (1) at (-2.5,1) {\scriptsize $0$};
  \node[cp, label={[label distance=-0.05cm]south:$y_{i-1}$}] (yj) at (-2.5,0) {\scriptsize $0$};
  \node[cp] (2) at (-1.5,2) {\scriptsize $0$};
  \node[cp] (3) at (-1.5,0) {\scriptsize $0$};
  \node[cp,label={[label distance=-0.05cm]north:$w_{i-1}$}] (wj) at (-.5,2) {\scriptsize $0$};
  \node[cp] (4) at (-.5,1) {\scriptsize $0$};
  \node[cpb,label={[label distance=-0.05cm]south:$z_{i-1}$}] (zj) at (-.5,0) {\scriptsize $k+1$};
  \node[cp, label={[label distance=-0.05cm]north:$x_i$}] (xj2) at (0.5,2) {\scriptsize $0$};
  \node[cp] (5) at (.5,1) {\scriptsize $0$};
  \node[cpb, label={[label distance=-0.05cm]south:$y_i$}] (yj2) at (0.5,0) {\scriptsize $k+1$};
  \node[cp] (6) at (1.5,2) {\scriptsize $0$};
  \node[cp] (7) at (1.5,0) {\scriptsize $0$};
  \node[cpb,label={[label distance=-0.05cm]north:$w_i$}] (wj2) at (2.5,2) {\scriptsize $k+1$};
  \node[cp] (8) at (2.5,1) {\scriptsize $0$};
  \node[cp,label={[label distance=-0.05cm]south:$z_i$}] (zj2) at (2.5,0) {\scriptsize $0$};

   \draw (xj) -- (1) -- (yj) -- (3) -- (zj) -- (4) -- (wj) -- (2) -- (xj);
   \draw (1) -- (4);
   \draw (2) -- (3);
   \draw (xj2) -- (5) -- (yj2) -- (7) -- (zj2) -- (8) -- (wj2) -- (6) -- (xj2);
   \draw (5) -- (8);
   \draw (6) -- (7);
   \draw (wj) -- (yj2);
   \draw (zj) -- (xj2);
\end{tikzpicture}
\caption{Link graph $LG_i$ with a vertex labeling $\varphi$. Note that the vertices $y_{i-1}$, $z_i$ and its neighbors have label $0$.}
\label{fig:subgrafoHiRotulado1}
\end{figure}

Next, we prove that $f_i$ is an [$k$]-IRDF of $B_i^t$. 
By induction hypothesis, the [$k$]-IRDF $f_{i-2}$ of $B_{i-2}^t$ is such that $f_{i-2}(a) = k+1$, $f_{i-2}(b) = 0$, $f_{i-2}(w_{i-2}) = k+1$, $f_{i-2}(z_{i-2}) = 0$. This implies that the labeling $f_i$ restricted to subgraph $B_{i-2}^t-E_{i-2}^{out} \subset B_i^t$ is almost a  [$k$]-IRDF of $B_{i-2}^t-E_{i-2}^{out}$ since $z_{i-2}$ and $b$ are the only vertices with label 0 in $B_{i-2}^t-E_{i-2}^{out}$ such that $f(N[z_{i-2}]) < |AN(z_{i-2})|+k$ and $f(N[b]) < |AN(b)|+k$. Also, by construction, the labeling $f_i$ restricted to subgraph $LG_i \subset B_i^t$ assigns label 0 to vertices $y_{i-1}$ and $z_i$ an these are the only vertices with label 0 in $LG_i$ that have $f(N[y_{i-1}]) < |AN(y_{i-1})|+k$ and $f(N[z_i]) < |AN(z_i)|+k$. 
Additionally, no two vertices with label $k+1$ in $LG_i$ are adjacent. 
Thus, $f_i$ restricted to $LG_i$ is almost a [$k$]-IRDF of $LG_i$ since $y_{i-1}$ and $z_i$ are the only vertices of $LG_i$ that have label 0 and $f(N[y_{i-1}])=f(N[z_i])=0$.
Therefore, in order to prove that $f_i$ is a [$k$]-IRDF of $B_i^t$, it suffices to show that the vertices $y_{i-1},z_{i-2},z_i,b$ have a neighbor in $B_i^t$ with label $k+1$. This comes down to analyzing the labels of the endpoints of the edges in the set $E_i^{in} = \{w_{i-2}y_{i-1}, z_{i-2}x_{i-1}, az_i, bw_i\}$ and verify if the vertices $w_{i-2}$, $x_{i-1}$, $a$, $w_i$ have label $k+1$. 
From the definition of $f_i$, we have that $f_i(w_{i-2})=f_{i-2}(w_{i-2})=k+1$, $f_i(x_{i-1})=\varphi(x_{i-1})=k+1$, $f_i(a)=f_{i-2}(a)=k+1$ and $f_i(w_i)=\varphi(w_i)=k+1$. Thus, the vertices $y_{i-1},z_{i-2},z_i,b$ (that have label 0) are adjacent to vertices with label $k+1$ in $B_i^t$, that is, the function $f_i$ is a [$k$]-IRDF of $B_i^t$. 

Now, we prove that $f_i$ is special. The weight of $f_i$ is given by the sum of the weights of the functions $f_{i-2}$ and $\varphi$. Thus, if $t=1$, $i \geq 5$ and $i$ odd, then $\omega(f_i) = \omega(f_{i-2})+\omega(\varphi) = (k+1)(2(i-2)+2)+2k + 4(k+1) = (k+1)(2i+2)+2k$; otherwise, we have that $\omega(f_i) = \omega(f_{i-2})+\omega(\varphi) = (k+1)(2(i-2)+3)+4(k+1) = (k+1)(2i+3)$. Note that $f_i(a)=k+1$, $f_i(b)=0$, $f_i(w_i)=k+1$ and $f_i(z_i)=0$ since these are the labels of each of these vertices in the subgraphs $B_{i-2}^t$ and $LG_i$. Therefore, $f_i$ is a special [$k$]-IRDF of $B_i^t$, and the result follows.
\end{proof}

By Theorem~\ref{thm:lowerBoundKgDelta}, $i_{[kR]}(B_i^t) \geq (k+1)(2i+2.5)$ for $k \geq 4$. However, for increasingly larger values of $k$, this lower bound moves away from the upper bounds given in Theorem~\ref{thm:upperiDRBlanusa}. Therefore, better lower bounds are needed. 
Theorems~\ref{thm:loweriDRBlanusa} and~\ref{thm:loweriDRBlanusa2} establish better lower bounds for the parameter $i_{[kR]}(B_i^t)$. In order to prove these results, we first present some additional definitions and auxiliary lemmas and theorems.

Given a graph $G$ and two disjoint sets $S_1 \subset V(G)$ and $S_2 \subset V(G)$, we denote by $E(S_1,S_2)$ the set of edges $uv \in E(G)$ such  that $u \in S_1$ and $v \in S_2$. Also, given $S \subseteq V(G)$, we denote by $N(S)$ the set of vertices $\{w \in V(G)\backslash S : uw \in E(G) \text{ and } u \in S\}$. We also define $N[S] = S \cup N(S)$.

\begin{lemma}
\label{lemma:auxBlanusa1}
Let $k \geq 2$ be an integer. 
If $G$ is a 3-regular graph with $n$ vertices and $f=(V_0,V_k,V_{k+1})$ is an $i_{[kR]}$-function of $G$, then $\displaystyle |V_k|\leq \frac{8i_{[kR]}(G)-2(k+1)n}{3k-5}$ and $\displaystyle |V_{k+1}|\geq \frac{2kn-5i_{[kR]}(G)}{3k-5}$.
\end{lemma}

\begin{proof}
Let $G$ be a 3-regular graph with $n$ vertices and $f=(V_0,V_k,V_{k+1})$ be an $i_{[kR]}$-function of $G$. Thus, $i_{[kR]}(G) = \omega(f) = k|V_k|+(k+1)|V_{k+1}|$. This fact implies that
\begin{equation}\label{eq:x20}
|V_{k+1}| = \frac{i_{[kR]}(G)-k|V_k|}{k+1} \quad \text{ and } \quad |V_{k}| = \frac{i_{[kR]}(G)-(k+1)|V_{k+1}|}{k}.
\end{equation}

Since $k\geq 2$, each vertex $v\in V(G)$ with $f(v)=0$ has at least one neighbor with label $k+1$ or at least two neighbors with label $k$. 
Let $S = V_0\cap N(V_{k+1})$ and $T=V_0\backslash S$. Since $G$ is 3-regular, each vertex in $V_{k+1}$ is adjacent to at most 3 vertices in $S$. Thus, $|S| \leq 3|V_{k+1}|$. Similarly, since each vertex in $V_k$ is adjacent to at most 3 vertices in $T$ and since each vertex in $T$ has at least two neighbors in $V_k$, we obtain that $2|T| \leq |E(V_k,T)| \leq 3|V_k|$, which imples that $|T| \leq \frac{3|V_k|}{2}$. Therefore, $|V_0| = |S|+|T| \leq 3|V_{k+1}|+\frac{3|V_k|}{2}$.

From the definition of [$k$]-IRDF, it follows that $n = |V_0|+|V_k|+|V_{k+1}|$. Hence, 
$n = |V_0|+|V_k|+|V_{k+1}| \leq 3|V_{k+1}|+\frac{3|V_k|}{2}+|V_k|+|V_{k+1}| = 4|V_{k+1}|+\frac{5|V_k|}{2}$, that is,
\begin{equation}\label{eq:x21}
    n \leq 4|V_{k+1}|+\frac{5|V_k|}{2}
\end{equation}
From Equation~\eqref{eq:x20} and Inequality~\eqref{eq:x21}, we have that $n \leq 4\cdot \frac{i_{[kR]}(G)-k|V_k|}{k+1} + \frac{5|V_k|}{2} = \frac{8i_{[kR]}(G)-(3k-5)|V_k|}{2(k+1)}$. From this last inequality, we conclude that $\displaystyle |V_k| \leq \frac{8i_{[kR]}(G)-2(k+1)n}{3k-5}$. Also, from Equation~\eqref{eq:x20} and Inequality~\eqref{eq:x21}, we have that $n \leq 4|V_{k+1}| + \frac{5i_{[kR]}(G)-5(k+1)|V_{k+1}|}{2k} = \frac{8k|V_{k+1}| + 5i_{[kR]}(G)-5(k+1)|V_{k+1}|}{2k} = \frac{5i_{[kR]}(G)+(3k-5)|V_{k+1}|}{2k}$. From this last inequality, we conclude that $\displaystyle |V_{k+1}| \geq \frac{2kn - 5i_{[kR]}(G)}{3k-5}$.
\end{proof}

\begin{lemma}
\label{lemma:relationGammaiRoman1}
Let $G$ be a graph and $k \geq 1$ be an integer. For any $i_{[kR]}$-function $f = (V_0,V_k,V_{k+1})$ of $G$, we have that $|V_{k+1}| \leq i_{[kR]}(G)-k\cdot i(G)$ and $|V_k|\geq (k+1)i(G)-i_{[kR]}(G)$.
\end{lemma}

\begin{proof}
Let $G$ be a graph with an  $i_{[kR]}$-function $f = (V_0,V_k,V_{k+1})$. Since $V_k \cup V_{k+1}$ is an independent dominating set of $G$, we have $i(G) \leq |V_k|+|V_{k+1}|$. Hence, $k\cdot i(G)\leq k|V_k|+k|V_{k+1}| = k|V_k|+(k+1)|V_{k+1}|-|V_{k+1}| = i_{[kR]}(G)-|V_{k+1}|$. This implies that $|V_{k+1}| \leq i_{[kR]}(G)-k\cdot i(G)$. In addition, $(k+1)i(G)\leq (k+1)|V_k|+(k+1)|V_{k+1}|=i_{[kR]}(G)+|V_k|$. This implies that $|V_k|\geq (k+1) i(G)-i_{[kR]}(G)$, and the result follows.
\end{proof}

The next result is used in our proofs and determines the domination number and independent domination number for generalized Blanu\v{s}a snarks.

\begin{theorem}[A.~Pereira~\cite{Pereira2020}]
\label{thm:PereiraGammaBlanusa}
Let $B^t_i$ be a generalized Blanu\v{s}a snark with $t \in \{1,2\}$ and $i\geq 1$. Then,
\[ i(B_i^t) = \gamma(B_i^t) =
  \begin{cases}
    2i+4   & \quad \text{if } t = 1 \text{ and } i \geq 3 \text{ with }  i \text{ odd};\\
    2i+3  & \quad \text{otherwise}.
  \end{cases}
\]
\end{theorem}

\begin{theorem}
\label{thm:loweriDRBlanusa}
Let $k\geq 2$ be an integer. Let $B_i^t$ be  a generalized Blanu\v{s}a snark such that $t=1$ and $i \geq 3$ with $i$ odd. Then,
\[
i_{[kR]}(B_i^t) \geq (k+1)(2i+2)+2k-2.
\]
\end{theorem}

\begin{proof}
By the definition of $B_i^t$, we have that $|V(B_i^t)|=8i+10$. Define $n=8i+10$. Let $f=(V_0,V_k,V_{k+1})$ be an $i_{[kR]}$-function of $B_i^t$.  
For the purpose of contradiction, suppose that $i_{[kR]}(B_i^t) \leq (k+1)(2i+2)+2k-3$. Next, we find a lower bound for $|V_k|$. By Theorem~\ref{thm:PereiraGammaBlanusa} and Theorem~\ref{lemma:relationGammaiRoman1}, $|V_k|\geq (k+1)i(G)-i_{[kR]}(G) \geq (k+1)(2i+4)-[(k+1)(2i+2)+2k-3] = 5$. Thus, $|V_k|\geq 5$. Next, we find an upper bound for $|V_k|$. By Lemma~\ref{lemma:auxBlanusa1}, $|V_k| \leq \frac{8i_{[kR]}(B_i^t)-2(k+1)n}{3k-5} \leq \frac{8((k+1)(2i+2)+2k-3)-2(k+1)(8i+10)}{3k-5} = \frac{12k-28}{3k-5} < 4$ for all $k \geq 2$. That is, $|V_k| < 4$. However, these facts imply that $5 \leq |V_k| < 4$, which is a contradiction.
\end{proof}

\begin{theorem}
\label{thm:loweriDRBlanusa2}
Let $k\geq 4$ be an integer. Let $B_i^t$ be  a generalized Blanu\v{s}a snark such that $t=2$, or $t=1$ with $i=1$, or $t=1$ with $i$ even. Then,
\[
i_{[kR]}(B_i^t) = (k+1)(2i+3).
\]
\end{theorem}

\begin{proof}
Let $k\geq 4$ be an integer. By Theorem~\ref{thm:upperiDRBlanusa}, $i_{[kR]}(B_i^t) \leq (k+1)(2i+3)$. So, in order to conclude the proof, it suffices to prove that $i_{[kR]}(B_i^t) \geq (k+1)(2i+3)$. By the definition of $B_i^t$, we have that $|V(B_i^t)|=8i+10$. Define $n=8i+10$. Let $f=(V_0,V_k,V_{k+1})$ be an $i_{[kR]}$-function of $B_i^t$. For the purpose of contradiction, suppose that $i_{[kR]}(B_i^t) \leq (k+1)(2i+3)-1$.

By Lemma~\ref{lemma:auxBlanusa1}, $|V_k| \leq \frac{8i_{[kR]}(B_i^k)-2(k+1)n}{3k-5} \leq \frac{8((k+1)(2i+3)-1)-2(k+1)(8i+10)}{3k-5} =
\frac{(k+1)[8(2i+3)-2(8i+10)]-8}{3k-5} =
\frac{4k-4}{3k-5}$. That is, $|V_k| \leq \frac{4k-4}{3k-5}$. For $k\geq 4$, we have that $\frac{4k-4}{3k-5} < 2$. This implies that $|V_k| < 2$ for all $k\geq 4$. On the other hand, by Lemma~\ref{lemma:relationGammaiRoman1} and Theorem~\ref{thm:PereiraGammaBlanusa}, $|V_k| \geq (k+1)i(B_i^t)-i_{[kR]}(B_i^t) \geq (k+1)(2i+3)-(k+1)(2i+3)+1 = 1$. These facts imply that $|V_k|=1$.

By the definition of [$k$]-RDF, $i_{[kR]}(B_i^t)=k|V_k|+(k+1)|V_{k+1}|=k+(k+1)|V_{k+1}|$. Moreover, since $i(B_i^t)=2i+3$, we have that $2i+3 = i(B_i^t) \leq |V_k|+|V_{k+1}|=1+|V_{k+1}|$, which implies that $|V_{k+1}| \geq 2i+2$. Hence, $i_{[kR]}(B_i^t)=(k+1)|V_{k+1}|+k \geq (k+1)(2i+2)+k$. From these facts, we have that $(k+1)(2i+2)+k \leq i_{[kR]}(B_i^t) \leq (k+1)(2i+3)-1$. However, since $(k+1)(2i+2)+k = (k+1)(2i+3)-1$, we obtain that $i_{[kR]}(B_i^t) = (k+1)(2i+2) + k$. Since $i_{[kR]}(B_i^t) = (k+1)(2i+2) + k$ and $|V_k|=1$ we obtain that $|V_{k+1}|=2i+2$. 

Since $B_i^t$ is 3-regular, each vertex in $V_{k+1}$ dominates at most 3 vertices in $V_0$. Thus, $|N(V_{k+1})| \leq 3|V_{k+1}|$. This implies that $|N[V_{k+1}]| = |V_{k+1}| + |N(V_{k+1})| \leq (2i+2) + 3(2i+2) = 8i+8$. In other words, there are at most $8i+8$ vertices that are either in $V_{k+1}$ or are dominated by vertices in $V_{k+1}$. Since $|V(B_i^t)| = 8i+10$, it remains $|V(B_i^t)|-(8i+8) = 2$ vertices in the set $V_0 \cup V_k$ that are not dominated by vertices with label $k+1$. One of these vertices belong to the set $V_k$, since $|V_k|=1$, and the other vertex, say $w$, belongs to the set $V_0$. Since $f(w)=0$ and $w$ has no neighbor in the set $V_{k+1}$, we conclude that $f(N[w]) < k + |AN(w)|$, which is a contradiction. 
\end{proof}

\begin{corollary}
\label{cor:indepRomanBlanusa}
Let $k\geq 4$ be an integer. If $B_i^t$ is a generalized Blanu\v{s}a snark, with
$t=2$, or $t=1$ with $i=1$, or $t=1$ with $i$ even, then $B_i^t$ is an independent [$k$]-Roman graph.
\end{corollary}

\begin{proof}
By Theorem~\ref{thm:loweriDRBlanusa2} and Theorem~\ref{thm:PereiraGammaBlanusa}, we have that $i_{[kR]}(B_i^t) = (k+1)(2i+3) = (k+1)i(B_i^t)$. Therefore, $B_i^t$ is an independent [$k$]-Roman graph.
\end{proof}

% ----------------------------
% SECTION
% ----------------------------
\section{The infinite family of Loupekine Snarks}
\label{sec:loupekine}

Around 1975, F.~Loupekine proposed a method of construction of infinite families of snarks using subgraphs of other known snarks. Loupekine's method was originally presented by Isaacs~\cite{Isaacs1976} in 1976. In this section, we consider two subfamilies of Loupekine snarks, called $LP_1$-snarks and $LP_0$-snarks, which are both obtained from fixed subgraphs called basic blocks. A \textit{basic block} $B_i$ is illustrated in Figure~\ref{fig:blocoBiLP}. Note that $B_i$ has five different degree-2  vertices, namely $r_i,s_i,t_i,u_i,v_i$, called \textit{border vertices}. The construction of the two families is described in the next paragraphs. 

Let $\ell \geq 3$ be an odd integer. An $\ell$-$LP_1$-snark $G_L$ is constructed from $\ell$ basic blocks $B_0, B_1,\ldots, B_{\ell-1}$. For each $i \in \{0,\ldots,\ell-1\}$, we connect the border vertices $s_i$ and $v_i$ of block $B_i$ to the border vertices $r_{i+1}$ and $u_{i+1}$ of block $B_{i+1}$ (indexes taken modulo $\ell$) with a pair of edges from the set $E_{i,i+1}$ that comprises either a pair of \textit{laminar edges} $\{s_ir_{i+1},v_iu_{i+1}\}$ or a pair of \textit{intersecting edges} $\{s_iu_{i+1},v_ir_{i+1}\}$, but not both. These edges connecting two consecutive basic blocks are called \textit{plug-edges}.

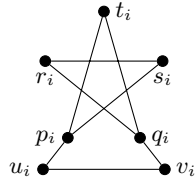
\begin{figure}[!htb]
    \centering
    \begin{tikzpicture}
           [cp/.style={circle,fill=black,draw,minimum size=.4em,inner sep=1pt]},scale=.6]

           \node[cp, label={[label distance=-0.05cm]west:\small $p_i$}] (1) at (-.8,.5) {};
           \node[cp, label={[label distance=-0.05cm]east:\small $q_i$}] (2) at (.8,.5) {};
           \node[cp, label={[label distance=-0.05cm]south:\small $r_i$}] (3) at (-1.3,2.2) {};
           \node[cp, label={[label distance=-0.05cm]south:\small $s_i$}] (4) at (1.3,2.2) {};
           \node[cp,, label={[label distance=-0.05cm]east:\small $t_i$}] (5) at (0,3.3) {};

           \draw (1)--(5)--(2)--(3)--(4)--(1);

           \node[cp, label={[label distance=-0.05cm]west:\small $u_i$}] (6) at (-1.35,-.2) {};
           \node[cp, label={[label distance=-0.05cm]east:\small $v_i$}] (7) at (1.35,-.2) {};

           \draw (6)--(7);

           \draw (6)--(1);
           \draw (7)--(2);
       \end{tikzpicture}
    \caption{Basic Block $B_i$.}
    \label{fig:blocoBiLP}
\end{figure}

Note that, after the addition of the plug-edges, the border vertices $t_i$, with $0\leq i \leq \ell-1$, still have degree 2. Thus, in the next step of the construction, any three distinct border vertices $t_i,t_j,t_s$, all of them with degree two, are linked to a new vertex $z_{i,j,s}$, called \emph{link-vertex}, by adding vertex $z_{i,j,s}$ and three new edges $t_iz_{i,j,s}$, $t_jz_{i,j,s}$ and $t_s z_{i,j,s}$ to $G_L$. The previous operation can be done an odd number $q$ of times, with $1\leq q \leq \lfloor \ell/3 \rfloor$. Since $\ell$ is odd, an even number $\ell-3q$ of border vertices with degree two remain. If $\ell-3q>0$, the remaining border vertices with degree two are paired up and each pair $t_i$ and $t_j$ is linked by a new edge $t_it_j$, called a \emph{repairing edge}, thus concluding the construction of an $\ell$-$LP_1$-snark $G_L$. Figure~\ref{fig:LP15} shows a $5$-$LP_1$-snark.

\begin{figure}[!htb]
\centering
\begin{tikzpicture}
   [cp/.style={circle,fill=black,draw,minimum size=.4em,inner sep=1pt]},scale=.6]

   \node[cp, label={[label distance=-0.05cm]west:\small $p_0$}] (1) at (-.8,.5) {};
   \node[cp, label={[label distance=-0.05cm]east:\small $q_0$}] (2) at (.8,.5) {};
   \node[cp, label={[label distance=-0.05cm]north:\small $r_0$}] (3) at (-1.3,2.2) {};
   \node[cp, label={[label distance=-0.05cm]north:\small $s_0$}] (4) at (1.3,2.2) {};
   \node[cp,, label={[label distance=-0.05cm]north:\small $t_0$}] (5) at (0,3.3) {};

   \draw (1)--(5)--(2)--(3)--(4)--(1);

   \node[cp, label={[label distance=-0.05cm]south:\small $u_0$}] (6) at (-1.35,-.2) {};
   \node[cp, label={[label distance=-0.05cm]south:\small $v_0$}] (7) at (1.35,-.2) {};

   \draw (6)--(7);

   \draw (6)--(1);
   \draw (7)--(2);

   \node[cp, label={[label distance=-0.05cm]west:\small $p_1$}] (8) at (3.2,.5) {};
   \node[cp, label={[label distance=-0.05cm]east:\small $q_1$}] (9) at (4.8,.5) {};
   \node[cp, label={[label distance=-0.05cm]north:\small $r_1$}] (10) at (2.7,2.2) {};
   \node[cp, label={[label distance=-0.05cm]north:\small $s_1$}] (11) at (5.3,2.2) {};
   \node[cp,, label={[label distance=-0.05cm]west:\small $t_1$}] (12) at (4,3.3) {};

   \draw (8)--(12)--(9)--(10)--(11)--(8);

   \node[cp, label={[label distance=-0.05cm]south:\small $u_1$}] (13) at (2.65,-.2) {};
   \node[cp, label={[label distance=-0.05cm]south:\small $v_1$}] (14) at (5.35,-.2) {};

   \draw (13)--(14);

   \draw (13)--(8);
   \draw (14)--(9);

   \node[cp, label={[label distance=-0.05cm]west:\small $p_2$}] (15) at (7.2,.5) {};
   \node[cp, label={[label distance=-0.05cm]east:\small $q_2$}] (16) at (8.8,.5) {};
   \node[cp, label={[label distance=-0.05cm]north:\small $r_2$}] (17) at (6.7,2.2) {};
   \node[cp, label={[label distance=-0.05cm]north:\small $s_2$}] (18) at (9.3,2.2) {};
   \node[cp,, label={[label distance=-0.05cm]east:\small $t_2$}] (19) at (8,3.3) {};

   \draw (15)--(19)--(16)--(17)--(18)--(15);

   \node[cp, label={[label distance=-0.05cm]south:\small $u_2$}] (20) at (6.65,-.2) {};
   \node[cp, label={[label distance=-0.05cm]south:\small $v_2$}] (21) at (9.35,-.2) {};

   \draw (20)--(21);

   \draw (20)--(15);
   \draw (21)--(16);

   \node[cp, label={[label distance=-0.05cm]west:\small $p_3$}] (22) at (11.2,.5) {};
   \node[cp, label={[label distance=-0.05cm]east:\small $q_3$}] (23) at (12.8,.5) {};
   \node[cp, label={[label distance=-0.05cm]north:\small $r_3$}] (24) at (10.7,2.2) {};
   \node[cp, label={[label distance=-0.05cm]north:\small $s_3$}] (25) at (13.3,2.2) {};
   \node[cp,, label={[label distance=-0.05cm]east:\small $t_3$}] (26) at (12,3.3) {};

   \draw (22)--(26)--(23)--(24)--(25)--(22);

   \node[cp, label={[label distance=-0.05cm]south:\small $u_3$}] (27) at (10.65,-.2) {};
   \node[cp, label={[label distance=-0.05cm]south:\small $v_3$}] (28) at (13.35,-.2) {};

   \draw (27)--(28);

   \draw (27)--(22);
   \draw (28)--(23);

   \node[cp, label={[label distance=-0.05cm]west:\small $p_4$}] (29) at (15.2,.5) {};
   \node[cp, label={[label distance=-0.05cm]east:\small $q_4$}] (30) at (16.8,.5) {};
   \node[cp, label={[label distance=-0.05cm]north:\small $r_4$}] (31) at (14.7,2.2) {};
   \node[cp, label={[label distance=-0.05cm]north:\small $s_4$}] (32) at (17.3,2.2) {};
   \node[cp,, label={[label distance=-0.05cm]north:\small $t_4$}] (33) at (16,3.3) {};

   \draw (29)--(33)--(30)--(31)--(32)--(29);

   \node[cp, label={[label distance=-0.05cm]south:\small $u_4$}] (34) at (14.65,-.2) {};
   \node[cp, label={[label distance=-0.05cm]south:\small $v_4$}] (35) at (17.35,-.2) {};

   \draw (34)--(35);

   \draw (34)--(29);
   \draw (35)--(30);

   \draw (7)--(13);
   \draw (4)--(10);

   \draw (14)--(17);
   \draw (11)--(20);

   \draw (21)--(27);
   \draw (18)--(24);

   \draw (28)--(31);
   \draw (25)--(34);

   \draw (6) .. controls (-2,-.2) .. (-2,-.7);
   \draw (-2,-.7) .. controls (-2,-1.2) .. (8,-1.2);
   \draw (8,-1.2) .. controls (17.75,-1.2) .. (17.75,-.7);
   \draw (17.75,-.7) .. controls (17.75, -.2) .. (35);

   \draw (3) .. controls (-2.4,2.2) .. (-2.4, 0);
   \draw (-2.4,0) .. controls (-2.4, -1.9) .. (8,-1.9);
   \draw (8, -1.9) .. controls (18.05,-1.9) .. (18.05,0);
   \draw (18.05,0) .. controls (18.05,2.2) .. (32);

   \node[cp, label={[label distance=-0.05cm]north:\small $z_{0,2,4}$}] (z0) at (8,4.5) {};

   \draw (5) -- (z0) -- (19);
   \draw (z0) --  (33);

    \draw (12) .. controls (8,4) .. (26);
   
\end{tikzpicture}
\captionsetup{width=10cm}
\caption{An $LP_1$-snark with 5 basic blocks and one link-vertex.}
\label{fig:LP15}
\end{figure}
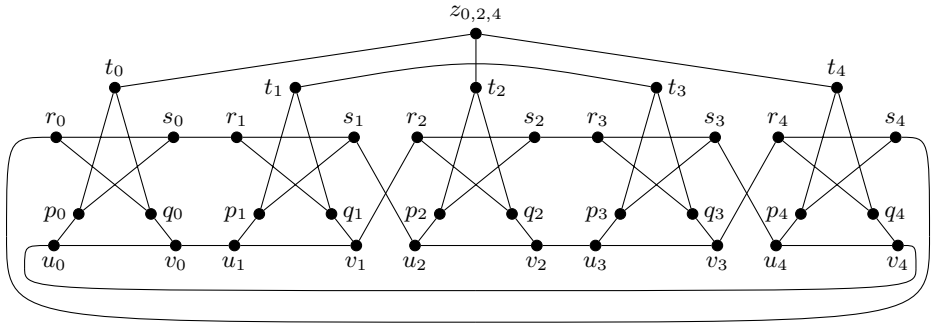

In the previous construction, if each link-vertex of $G_L$ is connected to three consecutive basic blocks $B_i,B_{i+1},B_{i+2}$, and if all repairing edges also connect consecutive basic blocks (are of the form $t_it_{i+1}$), then $G_L$ is said to be an \emph{$\ell$-$LP_0$-snark}. Figure~\ref{fig:LP05} shows a $5$-$LP_0$-snark. Every $LP_0$-snark is also an $LP_1$-snark but the converse is not true.

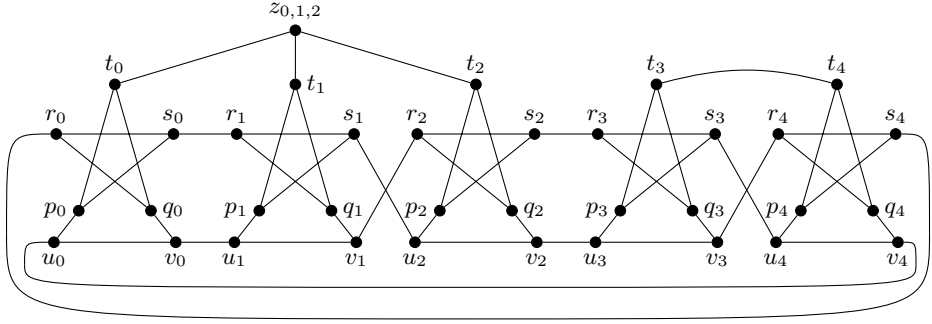
\begin{figure}[!htb]
    \centering
        \begin{tikzpicture}
           [cp/.style={circle,fill=black,draw,minimum size=.4em,inner sep=1pt]},scale=.6]

           \node[cp, label={[label distance=-0.05cm]west:\small $p_0$}] (1) at (-.8,.5) {};
           \node[cp, label={[label distance=-0.05cm]east:\small $q_0$}] (2) at (.8,.5) {};
           \node[cp, label={[label distance=-0.05cm]north:\small $r_0$}] (3) at (-1.3,2.2) {};
           \node[cp, label={[label distance=-0.05cm]north:\small $s_0$}] (4) at (1.3,2.2) {};
           \node[cp,, label={[label distance=-0.05cm]north:\small $t_0$}] (5) at (0,3.3) {};

           \draw (1)--(5)--(2)--(3)--(4)--(1);

           \node[cp, label={[label distance=-0.05cm]south:\small $u_0$}] (6) at (-1.35,-.2) {};
           \node[cp, label={[label distance=-0.05cm]south:\small $v_0$}] (7) at (1.35,-.2) {};

           \draw (6)--(7);

           \draw (6)--(1);
           \draw (7)--(2);

           \node[cp, label={[label distance=-0.05cm]west:\small $p_1$}] (8) at (3.2,.5) {};
           \node[cp, label={[label distance=-0.05cm]east:\small $q_1$}] (9) at (4.8,.5) {};
           \node[cp, label={[label distance=-0.05cm]north:\small $r_1$}] (10) at (2.7,2.2) {};
           \node[cp, label={[label distance=-0.05cm]north:\small $s_1$}] (11) at (5.3,2.2) {};
           \node[cp,, label={[label distance=-0.05cm]east:\small $t_1$}] (12) at (4,3.3) {};

           \draw (8)--(12)--(9)--(10)--(11)--(8);

           \node[cp, label={[label distance=-0.05cm]south:\small $u_1$}] (13) at (2.65,-.2) {};
           \node[cp, label={[label distance=-0.05cm]south:\small $v_1$}] (14) at (5.35,-.2) {};

           \draw (13)--(14);

           \draw (13)--(8);
           \draw (14)--(9);

           \node[cp, label={[label distance=-0.05cm]west:\small $p_2$}] (15) at (7.2,.5) {};
           \node[cp, label={[label distance=-0.05cm]east:\small $q_2$}] (16) at (8.8,.5) {};
           \node[cp, label={[label distance=-0.05cm]north:\small $r_2$}] (17) at (6.7,2.2) {};
           \node[cp, label={[label distance=-0.05cm]north:\small $s_2$}] (18) at (9.3,2.2) {};
           \node[cp,, label={[label distance=-0.05cm]north:\small $t_2$}] (19) at (8,3.3) {};

           \draw (15)--(19)--(16)--(17)--(18)--(15);

           \node[cp, label={[label distance=-0.05cm]south:\small $u_2$}] (20) at (6.65,-.2) {};
           \node[cp, label={[label distance=-0.05cm]south:\small $v_2$}] (21) at (9.35,-.2) {};

           \draw (20)--(21);

           \draw (20)--(15);
           \draw (21)--(16);

           \node[cp, label={[label distance=-0.05cm]west:\small $p_3$}] (22) at (11.2,.5) {};
           \node[cp, label={[label distance=-0.05cm]east:\small $q_3$}] (23) at (12.8,.5) {};
           \node[cp, label={[label distance=-0.05cm]north:\small $r_3$}] (24) at (10.7,2.2) {};
           \node[cp, label={[label distance=-0.05cm]north:\small $s_3$}] (25) at (13.3,2.2) {};
           \node[cp,, label={[label distance=-0.05cm]north:\small $t_3$}] (26) at (12,3.3) {};

           \draw (22)--(26)--(23)--(24)--(25)--(22);

           \node[cp, label={[label distance=-0.05cm]south:\small $u_3$}] (27) at (10.65,-.2) {};
           \node[cp, label={[label distance=-0.05cm]south:\small $v_3$}] (28) at (13.35,-.2) {};

           \draw (27)--(28);

           \draw (27)--(22);
           \draw (28)--(23);

           \node[cp, label={[label distance=-0.05cm]west:\small $p_4$}] (29) at (15.2,.5) {};
           \node[cp, label={[label distance=-0.05cm]east:\small $q_4$}] (30) at (16.8,.5) {};
           \node[cp, label={[label distance=-0.05cm]north:\small $r_4$}] (31) at (14.7,2.2) {};
           \node[cp, label={[label distance=-0.05cm]north:\small $s_4$}] (32) at (17.3,2.2) {};
           \node[cp,, label={[label distance=-0.05cm]north:\small $t_4$}] (33) at (16,3.3) {};

           \draw (29)--(33)--(30)--(31)--(32)--(29);

           \node[cp, label={[label distance=-0.05cm]south:\small $u_4$}] (34) at (14.65,-.2) {};
           \node[cp, label={[label distance=-0.05cm]south:\small $v_4$}] (35) at (17.35,-.2) {};

           \draw (34)--(35);

           \draw (34)--(29);
           \draw (35)--(30);

           \draw (7)--(13);
           \draw (4)--(10);

           \draw (14)--(17);
           \draw (11)--(20);

           \draw (21)--(27);
           \draw (18)--(24);

           \draw (28)--(31);
           \draw (25)--(34);

           \draw (6) .. controls (-2,-.2) .. (-2,-.7);
           \draw (-2,-.7) .. controls (-2,-1.2) .. (8,-1.2);
           \draw (8,-1.2) .. controls (17.75,-1.2) .. (17.75,-.7);
           \draw (17.75,-.7) .. controls (17.75, -.2) .. (35);

           \draw (3) .. controls (-2.4,2.2) .. (-2.4, 0);
           \draw (-2.4,0) .. controls (-2.4, -1.9) .. (8,-1.9);
           \draw (8, -1.9) .. controls (18.05,-1.9) .. (18.05,0);
           \draw (18.05,0) .. controls (18.05,2.2) .. (32);

           \node[cp, label={[label distance=-0.05cm]north:\small $z_{0,1,2}$}] (z0) at (4,4.5) {};

           \draw (5) -- (z0) -- (12);
           \draw (z0) -- (19);

           \draw (26) edge [bend left=15] (33);
       \end{tikzpicture}
\captionsetup{width=12cm}
\caption{An $LP_0$ with 5 basic blocks and one link-vertex.}
\label{fig:LP05}
\end{figure}

Theorem~\ref{thm:upperGammaDRLP1} establishes an upper bound for the independent [$k$]-Roman domination number of $LP_1$-snarks. 

\begin{theorem}
\label{thm:upperGammaDRLP1}
Let $k\geq 1$ be an integer. Let $G_L$ be an $\ell$-$LP_1$-snark with $\sigma$ link-vertices, with $\ell \geq 3$, $\ell$ odd and $\sigma \geq 1$. Then, $i_{[kR]}(G_L) \leq 2(k+1)\ell + k\sigma$.
\end{theorem}

\begin{proof}
For each $i \in \{0,1,\ldots,\ell-1\}$, a function $g_i\colon V(B_i) \to \{0,k+1\}$ for basic block $B_i$ is defined in Figure~\ref{fig:BiLP1Labeled}. Note that $g_i$ has weight $\omega(g_i)=2(k+1)$. Also, note that, under $g_i$, every vertex of $B_i$ with label 0 is adjacent to a vertex of $B_i$ with label $k+1$ and no two vertices with label $k+1$ are adjacent. Hence, $g_i$ is an [$k$]-IRDF of $B_i$.

Define a function $f\colon V(G_L) \to \{0,k,k+1\}$ for $G_L$ as follows. For each vertex $v\in V(G_L)$,
\[ f(v) =
  \begin{cases}
    g_i(v) & \quad \text{if } v \in  V(B_i), \text{ for }  0\leq i \leq \ell-1;\\
    k  & \quad \text{if } v \text{ is a link-vertex}.
  \end{cases}
\]

The weight of $f$ is given by $\omega(f) = \sum_{i=0}^{\ell-1}\omega(g_i) + k\sigma = 2(k+1)\ell+k\sigma$. It remains to show that $f$ is a [$k$]-IRDF. Note that the vertices that receive label $k+1$ are the vertices $p_i$ and $q_i$, for $0 \leq i \leq \ell-1$, and these vertices form an independent set of $G_L$. Moreover, the set of link-vertices also form an independent set of $G_L$. Every link-vertex has label $k$ and is adjacent to three vertices that have label 0. Therefore, no two vertices with labels $k$ or $k+1$ are adjacent and, as previously argued, every vertex of $B_i$ with label 0 is adjacent to a vertex with label $k+1$ that also belongs to $B_i$. Therefore, $f$ is a [$k$]-IRDF of $G_L$ with weight $2(k+1)\ell+k\sigma$.
\end{proof}

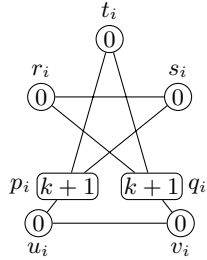
\begin{figure}[!htb]
\centering
\begin{tikzpicture}
   [cp/.style={circle,fill=white,draw,minimum size=1em,inner sep=1pt},
   cpb/.style={fill=white,draw,rounded corners=.09cm,minimum size=1em,inner sep=1pt},
   scale=.7]

   \node[cpb, label={[label distance=-0.05cm]west:\small $p_i$}] (1) at (-.8,.5) {\small $k+1$};
   \node[cpb, label={[label distance=-0.05cm]east:\small $q_i$}] (2) at (.8,.5) {\small $k+1$};
   \node[cp, label={[label distance=-0.05cm]north:\small $r_i$}] (3) at (-1.3,2.2) {\small $0$};
   \node[cp, label={[label distance=-0.05cm]north:\small $s_i$}] (4) at (1.3,2.2) {\small $0$};
   \node[cp,, label={[label distance=-0.05cm]north:\small $t_i$}] (5) at (0,3.3) {\small $0$};

   \draw (1)--(5)--(2)--(3)--(4)--(1);

   \node[cp, label={[label distance=-0.05cm]south:\small $u_i$}] (6) at (-1.35,-.2) {\small $0$};
   \node[cp, label={[label distance=-0.05cm]south:\small $v_i$}] (7) at (1.35,-.2) {\small $0$};
   
   \draw (6)--(7);
   \draw (6)--(1);
   \draw (7)--(2);
   \end{tikzpicture} 
\caption{Basic block $B_i$ with a [$k$]-IRDF $g_i$ with weight $2(k+1)$.}
\label{fig:BiLP1Labeled}
\end{figure}

Theorem~\ref{thm:upperBoundLP0iDR} shows a better upper bound for $i_{[kR]}$ than that shown in Theorem~\ref{thm:upperGammaDRLP1} when restricted to the subfamily of $LP_0$-snarks. Next, we define two subgraphs of an $LP_0$-snark that are used in our proof.

Let $G_L$ be an $\ell$-$LP_0$-snark. Given a repairing edge $t_it_{i+1}$ of $G_L$, a \emph{double gadget} $G_d$ of $G_L$ is a subgraph of $G_L$ induced by the set of vertices $V(B_i) \cup V(B_{i+1})$. On the other hand, given a link-vertex $z_{i,i+1,i+2}$ of $G_L$, a \emph{triple gadget} $G_t$ of $G_L$ is a subgraph of $G_L$ induced by the set of vertices  $V(B_i) \cup V(B_{i+1}) \cup V(B_{i+2}) \cup \{z_{i,i+1,i+2}\}$. Figure~\ref{fig:gadgets} shows a scheme of a double gadget and a triple gadget.

\begin{figure}[!htb]
    \centering
    \begin{subfigure}[b]{.42\textwidth}
        \centering
        \begin{tikzpicture}
           [cp/.style={circle,fill=black,draw,minimum size=.4em,inner sep=1pt]},scale=.61]

           \node[cp, label={[label distance=-0.05cm]west:\small $p_j$}] (1) at (-.8,.5) {};
           \node[cp, label={[label distance=-0.05cm]east:\small $q_j$}] (2) at (.8,.5) {};
           \node[cp, label={[label distance=-0.05cm]north:\small $r_j$}] (3) at (-1.3,2.2) {};
           \node[cp, label={[label distance=-0.05cm]north:\small $s_j$}] (4) at (1.3,2.2) {};
           \node[cp,, label={[label distance=-0.05cm]north:\small $t_j$}] (5) at (0,3.3) {};

           \draw (1)--(5)--(2)--(3)--(4)--(1);

           \node[cp, label={[label distance=-0.05cm]south:\small $u_j$}] (6) at (-1.35,-.2) {};
           \node[cp, label={[label distance=-0.05cm]south:\small $v_j$}] (7) at (1.35,-.2) {};

           \draw (6)--(7);

           \draw (6)--(1);
           \draw (7)--(2);

           \node[cp, label={[label distance=-0.08cm]west:\small $p_{j+1}$}] (8) at (4.2,.5) {};
           \node[cp, label={[label distance=-0.05cm]east:\small $q_{j+1}$}] (9) at (5.8,.5) {};
           \node[cp, label={[label distance=-0.05cm]north:\small $r_{j+1}$}] (10) at (3.7,2.2) {};
           \node[cp, label={[label distance=-0.05cm]north:\small $s_{j+1}$}] (11) at (6.3,2.2) {};
           \node[cp,, label={[label distance=-0.05cm]north:\small $t_{j+1}$}] (12) at (5,3.3) {};

           \draw (8)--(12)--(9)--(10)--(11)--(8);

           \node[cp, label={[label distance=-0.05cm]south:\small $u_{j+1}$}] (13) at (3.65,-.2) {};
           \node[cp, label={[label distance=-0.05cm]south:\small $v_{j+1}$}] (14) at (6.35,-.2) {};
    
           \draw (13)--(14);
           \draw (13)--(8);
           \draw (14)--(9);

           \draw (5) edge [bend left=15] (12);

           \draw[dashed] (7) -- (13);
           \draw[dashed] (7) -- (10);
           \draw[dashed] (4) -- (10);
           \draw[dashed] (4) -- (13);
       \end{tikzpicture}
       \caption{Scheme of a double gadget $G_d$.}
       \label{fig;gadgetDuplo}
    \end{subfigure}
    \begin{subfigure}[b]{.55\textwidth}
        \centering
        \begin{tikzpicture}
           [cp/.style={circle,fill=black,draw,minimum size=.4em,inner sep=1pt]},scale=.6]

           \node[cp, label={[label distance=-0.05cm]west:\small $p_l$}] (1) at (-.8,.5) {};
           \node[cp, label={[label distance=-0.05cm]east:\small $q_l$}] (2) at (.8,.5) {};
           \node[cp, label={[label distance=-0.05cm]north:\small $r_l$}] (3) at (-1.3,2.2) {};
           \node[cp, label={[label distance=-0.05cm]north:\small $s_l$}] (4) at (1.3,2.2) {};
           \node[cp,, label={[label distance=-0.05cm]north:\small $t_l$}] (5) at (0,3.3) {};

           \draw (1)--(5)--(2)--(3)--(4)--(1);

           \node[cp, label={[label distance=-0.05cm]south:\small $u_l$}] (6) at (-1.35,-.2) {};
           \node[cp, label={[label distance=-0.05cm]south:\small $v_l$}] (7) at (1.35,-.2) {};

           \draw (6)--(7);

           \draw (6)--(1);
           \draw (7)--(2);

           \node[cp, label={[label distance=-0.09cm]west:\small $p_{l+1}$}] (8) at (4.2,.5) {};
           \node[cp, label={[label distance=-0.05cm]east:\small $q_{l+1}$}] (9) at (5.8,.5) {};
           \node[cp, label={[label distance=-0.05cm]north:\small $r_{l+1}$}] (10) at (3.7,2.2) {};
           \node[cp, label={[label distance=-0.05cm]north:\small $s_{l+1}$}] (11) at (6.3,2.2) {};
           \node[cp,, label={[label distance=-0.05cm]east:\small $t_{l+1}$}] (12) at (5,3.3) {};

           \draw (8)--(12)--(9)--(10)--(11)--(8);

           \node[cp, label={[label distance=-0.05cm]south:\small $u_{l+1}$}] (13) at (3.65,-.2) {};
           \node[cp, label={[label distance=-0.05cm]south:\small $v_{l+1}$}] (14) at (6.35,-.2) {};

           \node[cp, label={[label distance=-0.05cm]west:\small $p_{l+2}$}] (15) at (9.2,.5) {};
           \node[cp, label={[label distance=-0.05cm]east:\small $q_{l+2}$}] (16) at (10.8,.5) {};
           \node[cp, label={[label distance=-0.05cm]north:\small $r_{l+2}$}] (17) at (8.7,2.2) {};
           \node[cp, label={[label distance=-0.05cm]north:\small $s_{l+2}$}] (18) at (11.3,2.2) {};
           \node[cp,, label={[label distance=-0.05cm]north:\small $t_{l+2}$}] (19) at (10,3.3) {};

           \draw (15)--(19)--(16)--(17)--(18)--(15);

           \node[cp, label={[label distance=-0.05cm]south:\small $u_{l+2}$}] (20) at (8.65,-.2) {};
           \node[cp, label={[label distance=-0.05cm]south:\small $v_{l+2}$}] (21) at (11.35,-.2) {};
             \node[cp, label={[label distance=-0.05cm]north:\small $z_{l,l+1,l+2}$}] (zl) at (5,4.5) {};

           \draw (20)--(21);

           \draw (20)--(15);
           \draw (21)--(16);

           \draw (13)--(14);

           \draw (13)--(8);
           \draw (14)--(9);

           \draw (5)--(zl)--(12);
           \draw (zl)--(19);

           \draw[dashed] (7) -- (13);
           \draw[dashed] (7) -- (10);
           \draw[dashed] (4) -- (10);
           \draw[dashed] (4) -- (13);

           \draw[dashed] (14) -- (20);
           \draw[dashed] (14) -- (17);
           \draw[dashed] (11) -- (17);
           \draw[dashed] (11) -- (20);
       \end{tikzpicture}
       \caption{Scheme of a triple gadget $G_t$.}
       \label{fig;gadgetTriplo}
    \end{subfigure}
    \caption{Gadgets of an $LP_0$-snark. Dashed edges represent the possible configurations for plug-edges connecting two consecutive blocks: either laminar edges or intersecting edges.}
    \label{fig:gadgets}
\end{figure}
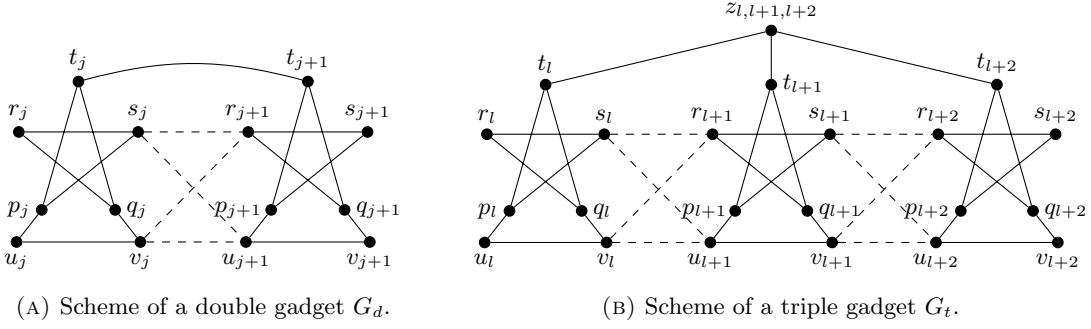

\begin{theorem}
\label{thm:upperBoundLP0iDR}
Let $k\geq 1$ be an integer. Let $G_L$ be an $\ell$-$LP_0$-snark with $\sigma$ link-vertices, with $\sigma \geq 1$ and odd $\ell \geq 3$. Then, $i_{[kR]}(G_L) \leq 2(k+1)\ell + (k-1)\sigma$.
\end{theorem}

\begin{proof}
Let $G_L$ be an $\ell$-$LP_0$-snark. In order to prove the theorem, we construct a [$k$]-IRDF $f\colon V(G_L) \to \{0,k,k+1\}$ of $G_L$ with weight $2(k+1)\ell+(k-1)\sigma$. Let $G_1,G_2,\ldots,G_r$ be gadgets of $G_L$ given in cyclic sequence, with $r\geq 1$. Figure~\ref{fig:gadgetDuploRotulado} exhibits an [$k$]-IRDF $\pi_d \colon V(G_d) \to \{0,k+1\}$ with weight $4(k+1)$ for double gadgets and Figure~\ref{fig:gadgetTriploRotulado} exhibits an [$k$]-IRDF $\pi_t \colon V(G_t) \to \{0,k,k+1\}$ with weight $5(k+1)+2k$ for triple gadgets. Note that, in both labelings $\pi_d$ and $\pi_t$, the set of vertices with label $k$ or $k+1$ form an independent set and every vertex with label 0 is adjacent to a vertex with label $k+1$. Hence, $\pi_d$ and $\pi_t$ are [$k$]-IRDFs.

For every vertex $v \in V(G_L)$, define: $f(v)=\pi_d(v)$ if $v$ belongs to a double gadget $G_d \subset G_L$; or $f(v)=\pi_t(v)$, otherwise. Next, we prove that the labeling $f$ is a [$k$]-IRDF of $G_L$. By the definition of $f$, we have that $f$ restricted to double gadgets or triple gadgets is a [$k$]-IRDF. Thus, it remains to show that $V_k \cup V_{k+1}$ is an independent set of $G_L$ (Recall that $V_i \subseteq V(G_L)$ is the set of vertices with label $i$). In order to see this, recall that the set of vertices $V_k\cup V_{k+1}$ form an independent set when restricted to double or triple gadgets.  This property still holds when considering the whole graph since the plug-edges that connect two consecutive gadgets $G_i$ and $G_{i+1}$ connect border vertices that have label 0. Therefore, $f$ is an [$k$]-IRDF of $G_L$. 

The weight of $f$ is the sum of the weight of its gadgets. Note that $\ell = 2p_2+3p_3$, where $p_2$ is the number of double gadgets and $p_3 = \sigma$ is the number of triple gadgets.  Thus, $\omega(f) = \omega(\pi_d)\cdot p_2+\omega(\pi_t)\cdot p_3 = 4(k+1)p_2+(5(k+1)+2k)p_3 = 2(k+1)(2p_2+3p_3)+(k-1)p_3 = 2(k+1)\ell+(k-1)p_3 = 2(k+1)\ell+(k-1)\sigma$.
\end{proof}

\begin{figure}[!htb]
    \centering
    \begin{subfigure}[b]{.41\textwidth}
        \centering
        \begin{tikzpicture}
           [cp/.style={circle,fill=white,draw,minimum size=1em,inner sep=1pt]},
           cpb/.style={fill=white,draw,rounded corners=.09cm,minimum size=1em,inner sep=1pt]},
           scale=.65]

           \node[cpb, label={[label distance=-0.05cm]west:\small $p_i$}] (1) at (-.8,.5) {\small $k+1$};
           \node[cpb, label={[label distance=-0.05cm]east:\small $q_i$}] (2) at (.8,.5) {\small $k+1$};
           \node[cp, label={[label distance=-0.05cm]north:\small $r_i$}] (3) at (-1.3,2.2) {\small $0$};
           \node[cp, label={[label distance=-0.05cm]north:\small $s_i$}] (4) at (1.3,2.2) {\small $0$};
           \node[cp, label={[label distance=-0.05cm]north:\small $t_i$}] (5) at (0,3.3) {\small $0$};

           \draw (1)--(5)--(2)--(3)--(4)--(1);

           \node[cp, label={[label distance=-0.05cm]south:\small $u_i$}] (6) at (-1.35,-.2) {\small $0$};
           \node[cp, label={[label distance=-0.05cm]south:\small $v_i$}] (7) at (1.35,-.2) {\small $0$};

           \draw (6)--(7);

           \draw (6)--(1);
           \draw (7)--(2);

           \node[cpb, label={above left:\small $p_{i+1}$}] (8) at (4.2,.5) {\small $k+1$};
           \node[cpb, label={[label distance=-0.05cm]east:\small $q_{i+1}$}] (9) at (5.8,.5) {\small $k+1$};
           \node[cp, label={[label distance=-0.05cm]north:\small $r_{i+1}$}] (10) at (3.7,2.2) {\small $0$};
           \node[cp, label={[label distance=-0.05cm]north:\small $s_{i+1}$}] (11) at (6.3,2.2) {\small $0$};
           \node[cp, label={[label distance=-0.05cm]north:\small $t_{i+1}$}] (12) at (5,3.3) {\small $0$};

           \draw (8)--(12)--(9)--(10)--(11)--(8);

           \node[cp, label={[label distance=-0.05cm]south:\small $u_{i+1}$}] (13) at (3.65,-.2) {\small $0$};
           \node[cp, label={[label distance=-0.05cm]south:\small $v_{i+1}$}] (14) at (6.35,-.2) {\small $0$};

           \draw (13)--(14);

           \draw (13)--(8);
           \draw (14)--(9);

           \draw (5) edge [bend left=15] (12);

           \draw[dashed] (7) -- (13);
           \draw[dashed] (7) -- (10);
           \draw[dashed] (4) -- (10);
           \draw[dashed] (4) -- (13);
       \end{tikzpicture}
       \caption{[$k$]-IRDF $\pi_d$ of $G_d$ with weight $4(k+1)$.}
       \label{fig:gadgetDuploRotulado}
    \end{subfigure}
    \begin{subfigure}[b]{.55\textwidth}
        \centering       
        \begin{tikzpicture}
           [cp/.style={circle,fill=white,draw,minimum size=1em,inner sep=1pt]},
           cpb/.style={fill=white,draw,rounded corners=.09cm,minimum size=1em,inner sep=1pt]},
           scale=.65]

           \node[cp, label={[label distance=-0.05cm]west:\small $p_i$}] (1) at (-.8,.5) {\small $0$};
           \node[cp, label={[label distance=-0.05cm]east:\small $q_i$}] (2) at (.8,.5) {\small $0$};
           \node[cp, label={[label distance=-0.05cm]north:\small $r_i$}] (3) at (-1.3,2.2) {\small $0$};
           \node[cpb, label={[label distance=-0.05cm]north:\small $s_i$}] (4) at (1.3,2.2) {\small $k+1$};
           \node[cp, label={[label distance=-0.05cm]north:\small $t_i$}] (5) at (0,3.3) {\small $0$};

           \draw (1)--(5)--(2)--(3)--(4)--(1);

           \node[cp, label={[label distance=-0.05cm]south:\small $u_i$}] (6) at (-1.35,-.2) {\small $0$};
           \node[cpb, label={[label distance=-0.05cm]south:\small $v_i$}] (7) at (1.35,-.2) {\small $k+1$};

           \draw (6)--(7);

           \draw (6)--(1);
           \draw (7)--(2);

           \node[cpb, label={[label distance=-0.09cm]west:\small $p_{i+1}$}] (8) at (4.2,.5) {\small $k$};
           \node[cpb, label={[label distance=-0.05cm]east:\small $q_{i+1}$}] (9) at (5.8,.5) {\small $k$};
           \node[cp, label={[label distance=-0.05cm]north:\small $r_{i+1}$}] (10) at (3.7,2.2) {\small $0$};
           \node[cp, label={[label distance=-0.05cm]north:\small $s_{i+1}$}] (11) at (6.3,2.2) {\small $0$};
           \node[cp,, label={[label distance=-0.05cm]east:\small $t_{i+1}$}] (12) at (5,3.3) {\small $0$};

           \draw (8)--(12)--(9)--(10)--(11)--(8);

           \node[cp, label={[label distance=-0.05cm]south:\small $u_{i+1}$}] (13) at (3.65,-.2) {\small $0$};
           \node[cp, label={[label distance=-0.05cm]south:\small $v_{i+1}$}] (14) at (6.35,-.2) {\small $0$};

           \node[cp, label={[label distance=-0.05cm]west:\small $p_{i+2}$}] (15) at (9.2,.5) {\small $0$};
           \node[cp, label={[label distance=-0.05cm]east:\small $q_{i+2}$}] (16) at (10.8,.5) {\small $0$};
           \node[cpb, label={[label distance=-0.05cm]north:\small $r_{i+2}$}] (17) at (8.7,2.2) {\small $k+1$};
           \node[cp, label={[label distance=-0.05cm]north:\small $s_{i+2}$}] (18) at (11.3,2.2) {\small $0$};
           \node[cp,, label={[label distance=-0.05cm]east:\small $t_{i+2}$}] (19) at (10,3.3) {\small $0$};

           \draw (15)--(19)--(16)--(17)--(18)--(15);

           \node[cpb, label={[label distance=-0.05cm]south:\small $u_{i+2}$}] (20) at (8.65,-.2) {\small $k+1$};
           \node[cp, label={[label distance=-0.05cm]south:\small $v_{i+2}$}] (21) at (11.35,-.2) {\small $0$};
             \node[cpb, label={[label distance=-0.05cm]north:\small $z_{i,i+1,i+2}$}] (zl) at (5,4.5) {\small $k+1$};

           \draw (20)--(21);

           \draw (20)--(15);
           \draw (21)--(16);

           \draw (13)--(14);

           \draw (13)--(8);
           \draw (14)--(9);

           \draw (5)--(zl)--(12);
           \draw (zl)--(19);

           \draw[dashed] (7) -- (13);
           \draw[dashed] (7) -- (10);
           \draw[dashed] (4) -- (10);
           \draw[dashed] (4) -- (13);

           \draw[dashed] (14) -- (20);
           \draw[dashed] (14) -- (17);
           \draw[dashed] (11) -- (17);
           \draw[dashed] (11) -- (20);
       \end{tikzpicture}
       \caption{[$k$]-IRDF $\pi_t$ of $G_t$ with weight $5(k+1)+2k$.}
       \label{fig:gadgetTriploRotulado}
    \end{subfigure}
    \caption{Double gadget $G_d$ with a [$k$]-IRDF $\pi_d \colon V(G_d) \to \{0,k+1\}$ and triple gadget $G_t$ with a [$k$]-IRDF $\pi_t \colon V(G_t) \to \{0,k,k+1\}$.}
    \label{fig:gadgetsRotulado}
\end{figure}
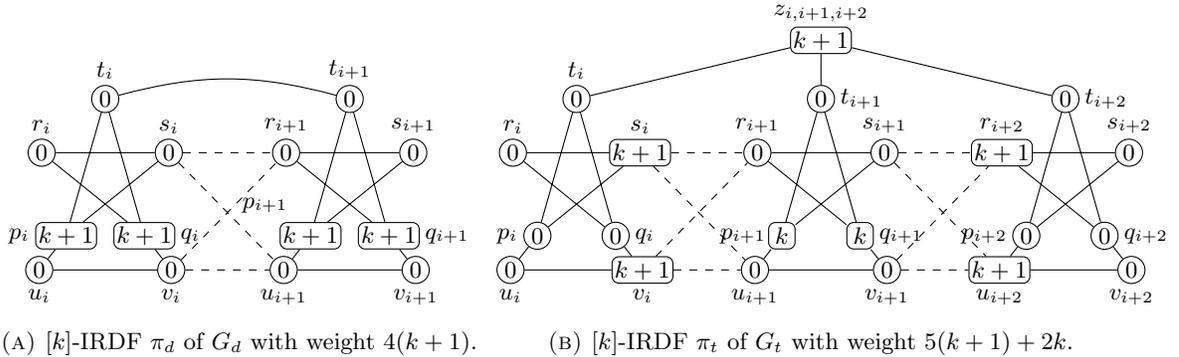

Theorem~\ref{thm:upperBoundLP0gammaDR} establishes an upper bound for the [$k$]-Roman domination number of $LP_0$-snarks.

\begin{theorem}
\label{thm:upperBoundLP0gammaDR}
Let $k \geq 1$ be an integer. Let $G_L$ be an $\ell$-$LP_0$-snark with $\ell\geq 3$, $\ell$ odd. Then, $\gamma_{[kR]}(G_L) \leq 2(k+1)\ell$.
\end{theorem}

\begin{proof}
Let $G_L$ be an $\ell$-$LP_0$-snark. We construct a [$k$]-RDF $f\colon V(G_L) \to \{0,k,k+1\}$ for $G_L$ with weight $2(k+1)\ell$. Figure~\ref{fig:gadgetTriploRotuladoV2} exhibits a [$k$]-RDF $\pi_t \colon V(G_t) \to \{0,k+1\}$ with weight $6(k+1)$ for triple gadgets $G_t \subset G_L$; and double gadgets $G_d \subset G_L$ are assigned the [$k$]-IRDF $\pi_d \colon V(G_d) \to \{0,k+1\}$ with weight $4(k+1)$ shown in  Figure~\ref{fig:gadgetDuploRotulado}. Note that, in both labelings $\pi_d$ and $\pi_t$, every vertex with label 0 is adjacent to a vertex with label $k+1$. Hence, $\pi_d$ e $\pi_t$ are [$k$]-RDFs.
For every vertex $v \in V(G_L)$, define: $f(v)=\pi_d(v)$ if $v$ belongs to a double gadget $G_d \subset G_L$; or $f(v)=\pi_t(v)$, otherwise. By the definition of $f$, we have that $f$ restricted to each gadget is a [$k$]-RDF (in fact, by the definition of $\pi_d$ and $\pi_d$, any vertex with label 0 is already adjacent to a vertex with label $k+1$). Hence, $f$ is a [$k$]-RDF of $G_L$. Furthermore, the weight of $f$ is the sum of the weight of its gadgets. Note that $\ell = 2p_2+3p_3$, where $p_2$ is the number of double gadgets and $p_3$ is the number of triple gadgets.  Thus, $\omega(f) = \omega(\pi_d)\cdot p_2+\omega(\pi_t)\cdot p_3 = 4(k+1)p_2+6(k+1)p_3 = 2(k+1)(2p_2+3p_3) = 2(k+1)\ell$.
\end{proof}

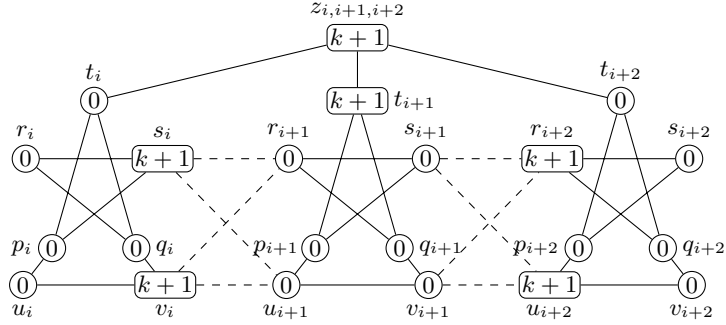
\begin{figure}[!htb]
    \centering        
        \begin{tikzpicture}
           [cp/.style={circle,fill=white,draw,minimum size=1em,inner sep=1pt]},
           cpb/.style={fill=white,draw,rounded corners=.09cm,minimum size=1em,inner sep=1pt]},
           scale=.7]

           \node[cp, label={[label distance=-0.05cm]west:\small $p_i$}] (1) at (-.8,.5) {\small $0$};
           \node[cp, label={[label distance=-0.05cm]east:\small $q_i$}] (2) at (.8,.5) {\small $0$};
           \node[cp, label={[label distance=-0.05cm]north:\small $r_i$}] (3) at (-1.3,2.2) {\small $0$};
           \node[cpb, label={[label distance=-0.05cm]north:\small $s_i$}] (4) at (1.3,2.2) {\small $k+1$};
           \node[cp, label={[label distance=-0.05cm]north:\small $t_i$}] (5) at (0,3.3) {\small $0$};

           \draw (1)--(5)--(2)--(3)--(4)--(1);

           \node[cp, label={[label distance=-0.05cm]south:\small $u_i$}] (6) at (-1.35,-.2) {\small $0$};
           \node[cpb, label={[label distance=-0.05cm]south:\small $v_i$}] (7) at (1.35,-.2) {\small $k+1$};

           \draw (6)--(7);

           \draw (6)--(1);
           \draw (7)--(2);

           \node[cp, label={[label distance=-0.09cm]west:\small $p_{i+1}$}] (8) at (4.2,.5) {\small $0$};
           \node[cp, label={[label distance=-0.05cm]east:\small $q_{i+1}$}] (9) at (5.8,.5) {\small $0$};
           \node[cp, label={[label distance=-0.05cm]north:\small $r_{i+1}$}] (10) at (3.7,2.2) {\small $0$};
           \node[cp, label={[label distance=-0.05cm]north:\small $s_{i+1}$}] (11) at (6.3,2.2) {\small $0$};
           \node[cpb, label={[label distance=-0.05cm]east:\small $t_{i+1}$}] (12) at (5,3.3) {\small $k+1$};

           \draw (8)--(12)--(9)--(10)--(11)--(8);

           \node[cp, label={[label distance=-0.05cm]south:\small $u_{i+1}$}] (13) at (3.65,-.2) {\small $0$};
           \node[cp, label={[label distance=-0.05cm]south:\small $v_{i+1}$}] (14) at (6.35,-.2) {\small $0$};

           \node[cp, label={[label distance=-0.05cm]west:\small $p_{i+2}$}] (15) at (9.2,.5) {\small $0$};
           \node[cp, label={[label distance=-0.05cm]east:\small $q_{i+2}$}] (16) at (10.8,.5) {\small $0$};
           \node[cpb, label={[label distance=-0.05cm]north:\small $r_{i+2}$}] (17) at (8.7,2.2) {\small $k+1$};
           \node[cp, label={[label distance=-0.05cm]north:\small $s_{i+2}$}] (18) at (11.3,2.2) {\small $0$};
           \node[cp,, label={[label distance=-0.05cm]north:\small $t_{i+2}$}] (19) at (10,3.3) {\small $0$};

           \draw (15)--(19)--(16)--(17)--(18)--(15);

           \node[cpb, label={[label distance=-0.05cm]south:\small $u_{i+2}$}] (20) at (8.65,-.2) {\small $k+1$};
           \node[cp, label={[label distance=-0.05cm]south:\small $v_{i+2}$}] (21) at (11.35,-.2) {\small $0$};
             \node[cpb, label={[label distance=-0.05cm]north:\small $z_{i,i+1,i+2}$}] (zl) at (5,4.5) {\small $k+1$};

           \draw (20)--(21);

           \draw (20)--(15);
           \draw (21)--(16);

           \draw (13)--(14);

           \draw (13)--(8);
           \draw (14)--(9);

           \draw (5)--(zl)--(12);
           \draw (zl)--(19);

           \draw[dashed] (7) -- (13);
           \draw[dashed] (7) -- (10);
           \draw[dashed] (4) -- (10);
           \draw[dashed] (4) -- (13);

           \draw[dashed] (14) -- (20);
           \draw[dashed] (14) -- (17);
           \draw[dashed] (11) -- (17);
           \draw[dashed] (11) -- (20);
       \end{tikzpicture}
    \caption{Triple gadget $G_t$ with a [$k$]-RDF $\pi_t \colon V(G_t) \to \{0,k+1\}$ with weight $6(k+1)$.}
    \label{fig:gadgetTriploRotuladoV2}
\end{figure}

By Theorem~\ref{thm:lowerBoundDelta3}, a $k$-$LP_0$-snark $G$ with $n$ vertices has $i_{[kR]}(G) \geq \left\lceil \frac{(k+1)n}{4} \right\rceil$. However, this lower bound is improved, for all $k\geq 4$, in Theorem~\ref{thm:lowerBoundLP0}. In order to prove it, we need the following lower bound on the domination number of $LP_0$-snarks~\cite{Pereira2020}.

\begin{theorem}[A.~Pereira~\cite{Pereira2020}]
\label{thm:PereiraLowerBoundGammaLP0}
Let $G_L$ be an $\ell$-$LP_0$-snark with $\ell\geq 3$ and $\sigma$ link-vertices. Then, $\gamma(G_L) \geq \left\lfloor \frac{n}{4} \right\rfloor+1$, where $n=7\ell+\sigma$.\qed
\end{theorem}

\begin{theorem}
\label{thm:lowerBoundLP0}
Let $k \geq 4$ be an integer. Let $G_L$ be an $\ell$-$LP_0$-snark with $\ell\geq 3$ and $\sigma$ link-vertices. Then, $i_{[kR]}(G_L) \geq \left\lceil \frac{(k+1)n}{4} \right\rceil + 1$, where $n = |V(G_L)| = 7\ell+\sigma$.
\end{theorem}

\begin{proof}
Let $G_L$ and $k$ be as stated in the hypothesis. For the purpose of contradiction, suppose that $i_{[kR]}(G_L) \leq \left\lceil \frac{(k+1)n}{4} \right\rceil$.  By Theorem~\ref{thm:lowerBoundDelta3}, we have that $i_{[kR]}(G_L) \geq \left\lceil \frac{(k+1)n}{4} \right\rceil$. From these two inequalities we obtain that $i_{[kR]}(G_L) = \left\lceil\frac{(k+1)n}{4}\right\rceil$. This fact along with Theorem~\ref{thm:lowerBoundKgDelta} and the fact that $k\geq 4$ imply that $G_L$ is independent [$k$]-Roman, that is, $i_{[kR]}(G_L) = (k+1)i(G_L)$. By Theorem~\ref{thm:PereiraLowerBoundGammaLP0}, we obtain that $i_{[kR]}(G_L)=(k+1)i(G_L)\geq (k+1)\gamma(G_L)\geq (k+1)\left\lfloor \frac{n}{4} \right\rfloor+(k+1)$. Thus, we have that $\left\lceil\frac{(k+1)n}{4}\right\rceil = i_{[kR]}(G_L)=(k+1)i(G_L)\geq (k+1)\left\lfloor \frac{n}{4} \right\rfloor+(k+1)$, that is,
\begin{equation}\label{eq:0023x}
\left\lceil\frac{(k+1)n}{4}\right\rceil  \geq (k+1)\left\lfloor \frac{n}{4} \right\rfloor+(k+1).
\end{equation}
Since $n = 7\ell + \sigma$ and $\ell$ and $\sigma$ are odd numbers, we obtain that $n$ is an even number. We split the proof into two cases depending of the value of $n$ modulo 4.

\smallskip 

\noindent \textbf{Case 1.} $n \equiv 0\pmod{4}$. Define $n = 4p$, $p \in \mathbb{Z}$. From inequality~\eqref{eq:0023x} we have that $\left\lceil\frac{(k+1)4p}{4}\right\rceil  \geq (k+1)\left\lfloor \frac{4p}{4} \right\rfloor+(k+1)$, which implies that $(k+1)p \geq (k+1)p+(k+1)$, which is a contradiction.

\smallskip 

\noindent \textbf{Case 2.} $n \equiv 2\pmod{4}$. Define $n = 4p+2$, $p \in \mathbb{Z}$. From inequality~\eqref{eq:0023x} we have that $\left\lceil\frac{(k+1)(4p+2)}{4}\right\rceil  \geq (k+1)\left\lfloor \frac{4p+2}{4} \right\rfloor+(k+1)$, which implies that $(k+1)p+\frac{k+1}{2} \geq (k+1)p+(k+1)$, which is a contradiction.

Since both cases led us to a contradiction, we conclude that $i_{[kR]}(G_L) \geq \left\lceil \frac{(k+1)n}{4} \right\rceil + 1$.
\end{proof}

%%%%%%%%%%%%%%%%%
% Section
%%%%%%%%%%%%%%%%%
\section{Closing Remarks}
\label{sec:conclusion}

In this work, we prove that, for all $k \geq 3$, the independent [$k$]-Roman domination problem and the [$k$]-Roman domination problem are $\np$-complete even when restricted to planar bipartite graphs with maximum degree 3 and also present lower and upper bounds for the parameter $i_{[kR]}(G)$. Moreover, we investigate  $i_{[kR]}(G)$ for two families of 3-regular graphs called generalized Blanu\v{s}a snarks and Loupekine snarks. 

In Corollary~\ref{cor:indepRomanBlanusa}, we present an infinite family of independent [$k$]-Roman graphs, which are graphs that have $i_{[kR]}(G)=(k+1)i(G)$. An interesting open problem is finding other classes of independent [$k$]-Roman graphs.

Adabi et al.~\cite{Adabi2012} proved that any graph $G$ with $\Delta(G) \leq 3$, has $\gamma_{[kR]}(G) = i_{[kR]}(G)$ for $k=1$.
We remark that the family of planar bipartite graphs with maximum degree 3 constructed in the reduction shown in Section~\ref{sec:complexity} is an example of infinite family of graphs with $\Delta(G)=3$ for which $\gamma_{[kR]}(G) = i_{[kR]}(G)$ for all $k\geq 1$. Thus, another interesting line of research is finding other classes of graphs with $\Delta(G)\leq 3$ for which $\gamma_{[kR]}(G) = i_{[kR]}(G)$ for $k \geq 2$. In fact, we conjecture that this property holds for all generalized Blanu\v{s}a snarks.

% ----------------------------
% SECTION
% ----------------------------
\section{Acknowledgments}

This work was partially supported by Conselho Nacional de Desenvolvimento Científico e Tecnológico – CNPq.


\begin{thebibliography}{10}
\expandafter\ifx\csname url\endcsname\relax
  \def\url#1{\texttt{#1}}\fi
\expandafter\ifx\csname urlprefix\endcsname\relax\def\urlprefix{URL }\fi
\expandafter\ifx\csname href\endcsname\relax
  \def\href#1#2{#2} \def\path#1{#1}\fi

\bibitem{haynes2023domination}
T.~Haynes, S.~Hedetniemi, M.~Henning, Domination in Graphs: Core Concepts,
  Springer Monographs in Mathematics, Springer International Publishing, 2023.

\bibitem{Cockayne2004Roman}
E.~J. Cockayne, P.~A. Dreyer~Jr, S.~M. Hedetniemi, S.~T. Hedetniemi, Roman
  domination in graphs, Discrete mathematics 278~(1-3) (2004) 11--22.

\bibitem{Stewart1999}
I.~Stewart, Defenden the {R}oman empire!, Scientific American 281~(6) (1999)
  136--139.

\bibitem{ReVelle2000}
C.~S. ReVelle, K.~E. Rosing, {Defendens Imperium Romanum: A Classical Problem
  in Military Strategy}, The American Mathematical Monthly 107~(7) (2000)
  585--594.

\bibitem{Chellali2020}
M.~Chellali, N.~J. Rad, S.~Sheikholeslami, L.~Volkmann, A survey on {R}oman
  domination parameters in directed graphs, J. Combin. Math. Combin.
  Comput.~(115) (2020) 141–171.

\bibitem{Chellali2020b}
M.~Chellali, N.~J. Rad, S.~M. Sheikholeslami, L.~Volkmann, {Varieties of Roman
  domination II}, AKCE International Journal of Graphs and Combinatorics 17~(3)
  (2020) 966--984.

\bibitem{Chellali2020c}
M.~Chellali, N.~Jafari~Rad, S.~M. Sheikholeslami, L.~Volkmann, Roman Domination
  in Graphs, Springer International Publishing, Cham, 2020, pp. 365--409.

\bibitem{Chellali2021}
M.~Chellali, N.~J. Rad, S.~M. Sheikholeslami, L.~Volkmann, {Varieties of
  {R}oman Domination}, Springer International Publishing, Cham, Switzerland,
  2021, pp. 273--307.

\bibitem{Chellali2022}
M.~Chellali, N.~J. Rad, S.~M. Sheikholeslami, L.~Volkmann, {The Roman Domatic
  Problem in Graphs and Digraphs: A Survey}, Discussiones Mathematicae Graph
  Theory 42~(3) (2022) 861--891.

\bibitem{ABDOLLAHZADEHAHANGAR2021125444}
H.~{Abdollahzadeh Ahangar}, M.~Álvarez, M.~Chellali, S.~Sheikholeslami,
  J.~Valenzuela-Tripodoro, Triple roman domination in graphs, Applied
  Mathematics and Computation 391 (2021) 125444.
\newblock \href {https://doi.org/https://doi.org/10.1016/j.amc.2020.125444}
  {\path{doi:https://doi.org/10.1016/j.amc.2020.125444}}.

\bibitem{Khalili2023}
M.~C. N.~Khalili, J.~Amjadi, S.~M. Sheikholeslami, On [k]-roman domination in
  graphs, AKCE International Journal of Graphs and Combinatorics 20~(3) (2023)
  291--299.
\newblock \href {https://doi.org/10.1080/09728600.2023.2241531}
  {\path{doi:10.1080/09728600.2023.2241531}}.

\bibitem{Valenzuela-Tripodoro2024}
M.~C.~L. J.~C. Valenzuela-Tripodoro, M. A. Mateos-Camacho, M.~P. Alvarez-Ruiz,
  Further results on the [k]-roman domination in graphs, Bulletin of the
  Iranian Mathematical Society 50~(2) (2024) 1--16.
\newblock \href {https://doi.org/10.1007/s41980-024-00872-1}
  {\path{doi:10.1007/s41980-024-00872-1}}.

\bibitem{BEELER201623}
R.~A. Beeler, T.~W. Haynes, S.~T. Hedetniemi, {Double Roman domination},
  Discrete Applied Mathematics 211 (2016) 23--29.
\newblock \href {https://doi.org/https://doi.org/10.1016/j.dam.2016.03.017}
  {\path{doi:https://doi.org/10.1016/j.dam.2016.03.017}}.

\bibitem{Amjadi2022}
J.~Amjadi, N.~Khalili, Quadruple roman domination in graphs, Discrete
  Mathematics, Algorithms and Applications 14~(03) (2022) 2150130.
\newblock \href {https://doi.org/10.1142/S1793830921501305}
  {\path{doi:10.1142/S1793830921501305}}.

\bibitem{Maimani2020a}
H.~Maimani, M.~Momeni, S.~N. Moghaddam, F.~R. Mahid, S.~M. Sheikholeslami,
  Independent double {R}oman domination in graphs, Bulletin of the Iranian
  Mathematical Society 46~(2) (2019) 543--555.
\newblock \href {https://doi.org/10.1007/s41980-019-00274-8}
  {\path{doi:10.1007/s41980-019-00274-8}}.

\bibitem{Luiz2024}
A.~G. Luiz, Roman domination and independent roman domination on graphs with
  maximum degree three, Discrete Applied Mathematics 348 (2024) 260--278.
\newblock \href {https://doi.org/https://doi.org/10.1016/j.dam.2024.02.006}
  {\path{doi:https://doi.org/10.1016/j.dam.2024.02.006}}.

\bibitem{MOHAR2001102}
B.~Mohar, Face covers and the genus problem for apex graphs, Journal of
  Combinatorial Theory, Series B 82~(1) (2001) 102--117.
\newblock \href {https://doi.org/https://doi.org/10.1006/jctb.2000.2026}
  {\path{doi:https://doi.org/10.1006/jctb.2000.2026}}.

\bibitem{8515207}
Z.~Shao, P.~Wu, H.~Jiang, Z.~Li, J.~Žerovnik, X.~Zhang, Discharging approach
  for double {R}oman domination in graphs, IEEE Access 6 (2018) 63345--63351.
\newblock \href {https://doi.org/10.1109/ACCESS.2018.2876460}
  {\path{doi:10.1109/ACCESS.2018.2876460}}.

\bibitem{Tait1880}
P.~G. Tait, Remarks on the colouring of maps, Proceedings of the Royal Society
  of Edinburgh 10~(4) (1880) 501--503.

\bibitem{Petersen1898}
J.~Petersen, Sur le th\'{e}orème de tait, L'Intermédiaire des Mathématiciens
  5 (1898) 225–227.

\bibitem{blanusa1946problem}
D.~Blanu\v{s}a, Problem cetiriju boja (croatian), hrvatsko priordoslorno
  drus\v{s}tvo glasnik mat-fiz, Astr Ser II 1 (1946) 31--42.

\bibitem{watkins1983construction}
J.~J. Watkins, On the construction of snarks, Ars Combinatoria 16 (1983)
  111--124.

\bibitem{Pereira2020}
A.~A. Pereira, Dominating sets in cubic graphs, Master thesis (in portuguese),
  University of Campinas, Institute of Computing, Campinas, São Paulo, Brazil,
  {A}vailable at: https://hdl.handle.net/20.500.12733/1640296 (2020).

\bibitem{Isaacs1976}
R.~Isaacs, Loupekhine’s snarks: a bifamily of non-{T}ait-colorable graphs,
  Journal of Combinatorial Theory, Series B (1976).

\bibitem{Adabi2012}
N.~J.~R. M.~Adabi, E. Ebrahimi~Targhi, M.~S. Moradi, Properties of independent
  {R}oman domination in graphs, Australasian Journal of Combinatorics 52 (2012)
  11--18.

\end{thebibliography}
\end{document}